\documentclass[11pt]{article}
\usepackage{amssymb,amsmath,accents}
\usepackage{amscd}
\usepackage{amsfonts,amsthm,mathrsfs}
\usepackage{setspace}
\usepackage{graphics}
\usepackage[dvips]{graphicx}
\usepackage{latexsym}

\usepackage{url}

\newtheorem{theorem}{Theorem}
\newtheorem{lemma}[theorem]{Lemma}
\newtheorem{proposition}[theorem]{Proposition}
\newtheorem{corollary}[theorem]{Corollary}

\theoremstyle{definition}
\newtheorem{definition}[theorem]{Definition}
\newtheorem{remark}[theorem]{Remark}

\providecommand{\keywords}
{
  \small
  \textbf{\textit{Keywords:}}
}

\numberwithin{equation}{section}

\begin{document}

\title{\textbf{Graph gradient flows : from discrete to continuum}}

\author{Yoshikazu~Giga, Yves~van~Gennip, and Jun~Okamoto}
\date{}
\maketitle

\begin{abstract}
This paper gives a framework to study a continuum limit of a gradient flow on a graph where the number of vertices increases in an appropriate way. As examples we prove the convergence of a discrete total variation flow and a discrete Allen--Cahn flow on discretised tori to their respective continuum limits.
\end{abstract}
%
%

\tableofcontents

\section{Introduction} \label{S1}
We prove the convergence of two graph-based discrete gradient flows, the total variation flow and the Allen--Cahn flow, to a continuum limiting problem when the graph approximates a Euclidean space. Although characterization of continuum limits has been studied variationally at the energy level quite extensively in the past decade, mainly $\Gamma$-convergence results, the convergence problem for discrete flows has only been attracting more attention fairly recently.

If $\Phi: H \to \mathbb{R}$ is a Fr\'echet differentiable function on a Hilbert space $(H, \langle \cdot, \cdot \rangle)$, with gradient $\nabla_H \Phi$, then the gradient flow of $\Phi$ (w.r.t. the inner product on $H$) is a function $u: (0, \infty) \to H$ that satisfies $\dot u = -\nabla_H \Phi(u)$. If $\Phi$ is not differentiable, but convex, the differential {\it equation} is replaced by the differential {\it inclusion} $\dot u \in -\partial \Phi(u)$, where $\partial \Phi(u)$ denotes the subdifferential of $\Phi$ at $u$. The function $\Phi$ is often called the {\it energy}, even when there is no direct physical interpretation as such. An equivalent formulation can be given in terms of an {\it evolution variational inequality} (see Definition~\ref{EVI}; recent in-depth overviews are given in \cite{MuratoriSavare20,ABS21}). Such a formulation offers more flexibility when generalising the notion of gradient flow to metric spaces that lack a Hilbert space structure. Additionally, the variational setting can be exploited when comparing gradient flows that are formulated on different underlying spaces. We make crucial use of this advantage to prove convergence of gradient flows that are formulated on discrete graphs to a continuum gradient flows.

In the current work, we consider two families of gradient flows: total variation gradient flows with the energy $\Phi$ being the graph-based or continuum total variation functional (Section~\ref{sec:thefunctionalsTV}) and Allen--Cahn gradient flows with the energy $\Phi$ being the graph-based or continuum Allen--Cahn functional (also known as the Ginzburg--Landau functional or Modica--Mortola functional; Section~\ref{sec:thefunctionalsAC}).

To be able to prove convergence results, and thus to compare flows formulated on a graph with flows formulated on a continuum, we require a way to embed the discrete flows in continuum space. In this paper we restrict ourselves to graphs that are obtained as a regular (square or (hyper)cubic) discretisation of the flat torus. In that setting, we can associate functions on the graph with functions on the torus via piecewise constant or piecewise linear interpolation. We use the former for our convergence results for total variation flow and the latter for Allen--Cahn flow.

To avoid additional technical complexities, for the Allen--Cahn flows we only consider the one-dimensional case.

The two main results in this paper are Theorem~\ref{TV1} and Theorem~\ref{CAC}. The former states that (under assumptions that will be given in detail later) the $L^2$-difference between the (piecewise constant interpolation of the) graph-based total variation flow and the continuum total variation flow is controlled by the $L^2$-difference of their initial conditions. In particular, if both flows start from the same initial condition, they are equal. It has been already noted in \cite{LMM} that if a function is constant on rectangles, it remains constant on those rectangles under the anisotropic $\ell^1$-total variation flow. The relation to the piecewise constant interpolation of the graph based total variation flow, however, was not mentioned there. Analysis of the space-discrete total variation flow by means of mode-decompositon methods is presented the recent preprint \cite{CohenBerkovGilboa21}.

The second main result, Theorem~\ref{CAC}, states that (again under assumptions that will be specified in due time) on any finite time interval the difference between the graph-based and continuum Allen--Cahn flows converges to zero, uniformly in time, as the mesh size of the discretisation of the one-dimensional torus goes to zero.

We start the paper with an abstract general framework based on evolution variational inequalities \cite{AGS} for gradient flows. Within this framework we can derive Theorem~\ref{AB1} and Corollary~\ref{CorA} which, after carefully casting the total variation flow problem in the language of the abstract setting and confirming that the required assumptions are satisfied in Sections~\ref{sec:setting} and~\ref{sec:TVflow}, immediately lead to Theorem~\ref{TV1}. Proving our result for the Allen--Cahn flows requires more work, because in this case it is not possible to work with piecewise constant embeddings of the graph-based functions (since the continuum Allen--Cahn functional is not well-defined for functions with jumps) and the linear interpolation embedding we use instead is not an isometry, which means the results from our abstract framework are not directly applicable. Moreover, inequality \eqref{U1} between the graph-based functions $\Phi_h$ and continuum function $\Phi$ does not hold when $\Phi_h$ and $\Phi$ are the graph-based and continuum Allen--Cahn functionals, respectively. Since this inequality is one of the requirements for Theorem~\ref{AB1}, we need to resort to using Theorem~\ref{AB2} instead, which requires a weaker condition \eqref{U2} to be satisfied, but at the cost of leading to a weaker conclusion, namely only an asymptotic comparison (when the mesh size approaches zero).

In this paper we prove the convergence of the flows using the evolution variational inequality formulation of the flows directly. An alternative approach is possible based on \cite[Theorem 4.11]{DMT}, which requires $\Gamma$-convergence of the discrete energies $\Phi_h$ to the continuum energy $\Phi$, asymptotic local equicoercivity of the energies $\Phi_h$ \cite[Definition 4.4]{DMT}, and well-preparedness of the initial conditions. The $\Gamma$-convergence results may be obtained in a way similar to the discrete-to-continuum $\Gamma$-convergence results in \cite{vGB}. Alternatively, we see that condition \eqref{U1} or condition \eqref{U2} allows for the construction of recovery sequences in the proof of $\Gamma$-convergence, and lower semicontinuity of $\Phi$ together with \eqref{eq:Phiinequality} implies the $\liminf$-inequality that is required for $\Gamma$-convergence. The (asymptotically local) equicoercivity condition follows from condition \eqref{L} together with compactness properties of sublevel sets of $\Phi$. We conclude that the a priori alternative approach from \cite{DMT} corresponds to our approach using Theorem~\ref{AB2}, or at the very least, can be employed after establishing the same conditions we use. We note that Theorem~\ref{AB1} does not require the `$\liminf$-condition', but can only be expected to be useful in those special occasion where the discrete gradient flows are
expected to also be continuum gradient flows (after embedding) --- a situation which is not expected to be typical.

We note that the ideas from \cite{DMT} (and relatedly \cite{O05}) have been generalised to functions that are not geodesically $\lambda$-convex (see Section~\ref{sec:framework} for a definition) by Sandier and Serfaty and others \cite{SandierSerfaty04,Serfaty06,Serfaty11,Braides14,F19,MuratoriSavare20}. A general framework provided by Mielke requires convergence of metrics as well as $\Gamma$-convergence of functionals to conclude convergence of gradient flows \cite{Mielke16}[Theorem 3.4.3]. It is assumed that the sets underlying the metric spaces in the sequence are independent of the limit parameter but the metrics can depend on it. We suspect that a suitable modification of such an idea can lead to similar convergence results as ours, but we do not pursue that course in the present paper. For other literature about this energy-dissipation-principle (EDP) convergence, which can be seen as an extension of the work by Sandier and Serfaty, we refer to \cite{MielkeRossiSavare12,DondlFrenzelMielke19,MielkeStephan20,MielkePeletierStephan21,MielkeMontefuscoPeletier21}. A key ingredient is the characterisation of gradient flows as minimizers (at value zero) of a De Giorgi functional. For example, in \cite[Theorem 3.14]{EPSS} this formulation is used to establish the continuum limit of nonlocal-interaction equations on graphs as a gradient flows in probability space equipped with Wasserstein distance.

As examples of other recent literature that is interested in the convergence of graph-based gradient flows to continuum gradient flows, we mention, and \cite{HFE18,HFCE}, in which the limit of the graph-based $p$-Laplacian evolution problem is established.

The paper is structured as follows. The general abstract setting is introduced in Section~\ref{sec:framework} and the setup of the graphs we use, with embeddings and projections into and from the continuum space are given in Section~\ref{sec:setting}. The results from this section will be applied to obtain our two main convergence results, which are given in Sections~\ref{sec:TVflow} (total variation flow) and~\ref{sec:ACflow} (Allen--Cahn flow). The main part of the paper ends with conclusions and suggestions for future work in Section~\ref{sec:conclusions}, after which follow appendices that give detailed proofs for some properties of total variation for piecewise constant functions (Appendix~\ref{app:TVforpiecewiseconstant}), properties of the operator $\Gamma$ which is important in Section~\ref{sec:ACflow} (Appendix~\ref{app:propertiesofGamma}), and two Poincar{\'e}--Wirtinger-type inequalities that we need in the proof of Lemma~\ref{ER} (Appendix~\ref{app:PoincareWirtinger}).

\section{An abstract framework}\label{sec:framework}

Let $(M,d)$ be a complete metric space with a distance $d$. Let $\gamma$  be a continuous function from an interval $I$ to $M,$ i.e., $\gamma \in C(I,M)$. Following \cite{AGS} we say that   $\gamma$   is an absolutely continuous curve if there is an integrable function\footnote{We note that if such an $f$ exists, it necessarily is nonnegative almost everywhere on $I$.} $f$   defined on   $I$ such that, for all $t, s\in I$,
\[d(  \gamma (t),  \gamma (s)) \leq \int_s^t f(r) \, dr.
\]

For any pair of metric spaces $M$ and $\tilde M$, by $C(M; \tilde M)$ we will denote the space of continuous functions with domain $M$ and codomain $\tilde M$. If additionally $M$ and $\tilde M$ are differentiable manifolds and $k$ is a positive integer $k$, then $C^k(M; \tilde M)$ will be the space of $k$ times continuously differentiable functions with domain $M$\footnote{In a slight and common abuse of notation, we will sometimes write $u \in C^k(M;\tilde M)$ if $u: \hat M \to \tilde M$ with $\hat M \supset M$ and $u|_M \in C^k(M;\tilde M)$.} and codomain $\tilde M$. If $\tilde M=\mathbb{R}$, we write $C^k(M)$. Moreover, $C^\infty(M; \tilde M) = \bigcup_{k\in \mathbb{N}} C^k(M; \tilde M)$.

  A curve $\gamma:(0,\infty) \rightarrow M$ is locally absolutely continuous if it is absolutely continuous in any bounded interval $(a,b) \subset (0,\infty)$.

  In this paper we use the notation $\Phi$ for a functional $\Phi : M \rightarrow \mathbb{R} \cup \{+\infty\}$. For $\lambda \in \mathbb{R}$, such a functional is called  $\lambda$-convex on a curve $\gamma \in C([0,1],M)$ if
\[\Phi(\gamma(t)) \leq (1-t)\Phi(\gamma(0))+t\Phi(\gamma(  1 ))-\frac{1}{2} \lambda t(1-t) d^2(\gamma(0),\gamma(1))\]
for all $t\in [0,1]$.  We note that, if $\lambda_2\leq \lambda_1$ and $\Phi$ is $\lambda_1$-convex on $\gamma$, $\Phi$ is also $\lambda_2$-convex on $\gamma$.

A curve $\gamma\in C([0,1],M)$ is a constant speed geodesic (from $\gamma(0)$ to $\gamma(1)$) if
\[d(\gamma(s),\gamma(t))=d(\gamma(0),\gamma(1))(t-s)
\quad \text{for all} \quad s,t \in [0,1] \quad \text{with} \quad s\leq t.
\]
A functional $\Phi$ is said to be geodesically $\lambda$-convex if for any $z_0,z_1 \in \mathcal{D}(\Phi)$ there is a constant speed geodesic $\gamma \in C([0,1],M)$ such that $\gamma(0)=z_0, \gamma(1)=z_1,$ and such that $\Phi$ is $\lambda$-convex on the curve $\gamma$. Here $\mathcal{D}(\Phi):=\{z \in M \mid \Phi(z) < \infty\}$ is the effective domain of $\Phi$.

\begin{remark}
In a Hilbert space geodesic $\lambda$-convexity is nothing but\footnote{Thus in particular, in a Hilbert space setting geodesic $0$-convexity is equivalent to convexity in the sense of vector spaces.} convexity of $\Phi(v)-\frac{\lambda\|v\|^2}{2}$. This follows from the fact that in a (real or complex) Hilbert space for each pair of distinct points $a$ and $b$, there is a unique constant speed geodesic from $a$ to $b$, given by the straight line of the form $\gamma(t) = a + t(b-a)$. It can be checked directly that this is indeed a constant speed geodesic. To prove uniqueness, let $z$ be a point on a constant speed geodesic $\tilde \gamma$ from $a:=\tilde \gamma(0)$ to $b:= \tilde \gamma(1)$, not equal to $a$ or to $b$.
 Then, by the triangle inequality, $\|x+y\|\leq\|x\|+\|y\|$ for $x:=z-a$ and, $y:=b-z$. By the definition of a constant speed geodesic, if $t\in (0,1)$ is such that $\tilde \gamma(t) = z$, we require
\begin{align}
	&\|x\| = \|z-a\| = \|b-a\| t = \|x+y\| t \quad \text{and} \notag\\
	&\|y\| = \|b-z\| = \|b-a\|(1-t) = \|x+y\| (1-t),\label{eq:normrequirements}
\end{align}
so we require equality in the triangle inequality.
In the case of a Hilbert space, this holds if and only if $\left(\|x\|+\|y\|\right)^2-\|x+y\|^2=2\left(\|x\|\|y\|-\langle x,y \rangle \right)=0$, where $\langle x,y \rangle$ denotes the inner product. Hence we require equality in the Cauchy--Schwarz inequality $\langle x, y\rangle \leq \|x\| \|y\|$, which is equivalent to $x$ and $y$ being linearly dependent, i.e., in geometric terms, $x$ and $y$ are parallel. Thus there exists a scalar $c$ such that $y=cx$ and thus $b-z=c(z-a)$. We note that $c\neq -1$, since $a\neq b$. It follows that $z= \frac1{1+c} b + \frac{c}{1+c} a$. Moreover, by \eqref{eq:normrequirements}, $\|x+y\| = |1+c| \|x\| = \frac1t \|x\|$ and thus $\left|\frac1{1+c}\right| = t \in (0,1)$. Hence we can write $\frac1{1+c} = t e^{i\vartheta}$ where $\vartheta := \mathrm{Arg}\left(\frac1{1+c}\right)$. Thus
\begin{align*}
\|\tilde \gamma(t) - a\| &= \|z-a\| = \|t e^{i\vartheta} (b-a) + (e^{i\vartheta} - 1) a\| \geq \left| t \|b-a\| + |e^{i\vartheta} -1| \|a\| \right|,\\
\|\tilde \gamma(t) - b\| &= \|z-b\| = \|(1-t) e^{i\vartheta} (a-b) + (e^{i\vartheta} - 1) b\| \\
&\geq \left| (1-t) \|b-a\| + |e^{i\vartheta} -1| \|b\| \right|.
\end{align*}
By continuity of $\tilde \gamma$, if we take $t\downarrow 0$ in the first line and $t\rightarrow 0$ in the second, we find $|e^{i\vartheta}-1| \|a\| = |e^{i\vartheta} -1| \|b\| =0$. Since $a\neq b$, this implies $e^{i\vartheta} = 1$. Hence $t=\frac1{1+c}$ and thus $\tilde \gamma(t) = z = tb + (1-t) a = a+ t(b-a)$.

 We note that the uniqueness may not be true for a Banach space even if it has finite dimension.
 A simple example is $\mathbb{R}^2$ equipped with $\ell^1$-norm so that the unit ball is a square.

\end{remark}

If $\lambda=0$, then a geodesically $\lambda$-convex $\Phi$ is called geodesically convex; if $\lambda \leq 0$, then such $\Phi$ is often called geodesically semiconvex.

\begin{definition}[Evolution variational inequality formulation of gradient flow]\label{EVI}
Assume that $\Phi$ is geodesically $\lambda$-convex. A gradient flow of $\Phi$ is a locally absolutely continuous curve $u: (0,\infty) \to M$, which satisfies the evolution variational inequality
\[\frac{1}{2} \frac{d}{dt} d^2(u(t),v)+\frac{1}{2} \lambda d^2(u(t),v) \leq \Phi(v)-\Phi(u(t))\]for almost all $t>0$ and for all $v \in \mathcal{D}(\Phi).$   Formally we write $\dot{u} \in -\partial \Phi(u)$. If we wish to emphasize the space or metric in which the gradient flow takes place, we speak of an $M$-gradient flow.
\end{definition}

\begin{remark}
Consider the setting of Definition~\ref{EVI}. For a fixed $v\in \mathcal{D}(\Phi)$, the function $w \mapsto d^2(w,v)$ is Lipschitz continuous on any bounded subset of $M$. By continuity of $u$, if $0<T_1<T_2$, then $u([T_1,T_2]) \subset M$ is bounded. Since the composition of a Lipschitz continuous and absolutely continuous function is again absolutely continuous, on any interval $[T_1,T_2]$, $t\mapsto d^2(u(t),v)$ is absolutely continuous and thus its derivative exists for almost all $t\in [T_1,T_2]$. As $T_1>0$ is arbitrary, the derivative $\frac{d}{dt} d^2(u(t),v)$ exists for almost all $t\in (0,\infty)$.
\end{remark}

\begin{lemma}\label{lem:constantsolution}
Let $\Phi$ be as in Definition~\ref{EVI}. If $c\in M$ is such that, for all $v\in \mathcal{D}(\Phi)$, $\Phi(v) \geq \Phi(c)$, then the constant curve $u=c$ is a gradient flow of $\Phi$.
\end{lemma}
\begin{proof}
Let $v\in \mathcal{D}(\Phi)$. If $\gamma$ is a constant speed geodesic from $c$ to $v$, then by the geodesic $\lambda$-convexity of $\Phi$ we have for all $t\in [0,1]$,
\[
\Phi(\gamma(t)) \leq (1-t) \Phi(c) + t \Phi(v) - \frac12 \lambda t (1-t) d^2(c,v).
\]
By assumption $\Phi$ has a global minimum at $c$, hence
\[
t \Phi(c) \leq \Phi(\gamma(t)) - \Phi(c) + t \Phi(c) \leq t \Phi(v) - \frac12 \lambda t (1-t) d^2(c,v).
\]
Let $t\neq 0$ and divide by $t$ to find
$
\Phi(c) \leq \Phi(v) - \frac12 \lambda (1-t) d^2(c,v).
$
Taking the limit $t\downarrow 0$, we conclude that
$
\Phi(c) \leq \Phi(v) - \frac12 \lambda d^2(c,v)
$
and thus, since $\frac{d}{dt} d^2(c,v) = 0$,
$
\frac12 \frac{d}{dt} d^2(c,v) + \frac12 \lambda d^2(c,v) \leq \Phi(v) - \Phi(c).
$
Therefore the constant curve $u=c$ satisfies the inequality from Definition~\ref{EVI} and hence is a gradient flow of $\Phi$.
\end{proof}

\begin{lemma}\label{lem:Phidecrease}
Let $\Phi$ be as in Definition~\ref{EVI} and additionally assume that $\Phi$ is lower semicontinuous. Let $u$ be a gradient flow of $\Phi$ with $u((0, \infty)) \subset \mathcal{D}(\Phi)$ and continuous on $[0,\infty)$. Then $t\mapsto \Phi(u(t))$ is a non-increasing function on $[0, \infty)$.

Moreover, if $I\subset [0,\infty)$ is a nondegenerate interval, $\Phi$ is as in Definition~\ref{EVI} (not necessarily lower semicontinuous) and $u$ is as above, then for almost all $t\in I$, $\Phi(u(t)) = \Phi(u(\inf I)) < +\infty$ if and only if, for all $t\in I$, $u(t) = u(\inf I)$.
\end{lemma}
\begin{proof}
First we prove the following claim: If $v \in M$, then $t \mapsto d^2(u(t), v)$ is locally absolutely continuous on $(0,\infty)$. To prove this, let $(a,b) \subset (0,\infty)$, since $u$ is locally absolutely continuous on $(0,\infty)$, if $a<b<\infty$, there exists an integrable function $f$ on $(a,b)$ such that, for all $s,t\in (a,b)$,
\[
d(u(t), u(s)) \leq \int_s^t f(r) \, dr.
\]
Since $u$ is continuous on $[a,b]$, so is $t\mapsto d(u(t), v)$, hence this function is bounded on $(a,b)$. In particular, there is a $C>0$ such that, for all $s,t\in (a,b)$, $|d(u(t),v)+d(u(s),v)| \leq C$, hence, for all $s,t\in (a,b)$,
\begin{align*}
|d^2(u(t),v) - d^2(u(s), v)| &= |d(u(t),v) + d(u(s), v)|  |d(u(t),v) - d(u(s),v)| \\
&\leq C d(u(t),u(s))\\
&\leq \int_s^t C f(r) \, dr,
\end{align*}
where we used the reverse triangle inequality for the first inequality.
Hence $t\mapsto d^2(u(t),v)$ is absolutely continuous on $(a,b)$ and thus locally absolutely continuous on $(0,\infty)$.

Next we prove that $\Phi \circ u$ is non-increasing. Let $t^* \in [0,\infty)$.
If $\Phi(u(t^*)) = +\infty$, then it is trivially true that, for all $t\geq t^*$, $\Phi(u(t)) \leq \Phi(u(t^*))$. Now assume that $\Phi(u(t^*)) < +\infty$, so that $u(t^*) \in \mathcal{D}(\Phi)$.
Define
\[
S := \{t>t^*: \Phi(u(t)) > \Phi(u(t^*))\}.
\]
Since $\Phi$ is lower semicontinuous and $u$ is (absolutely) continuous, $\Phi \circ u$ is lower semicontinuous, hence $S$ is an open set. For a proof by contradiction, assume that $S\neq\emptyset$, then $S$ is an (at most) countable union of pairwise disjoint open intervals \cite[Theorem 2.24]{Parthasarathy22}. Since $S$ is bounded below, so is each of these intervals. Let $J \subset S$ be such an interval. Since $S$ is open, $a:=\inf S \in [t^*, \infty) \setminus S$, thus $\Phi(u(a)) \leq \Phi(u(t^*))$. In particular, $a\not \in J$.

For notational convenience, define $g: J \to \mathbb{R}$ by $g(t) := d^2(u(t), u(a))$. Recall that $g$ is locally absolutely continuous on $J$.

In the inequality in Definition~\ref{EVI}, choose $v=u(a)$. Then, for almost all $t\in J$,
\[
\frac12 \left[  \frac{d}{dt} g(t) + \lambda g(t) \right]  \leq \Phi(u(a)) - \Phi(u(t)) \leq \Phi(u(t^*)) - \Phi(u(t)) < 0.
\]
We apply a Gr\"onwall argument: We have, for almost all $t\in J$,
\[
\frac{d}{dt}\left(e^{\lambda t} g(t)\right) = e^{\lambda t} \left[  \frac{d}{dt} g(t) + \lambda g(t) \right] < 0.
\]
Let $t\in J$ and integrate the inequality over $(a,t)$, using $g(a)=0$, to find
\[
e^{\lambda t} g(t) - e^{\lambda a} g(a) = e^{\lambda t} g(t) < 0
\]
and hence
\[
g(t) < 0,
\]
which contradicts $g(t) = d^2(u(t), u(a)) \geq 0$. Therefore $S=\emptyset$ and thus, for all $t\geq t^*$, $\Phi(u(t)) \leq \Phi(u(t^*))$, which proves the first claim of the lemma.

To prove the second claim, let $I\subset [0,\infty)$ be a nondegenerate interval such that, for almost all $t\in I$, $\Phi(u(t)) = \Phi(u(b)) < +\infty$, where $b:= \inf I$. The ``if'' statement follows directly since by assumption, if $t\in I\setminus\{0\}$, then $\Phi(u(t)) < +\infty$. For the ``only if'' claim, note that, by the inequality in Definition~\ref{EVI} with $v=u(b)$, we have for almost all $t\in I$, $h'(t) + \lambda h(t) \leq 0$, where $h(t):=d^2(u(t), u(b))$. By a similar argument as above for $g$, we find that $h(t) \leq 0$, hence $h(t) = 0$, and thus $u(t) = u(b)$ for all $t\in I$
\end{proof}

\begin{remark}\label{rem:lsc1}
 In the first part of Lemma~\ref{lem:Phidecrease} above, we assume that $\Phi$ is lower semicontinuous. In fact, in the proof we only use the fact that $\Phi \circ u$ is lower semicontinuous and thus the conclusions of the lemma remain true under this weaker assumption.
\end{remark}

\begin{remark}\label{rem:Hilbertspace}
Definition~\ref{EVI} is a natural extension of a gradient flow in a Hilbert space with an inner product $\langle , \rangle$. By \cite[Remark 1.1.3]{AGS}, in the Hilbert space setting the derivative $\dot u(t)$ of an absolutely continuous curve $u$ exists for a.e. $t$. Assume $t$ is such that $\dot{u}(t)$ exists. If $\Phi$ is convex ($\lambda=0$), then by definition of the subdifferential \cite{ET} we see that $\dot{u}(t) \in -\partial \Phi(u(t))$ is equivalent to, for all $v\in \mathcal{D}(\Phi)$,
 \[\langle -\dot{u}(t), v-u(t) \rangle \leq \Phi(v)-\Phi(u(t)).\]
The left hand side equals $\frac{d}{dt} \frac{\|u(t)-v\|^2}{2}$ and $u$ satisfies the evolution variational inequality (Definition~\ref{EVI})   for $\lambda=0$.

For $\lambda \neq 0$ we interpret $\dot{u} \in -\partial \Phi(u)$ as (for those $t$ for which $\dot{u}(t)$ exists)
\begin{equation}\label{eq:lambdaconvexsubdiff}
\dot{u}(t) \in -\partial\left(\Phi(u(t))-\lambda \frac{\|u(t)\|^2}{2}\right) -\lambda u(t),
\end{equation}
which is equivalent to, for all $v\in \mathcal{D}(\Phi)$,
\begin{equation}\label{eq:lambdaconvexflow}
\langle -\dot{u}(t)-\lambda u(t), v-u(t) \rangle \leq \Phi(v)-\lambda \frac{\|v\|^2}{2}-\left(\Phi(u(t))-\lambda \frac{\|u(t)\|^2}{2} \right).
\end{equation}
Since
\[\langle \lambda u(t), u(t)-v\rangle+\lambda\frac{\|v\|^2}{2}-\lambda\frac{\|u(t)\|^2}{2}=\lambda \frac{\|u(t)-v\|^2}{2},\]
the function $u$ satisfies the evolution variational inequality.
\end{remark}

We recall a basic unique solvability result. The next result is a special case of \cite[Theorem 4.0.4]{AGS}, where more general $\Phi$ are allowed.
\begin{proposition} \label{UE}
Assume that $\Phi: M \to \mathbb{R}\cup\{+\infty\}$ is geodesically $\lambda$-convex and lower semicontinuous with $\Phi \not\equiv +\infty.$ Assume furthermore that $\Phi$ is locally bounded below at some point $u_* \in \mathcal{D}(\Phi)$ i.e.,
\[m_* = \inf\left\{\Phi(v) \mid v \in M, d(v,u_*)<r_* \right\}>-\infty\]
for some $r_*>0$. Then for any $u_0 \in \overline{\mathcal{D}(\Phi)}$ (the closure of $\mathcal{D}(\Phi)$ in the topology of $M$) there exists a solution $u \in C([0,\infty),M)$ to $\dot{u} \in -\partial \Phi(u)$ satisfying $u(0)=u_0$.   Furthermore, for all $t>0$, $u(t) \in \mathcal{D}(\Phi)$. Moreover, such a solution is unique among all locally absolutely continuous curves $u$ for which $u(t)\rightarrow u_0$ as $t\downarrow 0$.
\end{proposition}

\begin{remark}\label{lem:lsc2}
 Without the assumption of lower semicontinuity of $\Phi$ in Proposition~\ref{UE} above, global-in-time existence with $u(t) \in \mathcal{D}(\Phi)$ for all $t>0$ is not guaranteed, even though a local-in-time solution may exist. Consider for example the function $\Phi: \mathbb{R} \to \mathbb{R}\cup\{+\infty\}$, defined by
 \[
  \Phi(v) := \begin{cases}
   -v, &\text{if } v<0,\\
   +\infty, &\text{if } v\geq 0,
  \end{cases}
 \]
which is convex and locally bounded (at all points in its effective domain), but not lower semicontinuous. From Definition~\ref{EVI} ---with $\lambda=0$ and the Euclidean distance on $\mathbb{R}$ for $d$--- it follows that, for any $a>0$, $u(t)=t-a$ satisfies the evolution variational inequality for $t<a$. However, at $t=a$ the curve leaves the effective domain of $\Phi$, i.e., $u(a) \not\in \mathcal{D}(\Phi)$. Moreover, any function $\tilde u$ that satisfies the evolution variational inequality locally in time must be of the form $\tilde u(t)=t-b$, for some $b>0$. Indeed, if $\Phi(\tilde u(t))=+\infty$, the evolution variational inequality does not hold, thus $\tilde u(t)<0$ has to hold for almost all $t>0$ and thus, by continuity, for all $t>0$. By Remark~\ref{rem:Hilbertspace}, the evolution variational inequality in this case is equivalent to the differential inclusion $\dot{\tilde{}} u(t) \in \partial \Phi(\tilde u(t)) = \{1\}$, where the equality follows from differentiability of $\Phi$ on $(-\infty,0)$ with $\Phi'(v)=-1$ \cite[Proposition 5.3]{ET}. Thus $\tilde u(t) = t-b$, as claimed and no global-in-time solutions exist.
On the other hand, if we define $\Phi(0):=0$ instead, then $\Phi$ is lower semicontinuous and indeed the function $u(t)=\min(t-a, 0)$ is a global solution.
\end{remark}

\begin{definition}\label{EVI2}
The solution from Proposition~\ref{UE} is called the gradient flow of $\Phi$ starting from $u_0$.
\end{definition}

In the setting of Proposition~\ref{UE}, $\Phi(u(t)) < +\infty$ for all $t>0$. Hence the evolution variational inequality from Definition~\ref{EVI}, which was assumed to hold for all $v\in \mathcal{D}(\Phi)$, trivially also holds for all $v\in M\setminus\mathcal{D}(\Phi)$.

The results when $M$ is a Hilbert space go back to the theory of maximal monotone operators \cite{Br73} pioneered by K\={o}mura \cite{Ko}, where the first existence result has been established. We note that the assumption $m_*>-\infty$ for some $u_*$ and $r_*$ is automatically fulfilled.

To observe the uniqueness part of Proposition \ref{UE}, let $u$ and $v$ be two solutions, then
\[\frac{1}{2} \frac{d}{dt} d^2(v(t),u(s))+\frac{1}{2}\lambda d^2(v(t),u(  s  ))\leq \Phi(u(s))-\Phi(  v (t) )\]
and
\[\frac{1}{2}\frac{d}{dt}d^2(u(t),v(s))+\frac{1}{2} \lambda d^2(u(t),v(s))\leq \Phi(v(s))-\Phi(u(t)).\]
Adding both sides and observing that
\[\frac{d}{dt}d^2(u(t),v(t)) = \frac{d}{dt} \left.d^2(u(t),v(s))\right|_{s=t}+\frac{d}{dt} \left.d^2(v(t),u(s))\right|_{s=t}, \]
we take $s=t$ to get
\[\frac{1}{2}\frac{d}{dt} d^2(u(t),v(t))+\lambda d^2(u(t),v(t)) \leq 0.\]
This implies a generalised contraction principle: for all $t \geq 0$,
\begin{equation}\label{C}
	d^2(u(t),v(t)) \leq e^{-2\lambda t} d^2(u(0),v(0)).
\end{equation}
If the initial data $u(0)$ and $v(0)$ are the same, then we conclude that $u \equiv v$. This idea is useful to define our convergence result.

  In what follows, we   consider a family of complete metric spaces $(M_h,d_h)_{h\in\mathfrak{H}}$ approximating $(M,d)$ as $h \rightarrow 0$ in some sense which is to be made precise. Here $\mathfrak{H}$ is a nonempty subset of $(0,1)$ of which $0$ is a limit point. In particular, we allow $\mathfrak{H}$ to be a proper subset of $(0,1)$. In fact, from Section~\ref{sec:setting} onwards we only consider $h\in (0,1)$ for which $\frac1{h}$ is an integer.

 Let $i_h:M_h \rightarrow M$ be an embedding map which is an isometry, i.e., $d_h(x,y)=d(i_h x,i_h y)$, for all $x,y \in M_h.$
 Let   $p_h: M \to M_h$ be a mapping   such that $p_h \circ i_h$ is the identity map on $M_h$. This is (in spirit) a projection,   in the sense that $p_h \circ i_h \circ p_h = p_h$ on $M$. We assume the following non-strict contraction property holds:
\begin{equation}\label{eq:contraction}
d_h(p_h x,p_h y) \leq d(x,y) \qquad   \text{for all} \quad x,y \in M.
\end{equation}
  In particular $p_h$ is continuous.

We want to consider functionals $\Phi : M \rightarrow \mathbb{R} \cup \{+\infty\}$ and, for all $h\in \mathfrak{H}$, $\Phi_h: M_h \rightarrow \mathbb{R} \cup \{+\infty\}$, for which gradient flows exist. Hence we will require that these functionals satisfy the conditions of Proposition \ref{UE}, i.e., they are geodesically $\lambda$-convex, lower semicontinuous, not identically equal to $+\infty$ and locally bounded below at some point in their respective domains. If, for all $h\in \mathfrak{H}$, the functionals $\Phi_h$ and $\Phi$ satisfy these conditions with the same $\lambda$, we will say $\Phi_h$ and $\Phi$ satisfy condition ($\lambda$-AGS), in deference to \cite[Theorem 4.0.4]{AGS}. We note that we do not require the points in the effective domain at which local lower boundedness is satisfied to be the same for all functionals, nor the values of the lower bounds or the radii of the local neighbourhoods on which the bounds hold.

  Additionally, we sometimes assume any or all of the following. When we do so, we will explicitly state this.
\begin{equation} \label{L}
	\Phi(i_h v) \leq \Phi_h(v) \qquad \text{for all } h \in \mathfrak{H} \text{ and all }  v \in M_h
\end{equation}
\begin{equation} \label{U1}
	\Phi_h(p_h w) \leq \Phi(w) \qquad \text{for all } h \in \mathfrak{H} \text{ and all }  w \in M.
\end{equation}
 \begin{equation} \label{P}
 	d^2_h(v,p_h w)+d^2(  i_h   p_h w,w)=d^2(  i_h   v,w)  \qquad \text{for all } h \in \mathfrak{H}, \text{ all } v \in M_h, \text{ and all } w\in M.
 \end{equation}
  The last assumption is similar to the Pythagorean theorem and thus says that as a projection $p_h$ is `orthogonal', in a sense.   A typical   example where this Pythagorean theorem holds is the case where $M$ and $M_h$ are inner product spaces with inner products $\langle \cdot, \cdot \rangle$ and $\langle \cdot, \cdot \rangle_h$, respectively, which satisfy $\langle i_h u, i_h v\rangle = \langle u, v\rangle_h$ and
\begin{equation}\label{Pythagoras}
\langle i_h p_h w-w, i_h u\rangle=0,
\end{equation}
for all $h\in \mathfrak{H}$, all $w\in M$, and all $u, v\in M_h$.

\begin{theorem}\label{AB1}
  Assume there exists a $\lambda \leq 0$ such that $\Phi_h$ and $\Phi$ satisfy condition ($\lambda$-AGS).
Assume that $\overline{\mathcal{D}(\Phi_h)}=M_h.$ Let $u^h$ be the gradient flow of $\Phi_h$ starting from $u^h_0 \in M_h.$ Assume moreover \eqref{L}, \eqref{U1}, and \eqref{P}. Then $i_h u^h$ is the gradient flow of $\Phi$ starting from $i_h u^h_0$.
\end{theorem}
\begin{corollary}\label{CorA}
  Assume there exists a $\lambda \leq 0$ such that $\Phi_h$, $M_h$, $\Phi$, $M$, and $u^h$ satisfy all the hypotheses from Theorem \ref{AB1}.  Let $u$ be the gradient flow of $\Phi$ starting from $u_0 \in M = \overline{\mathcal{D}(\Phi)}$. Then \[d^2(i_h u^h(t),u(t)) \leq e^{-2\lambda t} d^2(i_h u^h_0,u_0), \qquad \text{for all} \quad  t>0.\]
\end{corollary}
  If $i_h u^h_0 \rightarrow u_0$ as $h \rightarrow 0$, this yields, for all $T>0$, uniform convergence of $i_h u^h$ to $u$ on $(0,T)$ as $h \rightarrow 0$.  For example if $  i_h   p_h u_0 \rightarrow u_0$ in $M$ as $h \rightarrow 0$ and $u^h_0=p_h u_0$, then $i_h u^h$ converges to $u$ in $M$ uniformly on $(0,T)$ as $h \rightarrow 0$, for any $T>0$. Corollary~\ref{CorA} follows from Theorem \ref{AB1} if we recall the generalised contraction principle \eqref{C}.
\begin{proof}[Proof of Theorem \ref{AB1}]
By definition our assumptions \eqref{L} and \eqref{U1} yield
\begin{align*}
\frac{1}{2} \frac{d}{dt} d^2_h(u^h,p_h w)+\frac{\lambda}{2} d^2_h (u^h,p_h w) &\leq \Phi_h(p_h w)-\Phi_h(u^h) \\
&\leq \Phi(w)-\Phi(i_h u^h) \qquad \text{for all} \quad w \in M.
\end{align*}
By \eqref{P} we see
\[\frac{d}{dt}d^2_h(u^h,p_h w)=\frac{d}{dt}d^2(i_h u^h,w)\] since $d^2(  i_h   p_h w,w)$ does not depend on $t$. We thus conclude that \[\frac{1}{2} \frac{d}{dt} d^2 (i_h u^h,w)+\frac{\lambda}{2} d^2(i_h u^h,w) \leq \Phi(w)-\Phi(i_h u^h), \qquad \text{for all} \quad w \in M, \]   since $\lambda \leq 0$ and $d_h(u^h, p_h w) = d_h(p_h i_h u^h, p_h w)\leq d(i_h u^h, w)$ by the non-strict contraction property \eqref{eq:contraction}.   This says that $i_h u^h$ is the gradient flow of $\Phi$.
\end{proof}
The next theorem gives convergence under a weaker assumption. Instead of \eqref{U1} we consider the weaker assumption
\begin{equation} \label{U2}
	\overline{\lim_{h \rightarrow 0}} \Phi_h(p_h w) \leq \Phi(w)
\qquad \text{for all} \quad w \in M.
\end{equation}

\begin{theorem}\label{AB2}
  Assume there exists a $\lambda \leq 0$ such that $\Phi_h$ and $\Phi$ satisfy condition ($\lambda$-AGS) and let conditions \eqref{P} and \eqref{U2} be satisfied.   Let $u^h$ be the gradient flow of $\Phi_h$ starting from $u^h_0 \in M_h=\overline{\mathcal{D}(\Phi_h)}$ and let $u$ be the gradient flow of $\Phi$ starting from $u_0\in M=\overline{\mathcal{D}(\Phi)}$. Assume further that there exists a $T>0$, a $\delta>0$, and a nonnegative function $\Psi: M \to \mathbb{R} \cup \{+\infty\}$ such that $\Psi(u(\cdot)) \in L^1(0,T)$\footnote{By $L^1(0,T)$ we denote the set of Lebesgue integrable functions on $(0,T)$.} and, for all $w\in M$ and for all $h\in (0,\delta) \cap \mathfrak{H}$,
\begin{equation}\label{eq:PhiPsi}
\Phi_h(p_h w) \leq \Psi(w).
\end{equation}
Moreover, let there be a function $\varepsilon: (0,1) \to (0,\infty)$ such that   $\varepsilon(h) \to 0$ as $h \to 0$ and such that, for all $t\in[0,T]$,
\begin{equation}\label{eq:Phiinequality}
	\Phi \left(i_h u^h(t) \right) \leq \Phi_h \left(u^h(t) \right) + \varepsilon(h).
\end{equation}
Then
 \[\lim_{h\rightarrow 0} \sup_{t\in [0,T]} d(i_h u^h(t),u(t))=0,\] if $i_h u^h_0 \rightarrow u_0.$
\end{theorem}
\begin{proof}
As in the proof of uniqueness we observe that
\begin{equation}\label{GH}
	\frac{1}{2} \frac{d}{dt} d^2_h(u^h(t),p_h u(s))+\frac{\lambda}{2}d^2_h(u^h(t),p_h u(s)) \leq \Phi_h(p_h u(s))-\Phi_h(u^h(t))
\end{equation}
\begin{equation}\label{G}
	\frac{1}{2} \frac{d}{dt} d^2(u(t),i_h u^h(s))+\frac{\lambda}{2}d^2(u(t),i_h u^h(s)) \leq \Phi(i_h u^h(s))-\Phi(u(t)).
\end{equation}
By   the chain rule and   \eqref{P} we see that
\begin{align*}
\frac{d}{dt}d^2(u(t),i_h u^h(t)) &= \frac{d}{dt} \left.d^2(u(t),i_h u^h (s))\right|_{s=t}+\frac{d}{dt} \left.d^2(i_h u^h(t),u(s))\right|_{s=t} \\
&= \frac{d}{dt} \left.d^2(u(t),i_h u^h (s))\right|_{s=t}+\frac{d}{dt} \left.d^2_h(u^h(t),p_h u(s))\right|_{s=t}.
\end{align*}
Adding \eqref{GH} and \eqref{G} we now observe that
\begin{align*}
    &\hspace{0.4cm} \frac{1}{2} \frac{d}{dt} d^2(u(t),i_h u^h(t)) +\lambda d^2(u(t),i_h u^h(t)) \\
    &  \leq \frac{1}{2} \frac{d}{dt} d^2(u(t),i_h u^h(t)) +\frac\lambda2 d^2(u(t),i_h u^h(t)) + \frac\lambda2 d_h^2(u^h(t), p_h u((t))  \\
    &\leq \Phi(i_h u^h(t))-\Phi_h(u^h(t))+\Phi_h(p_h u(t))-\Phi(u(t))
	\leq \varepsilon(h) + \Phi_h \left(p_h u(t) \right) - \Phi \left(u(t) \right),
\end{align*}
  where for the first inequality we used that, by the non-strict contraction property \eqref{eq:contraction}, $d_h^2(u^h(t), p_hu(t)) = d_h^2(p_h i_h u^h(t), p_hu(t)) \leq d^2(i_h u^h(t), p_hu(t))$ and that $\lambda \le 0$.
We integrate this differential inequality for $d^2(u(t),i_h u^h(t))$ over $[0,t]$, for a $t\in [0,T]$, to get
\begin{align*}
d^2(u(t),i_h u^h(t)) &\leq e^{-2 \lambda t} d^2(u_0,i_h u^h_0) \\
&\hspace{0.3cm} + 2 e^{-2\lambda t} \int_0^t e^{2\lambda \tau}   \left\{\Phi_h(p_h u(\tau))-\Phi(u(\tau)) \right\} d\tau +   2   \varepsilon(h) t,
\end{align*}
  where we used that $\int_0^t e^{2\lambda \tau}\, d\tau \leq t$, as $\lambda \leq 0$.   Because
\[
t\mapsto e^{-2\lambda t} \int_0^t e^{2\lambda \tau} \left\{\Phi_h(p_h u(\tau))-\Phi(u(\tau)) \right\} d\tau
\]
is continuous on $[0,T]$, there exists a $t^* \in [0,T]$ such that
\begin{align*}
\sup_{0<t<T} d^2(u(t),i_h u^h(t)) &\leq e^{-2 \lambda T} d^2(u_0,i_h u^h_0) \\
&\hspace{0.3cm} + e^{-2\lambda t^*} \int_0^{t^*} e^{2\lambda \tau} \left\{\Phi_h(p_h u(\tau))-\Phi(u(\tau)) \right\} d\tau\\ &\hspace{0.3cm} + 2 \varepsilon(h) T,
\end{align*}
where we used that $\lambda \leq 0$. Since $\Psi$ is nonnegative, $\Psi(u(\tau)) \in L^1(0,T)$, and, for all $h\in (0,\delta) \cap \mathfrak{H}$, $\Phi_h(p_h u(\tau)) \leq \Psi(u(\tau))$, by the reverse Fatou lemma and \eqref{U2} we see that
\begin{multline*}
\overline{\lim_{h \rightarrow 0}} e^{-2\lambda t^*} \int_0^{t^*} e^{2\lambda \tau} \left\{\Phi_h(p_h u(\tau))-\Phi(u(\tau)) \right\} d\tau \\
\leq e^{-2\lambda t^*} \int_0^{t^*} e^{2\lambda \tau} \overline{\lim_{h \rightarrow 0}}  \left\{\Phi_h(p_h u(\tau))-\Phi(u(\tau)) \right\} d\tau \leq 0.
\end{multline*}
We deduce that
\[
\overline{\lim_{h \rightarrow 0}} \sup_{0<t<T}  d^2(u(t),i_h u^h(t)) \leq e^{-2\lambda T} \overline{\lim_{h \rightarrow 0}}  d^2(u_0,i_h u^h_0) + 0.
\]
Since $\underline{\lim}_{h \rightarrow 0} \sup_{0<t<T}  d^2(u(t),i_h u^h(t)) \geq 0$, the result follows.

\end{proof}

We use Theorem~\ref{AB1} and Corrollary~\ref{CorA} to prove convergence of discrete total variation flows to a continuum total variation flow in Theorem~\ref{TV1}. These results are not applicable in our proof of convergence of discrete Allen--Cahn flows to a continuum Allen--Cahn flow in Theorem~\ref{CAC}, as explained in detail in Section~\ref{sec:extension}. Instead, we first approximate the continuum flow by different discrete flows and use Theorem~\ref{AB2} to prove convergence of those new flows to the continuum flow in Theorem~\ref{CAC2}, after which Theorem~\ref{CAC3} establishes that the new flows approximate the discrete flows, resulting in Theorem~\ref{CAC}.

\section{Setting of the problem}\label{sec:setting}

  We consider an anisotropic total variation flow and an Allen--Cahn flow on  `cubic' graphs and their convergence to their continuum versions when the mesh size of the graphs tends to zero. In this section we define the graphs, the corresponding metric spaces $M_h$ and their continuum counterpart $M$, the embeddings $i_h$ and `projections' $p_h$, and the functionals which generate the gradient flows.

\subsection{The graphs}\label{sec:thegraphs}

We construct a graph $G_h$ by discretising a flat torus $\mathbb{T}^n := (\mathbb{R}/\mathbb{Z})^n$ with mesh size $h \in (0,1)$. In the remainder of this paper we assume that $h$ is such that $1/h$ is an integer, i.e.,
\[
\mathfrak{H} := \left\{h \in (0,1) \Bigm|  \exists k \in \mathbb{N} \ h = \frac1k \right\}.
\]
For any $n\in\mathbb{N}$, we denote the Euclidean norm on $\mathbb{R}^n$ by $|\cdot|$. We use the same notation for the induced Euclidean norms on $\mathbb{T}^n$.

Define the graphs $G_h=(V_h,E_h)$ via their node sets $V_h$ and edge sets $E_h$:
\begin{align*}
	V_h &:= \left( h (\mathbb{Z}/h^{-1}\mathbb{Z})\right)^n = \{0, h, 2h, \ldots, 1-h\}^n,\\
	E_h &:= \left\{ (z, \tilde z) \in V_h \times V_h \Bigm| |z-\tilde z|_{\ell^1} = h \right\}\\
	&\hspace{1cm} \cup \left\{(z, \tilde z) \in V_h \times V_h \Bigm| \exists j \ |z_j - \tilde z_j| = 1-h \text{ and } \forall i\neq j \ z_i = \tilde z_i\right\}.
\end{align*}

Here $|x|_{\ell^1}:=\sum^n_{i=1}|x_i|$. We parametrise the torus $\mathbb{T}^n$ by the hypercube $\Omega_h := [-h/2,1-h/2)^n$ with periodic boundary conditions (see Figure~\ref{fig:Omegah} for the case with $n=2$). Then the first set in $E_h$ contains the edges within $\overline{\Omega_h}$ and the second set the edges that `go over' the boundary of $\Omega_h$ and capture the periodicity on $\mathbb{T}^n$. We note that $(z, \tilde z) \in E_h$ if and only if $(\tilde z, z) \in E_h$. We identify these edges with each other so that the graphs are undirected. When it is clear from the context what $h$ is, we also write $z \sim \tilde z$ to indicate $(z, \tilde z) \in E_h$.

\begin{figure}
 \begin{center}
 \includegraphics[scale=.25]{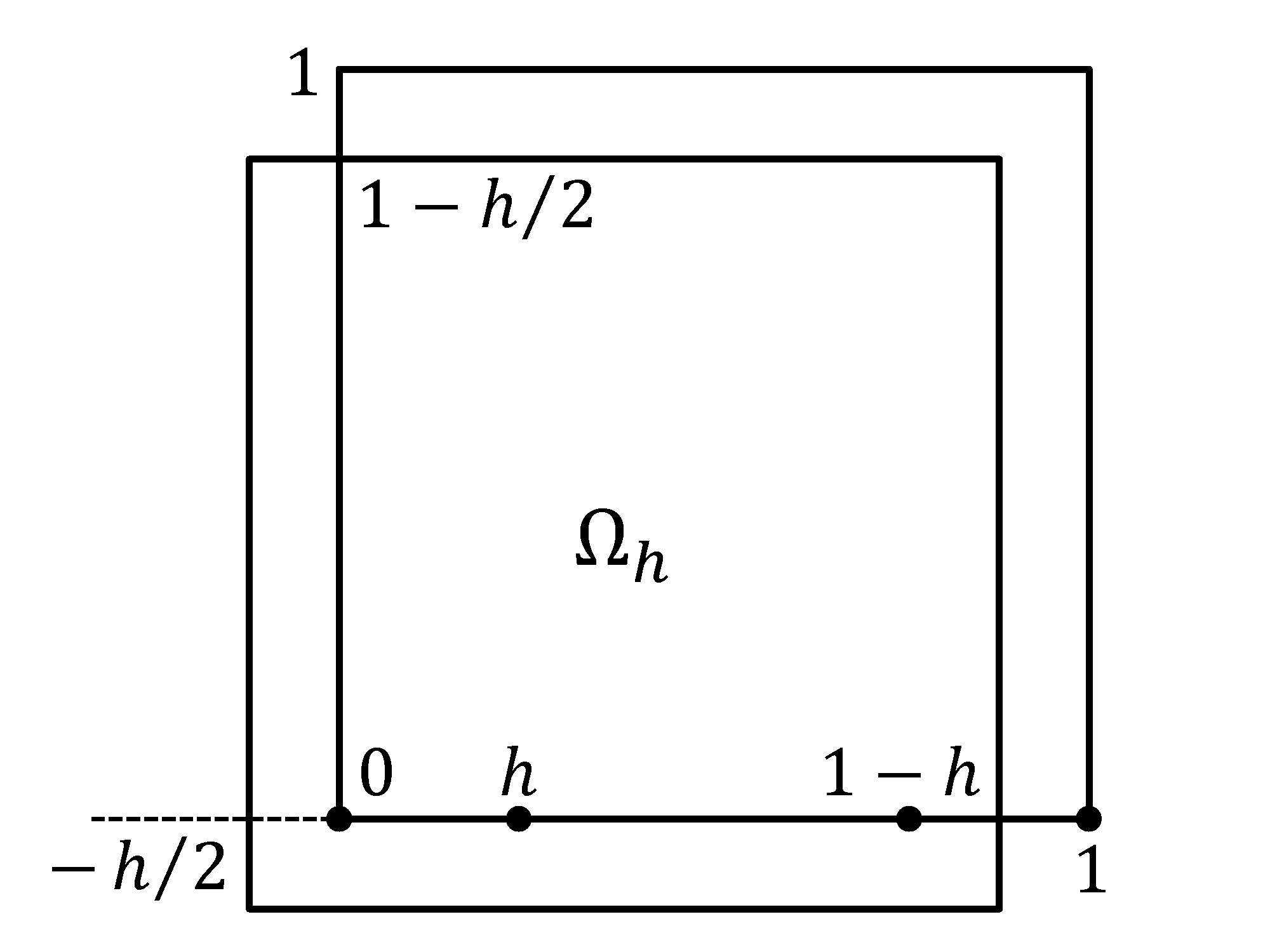}
 \end{center}
 \caption{The square $\Omega_h := [-h/2,1-h/2)^2$ with periodic boundary conditions parametrises the torus $\mathbb{T}^2$. The black circles indicate some of the nodes in the corresponding node set $V_h$, where the node at $(0,1)$ is identified with the one at $(0,0)$.}
\label{fig:Omegah}
\end{figure}

We note that elsewhere in the literature (e.g., in \cite{vGB} and references therein) the graph $G_h$ is viewed as an edge-weighted graph, in which to each edge a weight $h^{n-1}$ is assigned (and to each pair of unconnected nodes a weight of zero). In this paper we have chosen to directly incorporate this dependence on the mesh size into the definitions of our inner products and functionals which we give below (e.g., the factors $h^{n-1}$ and $h^n$ in $\|\nabla u \|_1$ and $\langle \cdot, \cdot \rangle_h$, respectively).

Using the notation $[n]:=\{1, \ldots, n\}$,
for $z\in V_h$ we define
\[
Q_z^h := \left\{ hy+z\in \mathbb{R}^n \Bigm| \forall i\in [n] \  -1/2 \leq y_i < 1/2 \right\},
\]
which is a `half-open' $n$-cube centred at the vertex $z$ with edges of length $h$. With
\begin{equation}\label{eq:Izh}
I_{z_i}^h := [z_i-h/2, z_i+h/2),
\end{equation}
we have $Q_z^h = \prod_{i=1}^n I_{z_i}^h := I_{z_1}^h \times \ldots \times I_{z_n}^h$. We note that, if $f$ is integrable on $\mathbb{T}^n$, then
\[
 \int_{\mathbb{T}^n} f(x) \, dx = \sum_{z\in V_h} \int_{Q_z^h} f(x) \, dx.
\]

\subsection{Function spaces, embedding, and projection}

Let $h\in \mathfrak{H}$. We define the $L^2$ inner product of functions $v_1, v_2: V_h\to\mathbb{R}$ as
\[
	\langle v_1, v_2 \rangle_h := \sum_{z \in V} v_1(z) v_2(z) h^n
\]
We denote the inner product space of all functions $v: V_h \to \mathbb{R}$ equipped with $\langle \cdot, \cdot \rangle_h$ by $L_h^2$. If we want to consider the set of real-valued vertex functions $v: V_h \to \mathbb{R}$ without a priori assuming it is equipped with a specific inner product, we denote this by $\mathcal{V}_h$. For functions $v\in \mathcal{V}_h$ we also define the norm $\|v\|_\infty := \max_{z\in V} |v(z)|$.

Since $L_h^2$ is a finite ($n$-)dimensional inner product space, it is complete (and thus a Hilbert space) and the corresponding norm $\|v\|_h := \sqrt{\langle v, v\rangle_h}$ is finite for all $v\in L_h^2$. In Section~\ref{sec:TVflow}, $L_h^2$ will play the role that $M_h$ played in Section~\ref{sec:framework}. In Section~\ref{sec:ACflow} we will introduce a different inner product on $\mathcal{V}_h$ and the resulting Hilbert space $\overline{L}_h^2$ will play the role of $M_h$ in that section instead (see Section~\ref{sec:extension}).

The role of $M$ will be played by the Lebesgue space $L^2(\mathbb{T}^n)$ consisting of all functions $w: \mathbb{T}^n \to \mathbb{R}$ which have finite norm $\|w\|_{L^2(\mathbb{T}^n)} := \sqrt{\langle w, w\rangle_{L^2(\mathbb{T}^n)}}$, where
\[
 \langle w_1, w_2\rangle_{L^2(\mathbb{T}^n)} := \int_{\mathbb{T}^n} w_1(x) w_2(x) \, dx = \int_{\Omega_h} w_1(x) w_2(x) \, dx,
\]
for functions $w_1, w_2 \in L^2(\mathbb{T}^n)$. We emphasize that the value of the integral on the right does not depend on the choice of $h$.
For later use, for $w: \mathbb{T}^n\to \mathbb{R}$ and $p\in\mathbb{N}$ we also define the norms $\displaystyle \|w\|_{L^p(\mathbb{T})} := \left(\int_{\mathbb{T}} |w(x)|^p \, dx\right)^{1/p}$ and $\|w\|_{L^\infty(\mathbb{T})} := \text{esssup}_{x\in \mathbb{T}} |w(x)|$.

We define the embedding operator $i_h: L_h^2 \to L^2(\mathbb{T}^n)$. If $v\in L_h^2$ and $x\in \mathbb{T}^n$, we set
\[
i_h v(x) := v(z),
\]
where $z\in V_h$ is such that $x\in Q_z^h$. We note that, by our identification of $\mathbb{T}^n$ with $\Omega_h$ (with periodic boundary conditions) for each $x$, $z$ is uniquely determined.

Next we define the projection operator $p_h: L^2(\mathbb{T}^n) \to L_h^2$. If $w\in L^2(\mathbb{T}^n)$ and $z\in V_h$, then
\[
p_hw(z) := \frac1{|Q_z^h|}\int_{Q_z^h} w(x) \, dx = h^{-n} \int_{Q_z^h} w(x) \, dx,
\]
where $|\cdot|$ denotes the Lebesgue measure and the second equality follows since $|Q_z^h| = h^n$, for all $z\in V_h$.

We end this subsection with some useful properties of $i_h$ and $p_h$. First we introduce some new notation which we will require in the proof of Theorem~\ref{thm:ih_ph_properties} (v\hspace{-1pt}i\hspace{-1pt}i\hspace{-1pt}i) and at later points in this paper. For $i\in [n]$, define
\begin{equation}\label{eq:tildeQ}
\tilde Q_{z,i}^h := \prod_{j\in [n]\setminus\{i\}} I_{z_j}^h,
\end{equation}
where again $\prod$ denotes the Cartesian product of sets. We refer to \eqref{eq:Izh} for the definition of $I_{z_j}^h$.

\begin{theorem}\label{thm:ih_ph_properties}
Let $h\in \mathfrak{H}$.
\begin{enumerate}
\item[(i)] If $v_1, v_2 \in L_h^2$, then
$
 \langle i_h v_1, i_h v_2 \rangle_{L^2(\mathbb{T}^n)} = \langle v_1, v_2 \rangle_h.
$
In particular, it follows that $i_h(L_h^2) \subset L^2(\mathbb{T}^n)$ as required, and that $i_h$ is an isometry.
\item[(i\hspace{-1pt}i)] The composition $p_h \circ i_h$ is the identity on $L_h^2$.
\item[(i\hspace{-1pt}i\hspace{-1pt}i)] If $w \in L^2(\mathbb{T}^n)$ and $x\in Q_z^h$, we have
\[
i_h p_h w(x) = h^{-n} \int_{Q_z^h} w(y) \, dy.
\]
\item[(i\hspace{-1pt}v)]  If $w_1, w_2 \in L^2(\mathbb{T}^n)$, then
\[
 \langle i_hp_h w_1, w_2 \rangle_{L^2(\mathbb{T}^n)} = \langle p_hw_1, p_h w_2\rangle_h = \langle w_1, i_hp_h w_2 \rangle_{L^2(\mathbb{T}^n)}.
\]
In particular the operator $i_h\circ p_h$ is self-adjoint on $L^2(\mathbb{T}^n)$.
\item[(v)] Condition \eqref{Pythagoras} is satisfied and thus the Pythagorean equality \eqref{P} holds.
\item[(v\hspace{-1pt}i)] If $f\in C(\mathbb{T}^n)$ then $i_hp_hf \to f$ uniformly, as $h\to 0$.
\item[(v\hspace{-1pt}i\hspace{-1pt}i)] For all $w \in L^2(\mathbb{T}^n)$, $i_h p_h w \to w$ in $L^2(\mathbb{T}^n)$ as $h \to 0$.
\item[(v\hspace{-1pt}i\hspace{-1pt}i\hspace{-1pt}i)] For all $w\in L^2(\mathbb{T}^n)$ and for all $z\in V_h$,
\begin{align*}
 &\int_{Q^h_z} (i_h p_h w)(x) \, dx = \int_{Q^h_z} w(x) \, dx \qquad \text{and} \\
 &\int_{Q^h_z} |i_h p_h w|(x) \, dx \leq \int_{Q^h_z} |w|(x) \, dx.
\end{align*}

\end{enumerate}
\end{theorem}
\begin{proof}
\begin{enumerate}
\item[(i)] A direct computation shows
\[
 \langle i_h v_1, i_h v_2 \rangle_{L^2(\mathbb{T}^n)} = \sum_{z\in V_h} \int_{Q_z^h} i_hv_1(x) i_h v_2(x) \, dx = \sum_{z\in V_h}  v_1(z) v_2(z) h^n = \langle v_1, v_2 \rangle_h.
\]

\item[(i\hspace{-1pt}i)] Let $v\in L_h^2$ and $z\in V_h$, then
\[
p_hi_hv(z) = h^{-n} \int_{Q_z^h} i_hv(x)\, dx = h^{-n} \int_{Q_z^h} v(z) \, dx = v(z).
\]

\item[(i\hspace{-1pt}i\hspace{-1pt}i)] This follows directly from $i_h p_h w(x) = p_hw(z)$ and the definition of $p_h$.

\item[(i\hspace{-1pt}v)] We compute
\begin{align*}
 \langle i_hp_h w_1, w_2 \rangle_{L^2(\mathbb{T}^n)} &= \sum_{z\in V_h} \int_{Q_z^h} \left(i_hp_hw_1\right)(x) w_2(x)\, dx \\
 &= \sum_{z\in V_h} \int_{Q_z^h} \left(h^{-n} \int_{Q_z^h} w_1(y) \, dy\right) w_2(x) \, dx\\
 &= \sum_{z\in V_h} \left(h^{-n} \int_{Q_z^h} w_1(y) \, dy\right) \left( h^{-n} \int_{Q_z^h}  w_2(x) \, dx\right) h^n \\
 &= \langle p_hw_1, p_h w_2\rangle_h,
\end{align*}
which proves the first equality. Continuing the computation to prove the second equality:
\begin{align*}
 \sum_{z\in V_h} &\int_{Q_z^h} \left(h^{-n} \int_{Q_z^h} w_1(y) \, dy\right) w_2(x) \, dx = \sum_{z\in V_h} \int_{Q_z^h} w_1(y) \left(h^{-n} \int_{Q_z^h} w_2(x) \, dx \right) \, dy\\
 &= \sum_{z\in V_h} \int_{Q_z^h} w_1(y) (i_hp_hw_2)(y)\, dy = \langle w_1, i_hp_h w_2 \rangle_{L^2(\mathbb{T}^n)}.
\end{align*}

\item[(v)] Since $i_h \circ p_h$ is self-adjoint on $L^2(\mathbb{T}^n)$ and $p_h \circ i_h$ is the identity on $L_h^2$, we have, for $u\in L_h^2$ and $w\in L^2(\mathbb{T}^n)$,
\[
 \langle i_h p_h w-w, i_h u\rangle_{L^2(\mathbb{T}^n)} = \langle w, i_h p_h i_h u - i_h u\rangle_{L^2(\mathbb{T}^n)} = 0.
\]

\item[(v\hspace{-1pt}i)] By compactness of $\mathbb{T}^n$, $f$ is uniformly continuous. Let $\varepsilon >0$, then there exists a $\delta>0$ such that, if $|x-y|<\delta$ then $|f(x)-f(y)| < \varepsilon$. Let $h< n^{-1/2}\delta$, $x\in \mathbb{T}^n$, and let $z\in V_h$ be such that $x\in Q_z^h$. If $y\in Q_z^h$, then $|x-y| < \sqrt{n} h = \delta$. Thus
\begin{align*}
\left|i_hp_hf(x)-f(x)\right| & = \left|h^{-n} \int_{Q_z^h} f(y) \, dy - f(x)\right| = \left|h^{-n} \int_{Q_z^h} \big(f(y)-f(x)\big) \, dy\right|\\
&\leq h^{-n} \int_{Q_z^h} |f(y)-f(x)| \, dy < h^{-n} \int_{Q_z^h} \varepsilon \, dy = \varepsilon.
\end{align*}

\item[(v\hspace{-1pt}i\hspace{-1pt}i)] Let $f\in C(\mathbb{T}^n)$. By the previous point $i_hp_h f \to f$ uniformly as $h\to 0$. Since $i_h \circ p_h$ is self-adjoint, we have, for $w \in L^2(\mathbb{T}^n)$,
\[
	\langle i_hp_h w, f \rangle_{L^2} = \langle w, i_hp_h f \rangle_{L^2} \to \langle w, f \rangle_{L^2}
	\quad\text{as}\quad h \to 0.
\]
By the Cauchy--Schwarz inequality, applied on a single cube $Q_z^h$, we get, for all $\tilde w\in L^2(\mathbb{T}^n)$,
\[
 \left(\int_{Q_z^h} \tilde w(y) \, dy\right)^2 \leq \left(\int_{Q_z^h} \tilde w(y) \, dy\right) \left(\int_{Q_z^h} 1^2 \, dy\right) = \left(\int_{Q_z^h} \tilde w(y) \, dy\right) h^n.
\]
If $\|\tilde w\|_{L^2(\mathbb{T}^n)} = 1$, then
\begin{align*}
\|i_hp_h \tilde w\|_{L^2(\mathbb{T}^n)}^2 &= \sum_{z\in V_h} \int_{Q_z^h} \left(h^{-n} \int_{Q_z^h} \tilde w(y) \, dy\right)^2 \, dx = \sum_{z\in V_h} h^{-n}  \left(\int_{Q_z^h} \tilde w(y) \, dy\right)^2\\
&\leq \sum_{z\in V_h} \int_{Q_z^h} \tilde w(y) \, dy = \int_{\mathbb{T}^n} \tilde w(y) \, dy = 1.
\end{align*}
Hence the operator norm of $i_h \circ p_h$ satisfies
\[
 \|i_h \circ p_h\|_{op} := \sup\left\{\|i_hp_h \tilde w\|_{L^2(\mathbb{T}^n)} \Bigm| \tilde w\in L^2(\mathbb{T}^n), \|\tilde w\|_{L^2(\mathbb{T}^n)} = 1\right\} \leq 1.
\]
Let $\tilde w\in L^2(\mathbb{T}^n)$. Since $C(\mathbb{T}^n)$ is dense in $L^2(\mathbb{T}^n)$, there exists a sequence $(f_n)$ in $C(\mathbb{T}^n)$ which converges to $\tilde w$ in $L^2(\mathbb{T}^n)$ as $n\to\infty$. Combining the Cauchy--Schwarz inequality with the bound on the operator norm, we find
\begin{align*}
 \langle i_hp_h w, \tilde w\rangle_{L^2(\mathbb{T}^n)} &= \langle i_hp_h w, \tilde f_n\rangle_{L^2(\mathbb{T}^n)} + \langle i_hp_h w, \tilde w-f_n\rangle_{L^2(\mathbb{T}^n)}\\
 &\leq \langle i_hp_h w, \tilde f_n\rangle_{L^2(\mathbb{T}^n)} + \|i_hp_hw\|_{L^2(\mathbb{T}^n)}\ \|\tilde w - f_n\|_{L^2(\mathbb{T}^n)}\\
 &\leq \langle i_hp_h w, \tilde f_n\rangle_{L^2(\mathbb{T}^n)} + \|w\|_{L^2(\mathbb{T}^n)}\ \|\tilde w - f_n\|_{L^2(\mathbb{T}^n)}.
\end{align*}
Taking first $h\to 0$ and then $n\to \infty$, the right-hand side vanishes and thus $i_h p_h w \rightharpoonup w$ (i.e., weakly) in $L^2(\mathbb{T}^n)$. By lower semicontinuity of the $L^2(\mathbb{T}^n)$ norm with respect to weak $L^2(\mathbb{T}^n)$ convergence and again the estimate on the operator norm, we have that
\[
	\|w\|_{L^2(\mathbb{T}^n)} \leq \liminf_{h \to 0} \|i_h p_h w\|_{L^2(\mathbb{T}^n)} \leq \liminf_{h \to 0} \|w\|_{L^2(\mathbb{T}^n)} = \|w\|_{L^2(\mathbb{T}^n)}.
\]
Therefore $\|i_h p_h w\|_{L^2(\mathbb{T}^n)} \to \|w\|_{L^2(\mathbb{T}^n)}$. Together with the weak convergence $p_h w \rightharpoonup w$ in $L^2(\mathbb{T}^n)$, this implies that $i_h p_h w \to w$ (i.e., strongly) in $L^2(\mathbb{T}^n)$.

\item[(v\hspace{-1pt}i\hspace{-1pt}i\hspace{-1pt}i)] For all $j\in [n]$, define the operator $q_h^j: L^2(\mathbb{T}^n) \to L^2(\mathbb{T}^n)$ by, for all $w\in L^2(\mathbb{T}^n)$ and all $x\in Q_z^h$,
\[
q_h^j w(x) := h^{-1} \int_{I_{z_j}^h} w(\tilde x_j, x_j) \, dx_j =  h^{-1} \int_{I_{z_j}^h} w(x) \, dx_j.
\]
Assume $\{\alpha_1, \ldots, \alpha_n\} = [n]$. Let $i, j \in [n]$ with $i\neq j$ and assume that $\alpha_{n-1}=i$ and $\alpha_n=j$. Let $x\in Q_z^h$, then
\begin{align*}
q_h^{\alpha_{n-1}} q_h^{\alpha_n} w(x) = q_h^i q_h ^j w(x) &= h^{-1} \int_{I_{z_i}^h} q_h^j w(x) \, dx_i = h^{-2} \int_{I_{z_i}^h} \int_{I_{z_j}^h} w(x) \, dx_i dx_j.
\end{align*}
Repeating this argument for $\alpha_1, \ldots, \alpha_{n-2}$, we find that
\begin{equation}\label{eq:qhalpha}
q_h^{\alpha_1} \ldots q_h^{\alpha_n} w = i_h p_h w.
\end{equation}

Furthermore, we compute for all $j\in [n]$ and all $z\in V_h$,
\begin{align*}
\int_{Q^h_z} |q_h^j w (x)| \, dx &= \int_{Q_z^h} \left| h^{-1} \int_{I_{z_j}^h} w(\tilde x_j, y) \, dy\right| \, dx\\
&= \int_{I_{z_j}^h} \int_{\tilde Q_{z,j}^h} \left| h^{-1} \int_{I_{z_j}^h} w(\tilde x_j, y) \, dy\right| \, dx_j \, d\tilde x_j\\
&= \int_{\tilde Q_{z,j}^h} \left|\int_{I_{z_j}^h} w(\tilde x_j, y) \, dy\right| \, d\tilde x_j \leq  \int_{\tilde Q_{z,j}^h} \int_{I_{z_j}^h} \left| w(\tilde x_j, y)\right| \, dy \, d\tilde x_j\\
&= \int_{Q^h_z} |w(x)| \, dx.
\end{align*}
In the absence of the absolute value function in the integrals above, the inequality becomes an equality. Repeatedly applying this (in)equality according to \eqref{eq:qhalpha} gives the desired result.
\end{enumerate}
\end{proof}

As mentioned before, in Section~\ref{sec:extension} we will introduce a different inner product on $\mathcal{V}_h$ which will give rise to the Hilbert space $\overline{L}_h^2$. We postpone the introductions of a corresponding new embedding operator $I_h$ (Section~\ref{sec:extension}) and projection operator $P_h$ (Section~\ref{sec:projection}) to Section~\ref{sec:ACflow}, where they will be needed for our proofs for the Allen--Cahn gradient flows.

\section{Continuum limit of a discrete total variation flow}\label{sec:TVflow}

In this section we consider the $L^2_h$-gradient flow of the discrete total variation $\varphi_{TV}^h$ and want to claim convergence to anisotropic total variation flow, i.e., the $L^2(\mathbb{T}^n)$-gradient flow of the continuum anisotropic total variation $\varphi_{TV}$. First we will define these functionals in Section~\ref{sec:thefunctionalsTV}

\subsection{Discrete and continuum total variation functionals}\label{sec:thefunctionalsTV}

Let $h\in \mathfrak{H}$. The discrete (anisotropic) graph total variation functional $\varphi^h_{TV}: L_h^2 \to \mathbb{R}$ is defined by
\[
\varphi^h_{TV}(u) := \frac12 \sum_{z; \tilde z\sim z} h^{n-1} \left|u(z)-u(\tilde z)\right|,
\]
where we have used the shorthand notation $\sum_{z; \tilde z\sim z}:= \sum_{z\in V_h} \sum_{\tilde z \in V_h: \tilde z\sim z}$. Since $L_h^2$ is a Hilbert space and thus closed, we note that $\overline{\mathcal{D}\left(\varphi_{TV}^h\right)} = L_h^2$.

The anisotropic total variation functional for functions $u: \mathbb{T}^n \to \mathbb{R}$ is
\begin{equation}\label{eq:TV}
 \int_{\mathbb{T}^n} |Du|_{l^1} := \sup \left\{\int_{\mathbb{T}^n} u(x) \operatorname{div} g(x) \, dx \Bigm| g\in C^1(\mathbb{T}^n; \mathbb{R}^n), \ \forall x\in \mathbb{T}^n \ |g(x)|_\infty \leq 1\right\},
\end{equation}
where $|g(x)|_\infty := \max_{1\leq i \leq n} |g_i(x)|$. If $u$ is regular enough to admit a weak partial derivative in each direction, the notation $\int_{\mathbb{T}^n} |Du(x)|_{l^1} \, dx$ is compatible with the earlier definition of $|\cdot|_{l^1}$.

If $g=(g_1, \ldots, g_n)$, the condition $|g(x)|_\infty \leq 1$ is equivalent to, for all $i\in [n]$, $|g_i(x)|\leq 1$, hence
\begin{equation}\label{eq:sumof1DTV}
 \sup_g \int_{\mathbb{T}^n} u(x) \operatorname{div} g(x) \, dx = \sup_g \sum_{i=1}^n \int_{\mathbb{T}^n} u(x) \frac{\partial g_i}{\partial x_i}(x) \, dx = \sum_{i=1}^n \sup_{g_i} \int_{\mathbb{T}^n} u(x) \frac{\partial g_i}{\partial x_i}(x) \, dx,
\end{equation}
where the supremum on the right-hand side is taken over all $g_i \in C^1(\mathbb{T}^n; \mathbb{R})$ that satisfy, for all $x\in \mathbb{T}^n$, $|g_i(x)| \leq 1$. (We allow ourselves the slight misuse of notation, to stress that the set of admissible functions for each supremum depends on $i$.)

We denote the space of functions of bounded variation, i.e., the space of all $u \in L^1(\mathbb{T}^n)$ (i.e., Lebesgue integrable functions on $\mathbb{T}^n$) for which $\int_{\mathbb{T}^n} |Du|_{l^1} < +\infty$, by $BV(\mathbb{T}^n)$. It can be shown that $\int_{\mathbb{T}^n} |Du|_{l^1}$ and the isotropic total variation (see for example \cite{Giusti}) are equivalent seminorms and thus $BV(\mathbb{T}^n)$ as defined above agrees with the standard definition based on the isotropic total variation.
It is worth noting that $BV(\mathbb{T}^n)$ is not the same as $BV(\operatorname{int}\Omega_h)$, since the anisotropic total variation on $\mathbb{T}^n$ also measures variations on the boundary of $\Omega_h$ (taking into account the periodic boundary conditions), whereas the anisotropic total variation on $\operatorname{int}\Omega_h$ does not.

An important property of the anisotropic total variation is that, for $u\in L_h^2$,
\begin{equation}\label{eq:embeddedTV}
 \int_{\mathbb{T}^n} |D i_h u|_{l^1} = \varphi_{TV}^h(u).
\end{equation}
We provide a proof of \eqref{eq:embeddedTV} in Appendix~\ref{app:TVforpiecewiseconstant}.

We are now ready to define the continuum counterpart of the discrete anisotropic total variation functional, $\varphi_{TV}: L^2(\mathbb{T}^n) \to \mathbb{R}\cup\{+\infty\}$, by
\begin{equation}\label{eq:varphiTV}
  \varphi_{TV}(u) := \begin{cases}
    \int_{\mathbb{T}^n} |D u|_{l^1}, &\text{if } u\in BV(\mathbb{T}^n)\cap L^2(\mathbb{T}^n), \\
    +\infty, &\text{if } u\in L^2(\mathbb{T}^n)\setminus BV(\mathbb{T}^n).
  \end{cases}
\end{equation}
Because $BV(\mathbb{T}^n)\cap L^2(\mathbb{T}^n)$ is dense in $L^2(\mathbb{T}^n)$ (as follows from the fact that the space of infinitely often continuously differentiable real-valued functions on $\mathbb{T}^n$, $C(\mathbb{T}^n)$, is dense in $L^2(\mathbb{T}^n)$ \cite[Corollary 2.30]{AF}), we have $\overline{\mathcal{D}\left(\varphi_{TV}\right)} = L^2(\mathbb{T}^n)$.

We note that, for all $h\in \mathfrak{H}$, $\varphi_{TV}^h$ is convex (geodesically $0$-convex), as is $\varphi_{TV}$. Furthermore, $\varphi_{TV}^h$ is continuous and $\varphi_{TV}$ is lower semicontinuous with respect to $L^1(\mathbb{T}^n)$ convergence \cite{Giusti} and thus also with respect to $L^2(\mathbb{T}^n)$ convergence. Hence $\varphi_{TV}^h$ and $\varphi_{TV}$ satisfy condition ($0$-AGS). Moreover, equation \eqref{eq:embeddedTV} implies that condition \eqref{L} is satisfied with $\Phi_h=\varphi_{TV}^h$ and $\Phi=\varphi_{TV}$.

Formally (ignoring issues when $\frac{\partial u}{\partial x_j} = 0$), the $L^2(\mathbb{T}^n)$-gradient flow of $\varphi_{TV}$ is a solution to
\[
	u_t = \operatorname{div} \left( \sum^n_{j=1} \frac{\partial }{\partial x_j} \Bigm/ \left|\frac{\partial u}{\partial x_j}\right| \right).
\]

\subsection{Convergence of flows}
 We will establish the following result.

\begin{theorem} \label{TV1}
Let $h\in \mathfrak{H}$ and let $u^h$ be the $L^2_h$-gradient flow of $\varphi_{TV}^h$ with initial data $u^h_0\in L^2_h$.
 Let $u$ be the   $L^2(\mathbb{T}^n)$-gradient flow of $\varphi_{TV}$   with initial data $u_0 \in L^2(\mathbb{T}^n)$.
 Then,   for all $t\geq 0$,
\[
	\| i_h u^h(t) -u(t) \|_{L^2(\mathbb{T}^n)} \leq \| i_h u^h_0 -u_0 \|_{L^2(\mathbb{T}^n)}.
\]
\end{theorem}
  We will apply our abstract results (Theorem~\ref{AB1} and Corollary~\ref{CorA}) to $\Phi=\varphi_{TV}$ on $M=L^2(\mathbb{T}^n)$ and $\Phi_h = \varphi_{TV}^h$ on $M_h=L^2_h$. Both these functionals are convex and nonnegative. Moreover, $\varphi_{TV}^h$ is continuous. The lower semicontinuity of $\varphi_{TV}$ is less obvious but it is standard \cite{Giusti}\cite[Theorem 1 in Section 5.2]{EvansGariepy}, \cite[Appendix A]{CvGO11} (note that these results survive the move to the $L^2(\mathbb{T}^n)$ topology). Hence $\varphi_{TV}^h$ and $\varphi_{TV}$ satisfy condition ($0$-AGS). Moreover, $L_h^2=\mathcal{D}(\varphi_{TV}^h)$ and by density of smooth functions in $L^2(\mathbb{T}^n)$ we have $L^2(\mathbb{T}^n)=\overline{\mathcal{D}(\varphi_{TV})}.$ By Theorem~\ref{thm:ih_ph_properties} we know that \eqref{P} is satisfied and equation \eqref{eq:embeddedTV} showed us that \eqref{L} also holds. In order to apply Theorem~\ref{AB1} and Corollary~\ref{CorA}, it remains to show that condition \eqref{U1} is satisfied as well. This is the content of the following lemma.

\begin{lemma} \label{ATV}
Let $h\in \mathfrak{H}$. For all $w \in L^2(\mathbb{T}^n)$,
\[
\varphi_{TV}^h (p_h w) \leq \varphi_{TV}(w).
\]
\end{lemma}

\begin{proof}
If $w\in L^2(\mathbb{T}^n) \setminus BV(\mathbb{T}^n)$, the inequality is trivially true. Now assume that $w\in BV(\mathbb{T}^n)$.

By \eqref{eq:embeddedTV}, \eqref{eq:decomposition}, and Corollary~\ref{cor:pwconstantpartialg}, we have that
\begin{align*}
\varphi_{TV}^h (p_h w) &= \int_{\mathbb{T}^n} |D i_h p_h w|_{l^1} =  \sum_{i=1}^n \int_{\mathbb{T}^n} |D_{x_i} i_h p_h w| \\
&= \sum_{i=1}^n \sup_{g \in H_{h,i}} \int_{\mathbb{T}^n} i_h p_h w(x) \frac{\partial g}{\partial x_i}(x) \, dx,
\end{align*}
where we recall the definition of the sets $H_{h,i}$ from \eqref{eq:Hhi}. In particular, if $g\in H_{h,i}$, then $\frac{\partial g}{\partial x_i}$ is constant a.e. on each set $\operatorname{int} Q^h_z$. Thus, if we denote its value on $\operatorname{int} Q^h_z$ by $c_z\in \mathbb{R}$, then
\begin{align}
 \int_{\mathbb{T}^n} i_h p_h w(x) \frac{\partial g}{\partial x_i}(x) \, dx &= \sum_{z\in V_h} \int_{Q^h_z} i_h p_h w(x) \frac{\partial g}{\partial x_i}(x) \, dx = \sum_{z\in V_h} c_z \int_{Q^h_z} i_h p_h w(x) \, dx\\
 &= \sum_{z\in V_h} c_z \int_{Q^h_z} w(x) \, dx = \sum_{z\in V_h} \int_{Q^h_z} w(x) \frac{\partial g}{\partial x_i}(x) \, dx\\
 &= \int_{\mathbb{T}^n} w(x) \frac{\partial g}{\partial x_i}(x) \, dx,
\end{align}
where we used Theorem~\ref{thm:ih_ph_properties} (v\hspace{-1pt}i\hspace{-1pt}i\hspace{-1pt}i) to establish the third equality above.

We deduce that
\[
 \sup_{g \in H_{h,i}} \int_{\mathbb{T}^n} i_h p_h w(x) \frac{\partial g}{\partial x_i}(x) \, dx = \sup_{g \in H_{h,i}} \int_{\mathbb{T}^n} w(x) \frac{\partial g}{\partial x_i}(x) \, dx \leq \int_{\mathbb{T}^n} |D_{x_i} w|,
\]
where the inequality follows from the fact that each set $H_{h,i}$ is a subset of the corresponding set of admissible functions in the supremum in \eqref{eq:TVxi}. Applying \eqref{eq:decomposition} once more, we conclude from \eqref{eq:varphiTV} that
\[
\varphi_{TV}^h (p_h w) \leq \sum_{i=1}^n \int_{\mathbb{T}^n} |D_{x_i} w| = \int_{\mathbb{T}^n} |Dw|_{l^1} = \varphi_{TV}(w).
\]
\end{proof}

\begin{proof}[Proof of Theorem \ref{TV1}]
By equation \eqref{eq:embeddedTV}, Lemma \ref{ATV}, and Theorem~\ref{thm:ih_ph_properties}, we know that $\Phi=\varphi_{TV}$ and $\Phi_h=\varphi_{TV}^h$ satisfy conditions \eqref{L}, \eqref{U1}, and \eqref{P}. Moreover, by our discussion in Section~\ref{sec:thefunctionalsTV} they also satisfy condition ($0$-AGS). We are able to apply Theorem~\ref{AB1} and Corollary~\ref{CorA} to get the desired result.
\end{proof}

\begin{remark}
In \cite{LMM} it is proved that $u^h$ is the gradient flow of $\varphi_{TV}$ in $L^2(\mathbb{T}^n)$ by calculating the minimal section of the flow. Our result gives a direct proof that $u^h$ is indeed the gradient flow of $\varphi_{TV}$ in $L^2(\mathbb{T}^n)$.
\end{remark}

\subsection{Convergence of energy}
We are interested in the convergence of the `energy' $\varphi_{TV}^h(u^h)$ to $\varphi_{TV}(u)$.

\begin{theorem} \label{TV2}
Let $h\in \mathfrak{H}$. Let $u^h$ and $u$ be as in Theorem \ref{TV1} and let $T>0$. Then
\[
\left| \int_0^T \left(\varphi_{TV}^h(u^h(s))-\varphi_{TV} (u(s)) \right) ds \right|\leq \left(\|i_h u^h_0\|_{L^2(\mathbb{T}^n)}+\|u_0\|_{L^2(\mathbb{T}^n)} \right)\|i_h u^h_0-u_0\|_{L^2(\mathbb{T}^n)}.
\]
\end{theorem}
\begin{proof} By definition $\dot{u}^h \in -\partial \varphi_{TV}^h(u^h)$ and $\dot{u} \in -\partial \varphi_{TV}(u).$ Since $\varphi_{TV}^h$ and $\varphi_{TV}$ are positively one-homogeneous, we see that\footnote{Assume $\varphi$ is a convex function on a (real or complex) vector space $V$. This function is positively one-homogeneous if, for all $u\in V$ and for all $r>0$, $\varphi(ru) = r \varphi(u)$. If $v\in \partial \varphi(u)$, then, for all $w\in V$, $\varphi(w)-\varphi(u) \geq \langle w-u, v\rangle$. Choosing $w=0$ and $w=2u$ leads to $\varphi(u) = \langle u, v \rangle.$}, for all $v\in \partial \varphi_{TV}^h(u^h)$ and for all $w\in \partial \varphi_{TV}(u)$ we have the Euler identity (see, for example, \cite[Lemma 1.7]{ACM}) \[\langle u^h, v \rangle_h=\varphi_{TV}^h(u^h), \qquad \langle u, w \rangle_{L^2(\mathbb{T}^n)}=\varphi_{TV}(u).\]
Choosing $v=\dot{u}^h$ and $w=\dot{u}$ yields \[\varphi_{TV}^h(u^h)=\frac{1}{2} \frac{d}{dt} \|u^h\|^2_h, \qquad \varphi_{TV}(u)=\frac{1}{2} \frac{d}{dt}\|u\|_{L^2(\mathbb{T}^n)}^2.\] These equalities hold almost everywhere on $(0,\infty)$. From this observation it follows that
\begin{align} \label{I1}
\int_0^T \left( \varphi_{TV}^h(u^h(s))-\varphi_{TV}(u(s))\right)ds=&\frac{1}{2} \left(\|i_h u^h(T)\|_{L^2(\mathbb{T}^n)}^2-\|u(T)\|_{L^2(\mathbb{T}^n)}^2 \right) \notag\\
&-\frac{1}{2} \left(\|i_h u^h_0\|_{L^2(\mathbb{T}^n)}^2-\|u_0\|_{L^2(\mathbb{T}^n)}^2 \right),
\end{align}
where we have also used the isometry property of $i_h$ (Theorem~\ref{thm:ih_ph_properties}).

By the contraction property from Theorem~\ref{TV1} we see that
\begin{align*}
&\|i_h u^h(T)-0\|_{L^2(\mathbb{T}^n)} \leq \|i_h u^h_0-0 \|_{L^2(\mathbb{T}^n)} , \qquad \|u(T)-0\|_{L^2(\mathbb{T}^n)} \leq \|u_0-0\|_{L^2(\mathbb{T}^n)} \text{ and}\\
& \|i_h u^h(T) - u(T)\|_{L^2(\mathbb{T}^n)} \leq \|i_h u_0^h - u_0\|_{L^2(\mathbb{T}^n)}.
\end{align*}
For the first two inequalities we used that the constant solutions $u^h=0$ and $u=0$ are gradient flows of $\varphi_{TV}^h$ starting from $0$ and of $\varphi_{TV}$ starting from $0$, respectively (Lemma~\ref{lem:constantsolution}).
Since, for all $a, b\in \mathbb{R}$, $a^2-b^2=(a+b)(a-b)$, we find, using the inequalities above and the reverse triangle inequality, that
\begin{align*}
&\left|\|i_h u^h(T)\|_{L^2(\mathbb{T}^n)}^2-\|u(T)\|_{L^2(\mathbb{T}^n)}^2 \right| \\
&= \left|\|i_h u^h(T)\|_{L^2(\mathbb{T}^n)}+\|u(T)\|_{L^2(\mathbb{T}^n)}\right| \ \left|\|i_h u^h(T)\|_{L^2(\mathbb{T}^n)}-\|u(T)\|_{L^2(\mathbb{T}^n)}\right|\\
&\leq  \left|\|i_h u^h(T)\|_{L^2(\mathbb{T}^n)}+\|u(T)\|_{L^2(\mathbb{T}^n)}\right| \ \left\|i_h u^h(T) - u(T)\right\|_{L^2(\mathbb{T}^n)}\\
&\leq \left|\|i_h u_0^h\|_{L^2(\mathbb{T}^n)}+\|u_0\|_{L^2(\mathbb{T}^n)}\right| \ \left\|i_h u_0^h - u_0\right\|_{L^2(\mathbb{T}^n)} \qquad\text{and}\\
&\left|\|i_h u_0^h\|_{L^2(\mathbb{T}^n)}^2-\|u_0\|_{L^2(\mathbb{T}^n)}^2 \right| \\
&= \left|\|i_h u_0^h\|_{L^2(\mathbb{T}^n)}+\|u_0\|_{L^2(\mathbb{T}^n)}\right| \ \left|\|i_h u_0^h\|_{L^2(\mathbb{T}^n)}-\|u_0\|_{L^2(\mathbb{T}^n)}\right|\\
&\leq \left|\|i_h u_0^h\|_{L^2(\mathbb{T}^n)}+\|u_0\|_{L^2(\mathbb{T}^n)}\right| \ \left\|i_h u_0^h - u_0\right\|_{L^2(\mathbb{T}^n)}.
\end{align*}
By \eqref{I1} we get the desired estimate.
\end{proof}

\section{Continuum limit of a discrete Allen--Cahn flow} \label{sec:ACflow}

In this section we will study the convergence of discrete Allen--Cahn gradient flows to a continuum Allen--Cahn gradient flow. These flows are determined, in the sense of Definition~\ref{EVI}, by the discrete and continuum Allen--Cahn functionals, $\varphi_{AC}^h$ and $\varphi_{AC}$, respectively. We start by defining these functionals in Section~\ref{sec:thefunctionalsAC}.

\subsection{Discrete and continuum Allen--Cahn functionals}\label{sec:thefunctionalsAC}

Let $h\in\mathfrak{H}$. The discrete Allen--Cahn functional consists of two terms, the Dirichlet energy $\varphi_D^h: \mathcal{V}_h \to \mathbb{R}$ and the double-well potential energy $\varphi_W^h: \mathcal{V}_h \to \mathbb{R}$, which we define separately:
\begin{align*}
 \varphi_D^h(u) &:= \frac14 \sum_{z; \tilde z \sim z} h^{n-2} \left(u(z)-u(\tilde z)\right)^2,\\
 \varphi_W^h(u) &:= \sum_{z\in V_h} h^n W\left(u(z)\right).
\end{align*}
Here $W$ is a double-well potential with wells of equal depth; in this paper, we shall fix $W(x)=\alpha(x^2-1)^2/4$ for a given $\alpha>0$. Now we define the discrete Allen--Cahn functional $\varphi_{AC}^h: \mathcal{V}_h \to \mathbb{R}$ as
\[
 \varphi_{AC}^h := \varphi_D^h + \varphi_W^h.
\]
We note here that we have chosen $\mathcal{V}_h$ as the domain of $\varphi_D^h$, $\varphi_W^h$, and $\varphi_{AC}^h$. At various times we will want to emphasize a specific inner product structure (and induced topology) on the domain, in which case we will write $L^2_h$ or $\overline{L}^2_h$ for the domain. The latter of these two Hilbert spaces will be defined in Section~\ref{sec:extension}.

The continuum versions of the Dirichlet energy, $\varphi_D: L^2(\mathbb{T}^n) \to \mathbb{R} \cup \{+\infty\}$, and the double-well energy, $\varphi_W: L^2(\mathbb{T}^n) \to \mathbb{R} \cup \{+\infty\}$, are given by
\begin{align*}
	&\varphi_D (u) := \begin{cases} \frac{1}{2} \int_{\mathbb{T}^n} |\nabla u|^2 \, dx, & \text{if } u\in H^1(\mathbb{T}^n),\\ +\infty, & \text{otherwise}, \end{cases} \qquad \text{and} \\
	\hspace{0.2cm}
	&\varphi_W (u) := \begin{cases} \int_{\mathbb{T}^n} W(u) \, dx, &\text{if } W\circ u\in L^1(\mathbb{T}^n),\\ +\infty, & \text{otherwise}. \end{cases}
\end{align*}
Here $|\nabla u|$ denotes the Euclidean norm of the (weak) gradient of $u$.
We remind ourselves that $H^m(\mathbb{T}^n)$ denotes the space of functions in $L^2(\mathbb{T}^n)$ that are $m\in\mathbb{N}$ times weakly differentiable with each of their weak derivatives also being in $L^2(\mathbb{T}^n)$\footnote{And similarly if $\mathbb{T}^n$ is replaced by an open subset of itself or of $\mathbb{R}^n$. Recall that $w\in L^2(\mathbb{T}^n)$ is the weak derivative of $u\in L^2(\mathbb{T}^n)$ with respect to $\mathbb{T}^n$-coordinate $x_i$, if, for all $\varphi \in C^\infty(\mathbb{T}^n)$, $\langle w, \varphi\rangle_{L^2(\mathbb{T}^n)} = \langle u, \frac{\partial}{\partial x_i} \varphi \rangle_{L^2(\mathbb{T}^n)}$.}. It is equipped with the (squared) norm $\|w\|_{H^m(\mathbb{T})}^2 := \|w\|_{L^2(\mathbb{T}^n)}^2 + \sum_{l=1}^m \|w^{(l)}\|_{L^2(\mathbb{T}^n)}^2$, where $w^{(l)}$ denotes the $l^{\text{th}}$ derivative of $w$.

We note that $\varphi_W$ is $L^2(\mathbb{T}^n)$-coercive since, by Jensen's inequality \cite[Appendix B Theorem 2]{Evans},
\[
\int_{\mathbb{T}^n} W(u) \, dx = \int_{\mathbb{T}^n} (u^2(x)-1)^2 \, dx \geq \left(\int_{\mathbb{T}^n} (u^2(x)-1) \, dx\right)^2 = \left( \|u\|_{L^2(\mathbb{T}^n)}^2 - 1 \right)^2.
\]

The continuum Allen--Cahn functional, $\varphi_{AC}: L^2(\mathbb{T}^n) \to \mathbb{R} \cup \{+\infty\}$ is then defined to be
\[
 	\varphi_{AC}(u) := \varphi_D (u) + \varphi_W (u).
\]
Neither $\varphi_{AC}^h$ nor $\varphi_{AC}$ are convex, but they are geodesically $(-\alpha)$-convex, uniformly in $h$ (for $\varphi_{AC}^h$). This follows since $\varphi_D^h$ and $\varphi_D$ are convex and since $\left(W(x)-\frac{(-\alpha)}2 x^2\right)'' = \alpha (3x^2-1) + \alpha \geq 0$, so that $W$ is geodesically $(-\alpha)$-convex, and thus so are $\varphi_W^h$ and $\varphi_W$.

The functional $\varphi_{AC}^h$ is continuous, if $\mathcal{V}_h$ is equipped with any norm-induced topology (as those topologies are all equivalent since $\mathcal{V}_h$ is finite-dimensional). Moreover, since the integrand of $\varphi_{AC}(u)$, $\frac12 |\nabla u|^2 + W(u)$, is smooth as a function of $(\nabla u, u)$ and convex in $\nabla u$, it is lower semicontinuous with respect to weak $H^1(\mathbb{T}^n)$ convergence \cite[\S 8.2 Theorem 1]{Evans} and thus, by the compact Rellich--Kondrachov embedding \cite[Theorem 6.3]{AF} also with respect to (strong) $L^2(\mathbb{T}^n)$ convergence. Hence $\varphi_{AC}^h$ and $\varphi_{AC}$ satisfy the ($-\alpha$-AGS) condition.

The Allen--Cahn functional (discrete or continuum) also goes by various other names in the literature, such as Ginzburg--Landau functional \cite{BertozziFlenner} or Modica--Mortola functional \cite{MM}.

\subsection{Additional setup}\label{sec:additionalsetup}

We are interested in discrete Allen--Cahn flows, i.e., solutions of $\dot{u} \in -\partial\varphi^h_{AC}(u)$, and in particular in their convergence to a continuum Allen--Cahn flow, i.e., a solution of $\dot{u} \in -\partial\varphi_{AC}(u)$ (Theorem~\ref{CAC}).
An explicit form of the equation for the Allen--Cahn flow is
\begin{equation} \label{AC}
	\dot{u} = \Delta u - W'(u) \quad\text{in } \mathbb{T}^n \times (0,\infty),
\end{equation}
which is often called the Allen--Cahn equation \cite{AC}. As usual $\Delta$ denotes the (continuum) Laplacian $\sum_{i=}^n \frac{\partial^2}{\partial x_i^2}$ where $x_i$ are the coordinates on $\mathbb{T}^n$. This explicit form of the gradient flow equation can be computed directly since the subdifferential in \eqref{eq:lambdaconvexsubdiff} is a singleton. Similarly an explicit form of the equation for the discrete Allen--Cahn flow can be computed:
\begin{equation}\label{eq:graphAC}
\dot u = \Delta_h u - W'(u) \quad\text{in } V_h \times (0, \infty),
\end{equation}
where
\begin{equation}\label{eq:graphLaplacian}
(\Delta _h u)(z) := h^{-2} \sum_{\tilde z\in V_h: \tilde z \sim z} \left(u(\tilde z)-u(z)\right)
\end{equation}
defines the graph Laplacian \cite{Chung,VonLuxburg} for $u\in L^2_h$ and $z\in V_h$. This discrete flow equation has been studied in more detail in \cite{vGGOB,BvG20,BvG21,BvGL21}. Details of the derivation in the one-dimensional ($n=1$) case follow in Lemma~\ref{IS}.

It will also be useful to define a graph gradient $\nabla_h: V_h \times V_h \to \mathbb{R}$; for $u\in L^2_h$ and $z, \tilde z \in V_h$,
\[
\nabla_h u (z, \tilde z) = \begin{cases}
h^{-1} \bigl( u(\tilde z) - u(z) \bigr), &\text{if } \tilde z\sim z,\\
0, &\text{otherwise}.
\end{cases}
\]
We note that $\nabla_h u$ is defined on $V_h \times V_h$. For such functions $\chi, \phi: V_h \times V_h \to \mathbb{R}$ we introduce the inner product
\[
\langle \chi, \phi \rangle_{h,h} := \frac12 \sum_{z; \tilde z \sim z} h^n \chi(z, \tilde z) \phi(z, \tilde z)
\]
and corresponding norm $\|\chi\|_{h,h} := \sqrt{\langle \chi, \chi \rangle_{h,h}}$. Then we have, for $u, v \in \mathcal{V}_h$,
\begin{align}
\langle -\Delta_h u, v \rangle_h &= \sum_{z; \tilde z \sim z} h^{n-2} \left(u(z)-u(\tilde z)\right) v(z) = \frac12 \sum_{z; \tilde z \sim z} h^{n-2} \left(u(z)-u(\tilde z)\right) \left(v(z)-v(\tilde z)\right)\notag\\
&= \langle \nabla_h u, \nabla_h v \rangle_{h,h} = \langle  u, -\Delta_hv \rangle_h.\label{eq:DeltaNabla}
\end{align}
and $\displaystyle \varphi_D^h(u) = \frac12 \|\nabla_h u\|_{h,h}^2$.

\begin{remark}
Two different sign conventions for the (graph) Laplacian appear in the literature. Our choice in \eqref{eq:graphLaplacian} is in line with the typical definition in the partial differential equations' literature which has the Laplacian be negative semidefinite. This differs by an overall minus sign from the common definition in the spectral graph literature, which has the graph Laplacian be positive semidefinite.
\end{remark}

\begin{remark}
It is also interesting to compare how the choices we have made in this paper regarding the scaling with $h$, correspond to the scaling choices in \cite[Section 2]{vGGOB} which are common in the literature. In \cite{vGGOB} the scaling of the inner products, operators, and summands in the functionals, depends on the edge weights $\omega$, the node degrees $d\sim \omega$, and two parameters, $r$ and $q$, as follows:
\begin{align*}
\langle \cdot, \cdot \rangle_h \sim d^r \sim \omega^r, \quad
&\langle \cdot, \cdot \rangle_{h,h} \sim \omega^{2q-1},\\
\nabla_h \sim \omega^{1-q},\quad
&\Delta_h \sim \omega^{1-r}, \\
\text{summands in }\varphi_D^h \sim \omega, \quad
&\text{summands in }\varphi_{TV}^h \sim \omega^q.
\end{align*}
We note that for the functionals $\varphi_D^h$ and $\varphi_{TV}^h$ we consider the scaling of the summands only, i.e., we do not consider the fact that the number of summands is $|V_h|=h^{-n}$. According to our choices in this paper, we require the summands of $\varphi_D^h$ to be proportional to $h^{n-2}$, which suggests the choice $\omega=h^{n-2}$ for the edge weights between neighbouring nodes. Comparisons for the other objects lead to
\[
\omega^r = h^n, \quad \omega^{2q-1} = h^n, \quad \omega^{1-q}=h^{-1}, \quad \omega^{1-r} = h^{-2}, \quad \text{and} \quad \omega^q=h^{n-1}.
\]
It follows that, when $n\neq 2$,
\[
q = \frac{n-1}{n-2} \quad \text{and} \quad r=2q-1=\frac{n}{n-2}.
\]
This implies that $q>1$ and $r>1$, which are not typically choices that are considered. When $n=2$, we have $\omega=1$, which cannot be made compatible with the other requirements.
\end{remark}

\subsection{Restriction to one-dimensional case ($n=1$)}

To avoid additional technical complexity, we consider the convergence problem only for the one-dimensional case, i.e., $n=1$. In this case the vertex set is $V_h = \{0, h, 2h, \ldots, 1-h\}$. We remind ourselves that $h\in \mathfrak{H}$ and thus $|V_h| = h^{-1}$ is an integer. It will be useful to write $\left[\frac1h-1\right]_0 := \{0\}\cup \left[\frac1h-1\right] = \{0, 1, \ldots, \frac1h-1\}$, so that to each $k\in \left[\frac1h-1\right]_0$ corresponds one and only one $kh \in V_h$.

Remember that $\mathcal{V}_h$ denotes the set of real-valued vertex functions $v: V_h \to \mathbb{R}$, without assuming a specific inner product structure.

To simplify notation, if $v \in \mathcal{V}_h$, we write $v_k := v(kh)$. Similarly, if $\chi: V_h\times V_h \to \mathbb{R}$, then we write $\chi_{k,l} := \chi(kh, lh)$. By periodicity $v_{\frac1h}=v_0$, $v_{\frac1h+1}=v_1$, $v_{\frac1h-1}=v_{-1}$, etc. Periodicity imposes similar relationships on $\chi_{k,l}$.

For future reference it is convenient to explicitly state to what some of the relevant definitions reduce in this one-dimensional case:
\begin{align*}
&(\Delta_h u)_k = h^{-2} \left(u_{k+1} + u_{k-1} - 2 u_k\right), \qquad (\nabla_h u)_{k,k\pm1} = h^{-1} \left(u_{k\pm1}-u_k\right),\\
&\langle u, v \rangle_h = \sum_{k=0}^{\frac1h-1} h u_k v_k, \qquad \langle \chi, \phi \rangle_{h,h} = \frac12 \sum_{k=0}^{\frac1h-1} h \left(\chi_{k,k+1}\phi_{k,k+1} + \chi_{k,k-1}\phi_{k,k-1}\right),\\
&\varphi_D^h(u) = \frac12 \sum_{k=0}^{\frac1h-1} h^{-1} \left(u_{k+1}-u_k\right)^2.
\end{align*}
We note that if $\chi$ and $\phi$ are both symmetric (i.e., $\chi_{k,l}=\chi_{l,k}$ and $\phi_{k,l}=\phi_{l,k}$) or both skew-symmetric (i.e., $\chi_{k,l}=-\chi_{l,k}$ and $\phi_{k,l}=-\phi_{l,k}$), then, by periodicity,
\begin{equation}\label{eq:onedimsym}
\langle \chi, \phi \rangle_{h,h} = \sum_{k=0}^{\frac1h-1} h \chi_{k,k+1}\phi_{k,k+1}.
\end{equation}
In particular, this holds if $\chi$ and $\phi$ are both graph gradients of a vertex function (and thus skew-symmetric).

We also observe that $\Delta_h u$ is equal to the second order central difference approximation of the second derivative of a function $u: \mathbb{T} \to \mathbb{R}$.

Finally we note that by the Sobolev embedding \cite[Theorem 4.12]{AF} and the finite measure of $\mathbb{T}$, $H^1(\mathbb{T}) \subset C(\mathbb{T}) \subset L^4(\mathbb{T})$. Hence, if $u\in H^1(\mathbb{T})$, then $W \circ u \in L^1(\mathbb{T})$ and thus $\varphi_W(u)<+\infty$. It follows that $H^1(\mathbb{T}) \subset \mathcal{D}\left(\varphi_{AC}\right)$. Hence, by density of $H^1(\mathbb{T})$ in $L^2(\mathbb{T})$ (as follows from the density of $C^{\infty}(\mathbb{T})$ in $L^2(\mathbb{T})$), we have $\overline{\mathcal{D}\left(\varphi_{AC}\right)} = L^2(\mathbb{T})$.

\subsection{Extension operator and induced inner product}\label{sec:extension}

In our study of the total varation flow in Section~\ref{sec:TVflow} we used an embedding operator $i_h: L^2_h \to L^2(\mathbb{T}^n)$, which preserved the total variation in the sense of formula \eqref{eq:embeddedTV}. This, in turn, ensured that condition \eqref{L} was satisfied as an equality. In the present case the situation is more complicated as we cannot find an embedding operator that preserves the Allen--Cahn functional in a similar way. We define a new embedding operator.

Let $h\in \mathfrak{H}$. We define $I_h :\mathcal{V}_h \to L^2(\mathbb{T})$ via the linear interpolation
\[
	I_hu(x) := \frac{u_{k+1} - u_k}{h} (x-kh) + u_k,
\]
where $k\in\left[\frac1h-1\right]_0$ is such that $kh \leq x < (k+1)h$, i.e., $x\in I_{kh}^h = Q_{kh}^h$ in the notation of Section~\ref{sec:thegraphs}.

A direct computation shows that, for all $u \in \mathcal{V}_h$,
\begin{align}
\varphi_D(I_h u) &= \frac12 \int_{\mathbb{T}} \left|(I_hu)'(x)\right|^2 \, dx = \frac12 \sum_{k=0}^{\frac1h-1} \int_{kh}^{(k+1)h} \left| (I_hu)'(x)\right|^2 \, dx\notag\\
&= \frac12 \sum_{k=0}^{\frac1h-1} h \left(\frac{u_{k+1}-u_k}h\right)^2 \, dx = \frac12 \sum_{k=0}^{\frac1h-1} h^{-1} \left(u_{k+1}-u_k\right)^2 \, dx\notag\\
&= \varphi^h_D(u),\label{eq:phiD}
\end{align}
where we used that $u_{\frac1h}-u_{\frac1h-1} = u_0-u_{\frac1h-1}$ by periodicity. We note that this does not suffice to conclude that \eqref{L} holds for the full functionals $\varphi_{AC}$ and $\varphi_{AC}^h$. In fact, when we prove Theorem~\ref{CAC2} below, we will use Theorem~\ref{AB2} rather than Theorem~\ref{AB1}, to avoid needing requirement \eqref{L}. The identity in \eqref{eq:phiD} above, will be very useful in the proof of Theorem~\ref{CAC2}.

We note that $I_h : \mathcal{V}_h \to L^2(\mathbb{T})$ is a linear operator, but unlike $i_h$, when the domain is equipped with our usual inner product, $I_h: L^2_h \to L^2(\mathbb{T})$ is not an isometry. This is problematic, as we required our embedding to be an isometry in Section~\ref{sec:framework}. Therefore we introduce the inner product $(\cdot ,\cdot)_h$ that is induced by $I_h$, i.e., for $v_1, v_2\in \mathcal{V}_h$,
\[
	(v_1,v_2)_h := \langle I_h v_1, I_h v_2 \rangle_{L^2(\mathbb{T})}.
\]
 To differentiate the Hilbert space of vertex functions equipped with this new inner product from $L^2_h$, we denote the set $\mathcal{V}_h$ equipped with $(\cdot, \cdot)_h$ by $\overline{L}^2_h$. We note that, by definition, this ensures that $I_h: \overline{L}^2_h \to L^2(\mathbb{T})$ {\it is} an isometry. We denote the norm for $v\in \overline{L}_h^2$ by $\|v\|_{\overline h} := \sqrt{(v, v)_h}$.

\begin{lemma} \label{IP}
The operator $I_h$ satisfies, for all $u, v \in \mathcal{V}_h$,
\[
	\langle I_h u, I_h v \rangle_{L^2(\mathbb{T})}
	= \frac{h}{3} \sum^{\frac1h-1}_{k=0} \left[ 2u_k v_k
	+ \frac{1}{2} \bigl( u_k v_{k+1} + u_{k+1} v_k \bigr) \right].
\]
Moreover, $\frac13 \| u \|^2_h = \frac13 \| i_h u \|_{L^2(\mathbb{T})}^2 \leq \| I_h u \|_{L^2(\mathbb{T})}^2 = \|u\|_{\overline h}^2 \leq \| i_h u \|_{L^2(\mathbb{T})}^2 = \| u \|^2_h$.
\end{lemma}

\begin{proof}
For $a, b \in \mathbb{R}$, a direct computation shows that
\[
	\int^h_0 \left| \frac{b-a}{h} x + a \right|^2 dx = \frac{h}{3} (a^2 + b^2 + ab).
\]
This gives for $u\in \mathcal{V}_h$,
\begin{align}
	\| I_h u \|_{L^2(\mathbb{T})}^2 &= \sum^{\frac1h-1}_{k=0} \int^{(k+1)h}_{kh} \left| \frac{u_{k+1}-u_k}{h} (x-kh) + u_k \right|^2 \, dx \notag \\
	&=  \sum^{\frac1h-1}_{k=0} \int^h_0 \left| \frac{u_{k+1}-u_k}{h} x + u_k \right|^2 \, dx\notag \\
&=  \frac{h}3 \sum^{\frac1h-1}_{k=0} \left[u_k^2+u_{k+1}^2 + u_k u_{k+1}\right].\label{eq:normIhu}
\end{align}
Let $u, v \in \mathcal{V}_h$. By the polarization identity, we now get
\begin{align*}
4 \langle I_h u, I_h v \rangle_{L^2(\mathbb{T})} &= \left\| I_h(u+v) \right\|_{L^2(\mathbb{T})}^2 - \left\| I_h(u-v) \right\|_{L^2(\mathbb{T})}^2\\
&= \frac{h}3 \sum^{\frac1h-1}_{k=0} \bigl[(u_k+v_k)^2+(u_{k+1}+v_{k+1})^2 + (u_k+v_k) (u_{k+1}+v_{k+1})\\
&\hspace{0.7cm} - (u_k-v_k)^2 - (u_{k+1}-v_{k+1})^2 - (u_k-v_k) (u_{k+1}-v_{k+1})\bigr]\\
&= \frac{h}3 \sum^{\frac1h-1}_{k=0} \bigl[ 4 u_k v_k + 4 u_{k+1} v_{k+1} + 2 u_k v_{k+1} + 2 u_{k+1} v_k \bigr],
\end{align*}
from which the first desired identity follows, by periodicity.

Using periodicity in \eqref{eq:normIhu}, we also find that
\[
	\| I_h u \|_{L^2(\mathbb{T})}^2 = \frac{h}{3} \sum^{\frac1h-1}_{k=0} (2 u^2_k + u_k u_{k+1}).
\]
Using Young's inequality $|u_k u_{k+1}| \leq (u^2_k + u^2_{k+1})/2$ and periodicity again, yields
\begin{align*}
\frac13 \|u\|^2_h &= \frac{h}3 \sum_{k=0}^{\frac1h-1} u_k^2 = \frac{h}{3} \sum^{\frac1h-1}_{k=0} \left[2 u_k^2 - \frac{u^2_k + u^2_{k+1}}2\right] \leq \frac{h}{3} \sum^{\frac1h-1}_{k=0} (2 u^2_k + u_k u_{k+1})\\
&= \|I_h u\|_{L^2(\mathbb{T})}^2  \leq \frac{h}{3} \sum^{\frac1h-1}_{k=0} \left[2 u_k^2 + \frac{u^2_k + u^2_{k+1}}2\right]\\
&=  \|u\|^2_h.
\end{align*}
By the isometry property of Theorem~\ref{thm:ih_ph_properties} the required inequalities now follow.

\end{proof}

\begin{remark} \label{Q}
Let $a, b\in \mathbb{R}$ and $m\in \mathbb{N}$. A direct computation shows
\[
	\int^h_0 \left| \frac{b-a}{h} x + a \right|^m \, dx  = \begin{cases}
\frac{h}{(b-a)(m+1)} \bigl[ b^{m+1} - a^{m+1}\bigr], & \text{if } m \text{ is even},\\
\frac{h}{(b-a)(m+1)} \bigl[ (\mathrm{sgn }\, b) b^{m+1} - (\mathrm{sgn }\, a) a^{m+1} \bigr], & \text{if } m \text{ is odd}.
\end{cases}
\]
A proof by induction shows that $\displaystyle b^{m+1} - a^{m+1} = (b-a) \sum_{l=0}^m a^l b^{m-l}$, hence if $m$ is even we have
\[
\int^h_0 \left| \frac{b-a}{h} x + a \right|^m \, dx = \frac{h}{m+1} \sum^m_{\ell=0} a^\ell b^{m-\ell}.
\]
This implies in particular that, for $m$ even and for $u \in \mathcal{V}_h$,
\begin{align}
	\|I_h u\|_{L^m(\mathbb{T})}^m &:= \int_{\mathbb{T}} | I_h u |^m \, dx = \sum_{k=0}^{\frac1h-1} \frac{h}{m+1} \sum_{l=0}^m u_k^l u_{k+1}^{m-l} \notag\\
	&\leq \sum_{k=0}^{\frac1h-1} \frac{h}{m+1} \sum_{l=0}^m \left(l \frac{|u_k|^m}{m} + (m-l) \frac{|u_k|^m}{m}\right)\notag\\
&= 	\sum^{1/h-1}_{k=0} h \left| u_k \right|^m := \|u\|^m_{L^m_h}. \label{eq:Lminequality}
\end{align}
The inequality follows from Young's inequality, $\displaystyle |u_k^l u_{k+1}^{m-l}| \leq \frac{l}m |u_k|^m + \frac{m-l}m |u_{k+1}|^m$, and periodicity in $k$.

On the other hand, when $m$ is even, $l$ is odd, and $1\leq l \leq m-1$, Young's inequality $|u_k u_{k+1}| \leq (u^2_k + u^2_{k+1})/2$ tells us that
\begin{align*}
u_k^l u_{k+1}^{m-l} = u_k^{l-1} u_{k+1}^{m-l-1} (u_k u_{k+1}) &\geq u_k^{l-1} u_{k+1}^{m-l-1} \left(-\frac12 u_k^2 -\frac12 u_{k+1}^2\right) \\
&= -\frac12 u_k^{l+1} u_{k+1}^{m-l-1} - \frac12 u_k^{l-1} u_{k+1}^{m-l+1}.
\end{align*}
Using periodicity in $k$, it follows that
\begin{align*}
\sum_{l=0}^m u_k^l u_{k+1}^{m-l} &= \frac12 u_k^m + \frac12 u_{k+1}^m + \sum_{\substack{l=1\\l \text{ is odd}}}^{m-1} \left(u_k^l u_{k+1}^{m-l} + \frac12 u_k^{l+1} u_{k+1}^{m-l-1} + \frac12 u_k^{l-1} u_{k+1}^{m-l+1}\right)\\
&\geq \frac12 u_k^m + \frac12 u_{k+1}^m + \sum_{\substack{l=1\\l \text{ is odd}}}^{m-1} \biggl(-\frac12 u_k^{l+1} u_{k+1}^{m-l-1} - \frac12 u_k^{l-1} u_{k+1}^{m-l+1} \\
&\hspace{3.6cm}+ \frac12 u_k^{l+1} u_{k+1}^{m-l-1} + \frac12 u_k^{l-1} u_{k+1}^{m-l+1}\biggr)\\
&= \frac12 u_k^m + \frac12 u_{k+1}^m.
\end{align*}
Summing over $k$ and using periodicity again, we thus find that, for $m$ even,
\begin{align}
\|I_h u\|_{L^m(\mathbb{T})}^m &= \frac{h}{m+1} \sum_{k=0}^{\frac1h-1} \sum_{l=0}^m u_k^l u_{k+1}^{m-l} \geq \frac{h}{m+1} \sum_{k=0}^{\frac1h-1} \left(\frac12 |u_k|^m + \frac12 |u_{k+1}|^m\right) \notag\\
&= \frac{h}{m+1} \sum_{k=0}^{\frac1h-1} |u_k|^m\notag\\
&= \frac1{m+1} \|u\|_{L^m_h}^m.\label{eq:Lminequality2}
\end{align}

We note that $\|\cdot\|_h = \|\cdot\|_{L_h^2}$, so that the inequalities from Lemma~\ref{IP} correspond to the case $m=2$.
\end{remark}

\begin{lemma}\label{lem:ihIhdiff}
Let $u\in \mathcal{V}_h$, then
\[
\|I_hu - i_h u\|_{L^2(\mathbb{T})}^2 = \frac{h^2}2 \|\nabla_h u\|_{h,h}^2 = h^2 \varphi_D^h(u) \leq h^2 \varphi_{AC}^h(u).
\]
\end{lemma}
\begin{proof}
We compute
\begin{align*}
\|I_h u - i_h u \|_{L^2(\mathbb{T})}^2 &= \sum_{k=0}^{\frac1h-1} \int_{kh}^{(k+1)h} \left| \frac{u_{k+1}-u_k}h (x-kh) + u_k - u_k\right|^2 \, dx \\
&= \sum_{k=0}^{\frac1h-1} \int_0^h \left(\frac{u_{k+1}-u_k}h\right)^2 x^2 \, dx = \frac{h}3  \sum_{k=0}^{\frac1h-1} (u_{k+1}-u_k)^2\\
&= \frac{h^2}2 \|\nabla_h u\|_{h,h}^2 = h^2 \varphi_D^h(u) \leq h^2 \varphi_{AC}^h(u).
\end{align*}
\end{proof}

As a consequence of Lemma~\ref{IP} we have, for $u, v\in \mathcal{V}_h$,
\begin{multline}\label{eq:inprodrelation}
(u, v)_h = \frac{h}{3} \sum^{1/h-1}_{k=0} \left[ 2 u_k v_k 	+ \frac{1}{2} \bigl( u_k v_{k+1} + u_{k+1} v_k \bigr) \right] \\
= \frac{h}{3} \sum^{1/h-1}_{k=0} \left[2 u_k + \frac{u_{k-1}+u_{k+1}}2\right] v_k = \langle \Gamma u, v \rangle_h,
\end{multline}
where we used periodicity and defined the operator $\Gamma$ by
\begin{equation}\label{eq:Gamma}
(\Gamma u)_k := \frac13 \left(2u_k + \frac{u_{k-1}+u_{k+1}}2\right).
\end{equation}
We can consider $\Gamma$ as an operator from $L_h^2$ to $L_h^2$ or from $\overline{L}^2_h$ to $\overline{L}^2_h$.
We note that $\langle \Gamma u, v \rangle_h = (u, v)_h = \langle I_h u, I_h v \rangle_{L^2(\mathbb{T})}$ and thus $\langle \Gamma u, u \rangle_h =  \|u\|_{\overline h}^2 = \|I_h u\|_{L^2(\mathbb{T})}^2$. Hence, by Lemma~\ref{IP},
\begin{equation}\label{eq:Gammaproductinequality}
\frac13 \|u\|_h^2 \leq \langle \Gamma u, u \rangle_h \leq \|u\|_h^2.
\end{equation}

The operator $\Gamma$ can be identified with a circulant $1/h$ by $1/h$-matrix, if one regards $u$ as a $1/h$-dimensional vector. This identification is helpful in deriving some interesting properties of $\Gamma$ and its matrix exponential in Appendix~\ref{app:propertiesofGamma}.

Defining the operators $\tau_+$ and $\tau_-$ by
\begin{equation}\label{eq:taupm}
(\tau_\pm u)_k := u_{k\pm1},
\end{equation}
we have $\Gamma = \frac23 \mathrm{Id} + \frac16 \tau_- + \frac16 \tau_-$.

By direct computation we find that $\tau_\pm$ and $\Delta_h$ commute: $\tau_\pm \Delta_h u = \Delta_h \tau_\pm u$. It follows that $\Gamma$ and $\Delta_h$ also commute. It is also useful to observe that
\[
\left(\nabla_h \tau_\pm u\right)_{k,k+1} = h^{-1} \big( (\tau_\pm u)_{k+1} - (\tau_\pm u)_k\big) = h^{-1} \left( u_{k+1\pm 1} - u_{k\pm 1}\right) = \left(\nabla_h u\right)_{k\pm 1, k+1\pm 1},
\]
and thus by \eqref{eq:onedimsym}, periodicity, and Young's inequality,
\begin{align}
\langle \nabla_h u, \nabla_h \Gamma u \rangle_{h,h} &= h \sum_{k=0}^{\frac1h-1} \left[\left(\nabla_h u\right)_{k,k+1} \left(\nabla_h \left(\frac23 u + \frac16 \tau_- u + \frac16 \tau_+ u\right)\right)_{k, k+1}\right] \notag\\
&=  h \sum_{k=0}^{\frac1h-1} \biggl[\frac23 \left(\nabla_h u\right)_{k,k+1}^2 + \frac16  \left(\nabla_h u\right)_{k,k+1}\left(\nabla_h u\right)_{k-1, k} \notag\\
&\hspace{1.3cm}+ \frac16  \left(\nabla_h u\right)_{k,k+1}\left(\nabla_h u\right)_{k+1, k+2}\biggr]\notag\\
&=  h \sum_{k=0}^{\frac1h-1} \left[\frac23 \left(\nabla_h u\right)_{k,k+1}^2 + \frac13  \left(\nabla_h u\right)_{k,k+1}\left(\nabla_h u\right)_{k-1, k}\right]\notag\\
&\geq  h \sum_{k=0}^{\frac1h-1} \left[\frac23 \left(\nabla_h u\right)_{k,k+1}^2 - \frac16  \left(\nabla_h u\right)_{k,k+1}^2 - \frac16 \left(\nabla_h u\right)_{k-1, k}^2\right]\notag\\
&=  h \sum_{k=0}^{\frac1h-1} \left[\frac23 \left(\nabla_h u\right)_{k,k+1}^2 - \frac13  \left(\nabla_h u\right)_{k,k+1}^2\right] =  h \sum_{k=0}^{\frac1h-1} \frac13 \left(\nabla_h u\right)_{k,k+1}^2\notag\\
&= \frac13 \|\nabla_h u\|_{h,h}^2.\label{eq:nablaGamma}
\end{align}

We note that by periodicity we also have
\begin{equation}\label{eq:taunabla}
\| u - \tau_+ u \|_h^2 = \| u - \tau_- u \|_h^2 = h^2 \| \nabla_h u \|_{h,h}^2.
\end{equation}

Again by periodicity, we have that $\langle \tau_\pm u, v \rangle_h = \langle u, \tau_\mp v \rangle_h$ and thus $\Gamma$ is a self-adjoint operator on $L_h^2$. It follows that $\Gamma - \frac13 \mathrm{Id}$ is also self-adjoint; moreover, by Lemma~\ref{IP} it is positive semidefinite on $L_h^2$:
\[
\left\langle \left(\Gamma - \frac13 \mathrm{Id}\right)u, u \right\rangle_h = \| I_h u \|_{L^2(\mathbb{T})}^2 - \frac13  \| u \|^2_h  \geq 0.
\]

As an aside, a direct computation also shows that $\displaystyle \left\langle \left(\Gamma - \mathrm{Id}/3\right)u, u \right\rangle_h$ is proportional to the {\it signless} graph Dirichlet functional\footnote{The plus sign in $u_k+u_{k+1}$ on the penultimate right-hand side of \eqref{eq:positivesemidef}, rather than a minus sign, makes it (propertional to) the {\it signless} graph Dirichlet functional.} on $\mathcal{V}_h$ (see, for example, \cite[Section 2.1]{KvG}):
\begin{align}
\left\langle \left(\Gamma - \frac13 \mathrm{Id}\right)u, u \right\rangle_h &= \sum_{k=0}^{\frac1h-1} h \left(\frac13 u_k^2 + \frac16 u_{k-1}u_k + \frac16 u_{k+1}u_k\right) = \frac{h}3 \sum_{k=0}^{\frac1h-1} \left(u_{k+1} u_k + u_k^2\right)\notag\\
&= \frac{h}3 \sum_{k=0}^{\frac1h-1} \left(u_{k+1} u_k + \frac12 u_k^2 + \frac12 u_{k+1}^2\right) = \frac{h}6 \sum_{k=0}^{\frac1h-1} \left(u_k + u_{k+1}\right)^2 \geq 0,\label{eq:positivesemidef}
\end{align}
where we used periodicity twice. We point out that the computation above also proves that $\Gamma$ itself is positive definite on $L_h^2$, since $\langle \Gamma u, u \rangle_h \geq \frac13 \langle u, u \rangle_h$. In particular, since $L_h^2$ has finite dimension, the (linear) inverse operator $\Gamma^{-1}: L_h^2 \to L_h^2$ exists.

By Lemma~\ref{IP} it also follows that $\mathrm{Id}-\Gamma$ is self-adjoint. It is also positive semidefinite on $L^2_h$, since, by \eqref{eq:inprodrelation}, $\langle (\mathrm{Id}-\Gamma)u, u\rangle_h = \| u \|^2_h - \| I_h u \|_{L^2(\mathbb{T})}^2\geq 0$. It will be useful to have an explicit expression for this difference of norms:
\begin{align}
\| u \|^2_h - \| I_h u \|_{L^2(\mathbb{T})}^2 &= \sum_{k=0}^{\frac1h-1} h \left[\frac12 u_k^2 + \frac12 u_{k+1}^2 - \frac13 u_k^2 - \frac13 u_{k+1}^2 - \frac13 u_k u_{k+1}\right] \notag\\
&= \frac16 \sum_{k=0}^{\frac1h-1} h (u_{k+1}-u_k)^2\notag\\
&= \frac{h^2}{6} \| \nabla_h u \|^2_{h,h},\label{eq:u-I_hu}
\end{align}
where we used \eqref{eq:normIhu} and periodicity.

Since $\Gamma$ is a positive definite operator on $L_h^2$, the square root $\Gamma^{\frac12}: L_h^2 \to L_h^2$ exists and is positive definite. Moreover, since $\Gamma$ is symmetric, so\footnote{Since $\Gamma$ is self-adjoint, there is a spectral decomposition $\Gamma=\sum_j\lambda_j P_j$ where the $P_j$ are the orthogonal projections onto the eigenspaces corresponding to the eigenvalues $\lambda_j$, which all are real. The summation is over finitely many $j$, since $L_h^2$ has finite dimension.
Since $\Gamma$ is positive definite, we have, for all $j$, $\lambda_j>0$. Then $\Gamma^{1/2}$ is defined as
\[
	\Gamma^{1/2} := \sum_j \lambda^{1/2}_j P_j.
\]
Since $P_j$ is self-adjoint, this $\Gamma^{1/2}$ must be self-adjoint.
 This argument can be generalized to self-adjoint operators on a Hilbert space. For any real valued continuous function $f$ and self-adjoint operator $A$ on a Hilbert space, $f(A)$ is defined by the spectral decomposition \cite[Section XI.6, Theorem 1]{Yosida95} and $f(A)$ is self-adjoint \cite[Section XI.12, Theorem 3]{Yosida95}. We also refer to \cite[Chapter 10]{Roman08} for further details about square roots of positive semidefinite self-adjoint operators.} is $\Gamma^{\frac12}$.
By \eqref{eq:Gammaproductinequality}, it follows that
\begin{align*}
(\Gamma u, u)_h = \langle \Gamma^2 u, u \rangle_h &= \langle \Gamma^{\frac12} \Gamma u, \Gamma^{\frac12} u \rangle_h = \langle \Gamma \Gamma^{\frac12} u, \Gamma^{\frac12} u \rangle_h \geq \frac13 \langle \Gamma^{\frac12} u, \Gamma^{\frac12} u \rangle_h \\
&= \frac13 \langle \Gamma u, u \rangle_h = \frac13 (u,u)_h,
\end{align*}
thus $\Gamma$ is also positive definite as operator on $\overline{L}_h^2$. Hence, since $\overline{L}_h^2$ has finite dimension, the (linear) inverse operator $\Gamma^{-1}: \overline{L}_h^2 \to \overline{L}_h^2$ exists. Furthermore, since $(\Gamma u, u)_h = \langle \Gamma^2 u, u \rangle_h = (u, \Gamma u)_h$, $\Gamma$ is also symmetric as operator on $\overline{L}_h^2$.

We end this subsection with the computation of the subdifferential of $\varphi_{AC}^h$ with respect to the $L_h^2$ and $\overline{L}^2_h$ inner products. To be able to distinguish the subdifferential with respect to $\langle \cdot, \cdot \rangle_h$ (which is the one we have been using until now) from the subdifferential with respect to $(\cdot, \cdot)_h$, from now on we will write $\partial_{L_h^2}$ for the former and $\partial_{\overline{L}_h^2}$ for the latter.

We recall from \eqref{eq:lambdaconvexsubdiff} that, since $\varphi_{AC}^h$ is geodesically ($-\alpha$)-convex, we are interested in the shifted subdifferentials $-\partial_{L_h^2}\left(\varphi_{AC}^h(u)+ \frac\alpha2 \|u\|_h^2\right) + \alpha u$ and
\[
-\partial_{\overline{L}_h^2}\left(\varphi_{AC}^h(u) + \frac\alpha2 \|u\|_{\overline h}^2\right) +\alpha u.
\]

\begin{lemma} \label{IS}
Let $u\in \mathcal{V}_h$.
The subdifferential of $\varphi^h_D$ at $u$ with respect to the $L^2_h$ inner product is the singleton $\partial_{L_h^2} \varphi^h_D(u)=\{-\Delta_h u\}\subset L^2_h$. Moreover,
\[
-\partial_{L_h^2}\left(\varphi_{AC}^h(u)+ \frac{\alpha}{2} \|u\|_h^2\right) + \alpha u =\{\Delta_h u-W'\circ u\}\subset L^2_h.
\]

The subdifferential of $\varphi^h_D$ at $u$ with respect to the $\overline{L}^2_h$ inner product is the singleton $\partial_{\overline{L}_h^2} \varphi^h_D(u)=\{-\Gamma^{-1} \Delta_h u\}\subset \overline{L}^2_h$. Furthermore,
\[
-\partial_{\overline{L}_h^2}\left(\varphi_{AC}^h(u)+ \frac\alpha2\|u\|_{\overline h}^2\right) + \alpha u =\{\Gamma^{-1} \Delta_h u - \Gamma^{-1}(W'\circ u)\}\subset L^2_h.
\]
\end{lemma}

\begin{proof}
First we note that, for all $v\in \mathcal{V}_h$,
\begin{align*}
	\varphi_D^h(u+v) - \varphi_D^h(u) &= h^{-1} \sum_{k=0}^{\frac1h-1} \bigl[(u_{k+1}-u_k)(v_{k+1}-v_k) + \frac12 (v_{k+1}-v_k)^2 \bigr]\\ &= \langle \nabla_h u, \nabla_h v\rangle_{h,h} + \frac12 \|\nabla_h v\|_{h,h}^2 \geq \langle -\Delta_h u, v \rangle_h,
\end{align*}
where we used \eqref{eq:DeltaNabla} and \eqref{eq:onedimsym}. Since $\varphi_D^h$ is convex, it follows from the definition of subdifferential that $-\Delta_h u\in \partial_{L^2_h}\varphi^h_D$. Because the domain $V_h$ of $\varphi_D^h$ has finite dimension and since the functional is smooth, the subdifferential is in fact a singleton (and agrees with the gradient of the functional), thus $\partial_{L^2_h}\varphi^h_D = \{-\Delta_h u\}$.

Moreover, by \eqref{eq:inprodrelation} we have $(-\Delta_h u, v )_h =\langle -\Gamma\Delta_h u, v \rangle_h$ and thus
\[
\varphi_D^h(u+v) - \varphi_D^h(u) \geq  (-\Gamma^{-1}\Delta_h u, v )_h.
\]
As above, it follows that $-\Gamma^{-1}\Delta_h u\in \partial_{\overline{L}^2_h}\varphi^h_D$ and, since the subdifferential is a singleton, in fact $\partial_{\overline{L}^2_h}\varphi^h_D = \{-\Gamma^{-1}\Delta_h u\}$.

Since $W(\cdot)+\frac\alpha2\|\cdot\|_h^2$ is convex and differentiable, we have, for all $v\in \mathcal{V}_h$,
\[
W(u+v)+\frac\alpha2\|u+v\|_h^2 - W(u)-\frac\alpha2\|u\|_h^2 \geq \langle W'\circ u + \alpha u, v \rangle_h.
\]
Hence, for all $v\in \mathcal{V}_h$,
\[
\varphi_{AC}^h(u+v) + \frac\alpha2 \|u+v\|_h^2 - \varphi_{AC}^h(u) - \frac\alpha2 \|u\|_h^2  \geq \langle -\Delta_h u + W'\circ u + \alpha u, v \rangle_h
\]
and similarly
\begin{align*}
\varphi_{AC}^h(u+v) + \frac\alpha2 \|u+v\|_{\overline h}^2 - \varphi_{AC}^h(u) - \frac\alpha2 \|u\|_{\overline h}^2  &\geq \langle -\Delta_h u + W'\circ u + \alpha \Gamma u, v \rangle_h\\
&= -(\Gamma^{-1} \Delta_h u - \Gamma^{-1} (W'\circ u) - \alpha u, v)_h,
\end{align*}
where the equality follows from \eqref{eq:inprodrelation}.

The remaining results now follow, using again the fact that the subdifferentials are singletons.
\end{proof}

\subsection{Convergence result}

Our main result will be the convergence of discrete Allen--Cahn gradient flows to a continuum Allen--Cahn gradient flow, which we formalise in Theorem~\ref{CAC}.

We remind ourselves that a discrete Allen--Cahn gradient flow is a solution of $\dot{u} \in -\partial_{L^2_h} \varphi^h_{AC}$. It satisfies equation \eqref{eq:graphAC}, which (by Lemma~\ref{IS}) for $n=1$ is explicitly given by
\begin{equation} \label{DAC}
	\dot{u}_k = (\Delta_h u)_k - W'(u_k)
\end{equation}
for $k\in\left[\frac1h-1\right]_0$. A continuum Allen--Cahn flow is an $L^2(\mathbb{T})$-gradient flow of $\varphi_{AC}(u)$ in $L^2(\mathbb{T})$ and is given by \eqref{AC} (with $n=1$). Given initial data, the unique existence of such a flow is guaranteed by Proposition~\ref{UE}, since $\varphi_{AC}$ is $(-\alpha)$-convex. The same proposition, or the Picard--Lindel\"of theorem for ordinary differential equations (ODE) \cite[Theorem I.3.1]{Hale}, also guarantees the unique existence of a discrete Allen--Cahn gradient flow, i.e., a solution of \eqref{DAC}), given initial data.

\begin{theorem} \label{CAC}
Let $u_0 \in L^2(\mathbb{T})$ and, for all\footnote{We can also restrict ourselves to a proper subset of $\mathfrak{H}$, as long as it contains a sequence converging to zero.} $h\in \mathfrak{H}$, $u^h_0 \in L^2_h$ (for $n=1$).
 Assume that there exists an $\overline h>0$, such that $h\mapsto\|u^h_0\|_\infty$ is bounded on $(0,\overline h)\cap\mathfrak{H}$. Furthermore, assume that $h\|\nabla u_0^h\|_{h,h} \to 0$ and $I_h u^h_0 \to u_0$ in $L^2(\mathbb{T})$ as $h\to 0$ in $\mathfrak{H}$.
 Let $u^h$ be the solution of \eqref{DAC} starting from $u_0^h$.
 Let $u$ be the solution of \eqref{AC} starting from $u_0$.
 Then, for all $T>0$,
\[
	\lim_{h \to 0} \sup_{t\in [0,T]} \|I_h u^h(t) - u(t)\|_{L^2(\mathbb{T})} = 0.
\]
\end{theorem}
\begin{remark}
By Lemma~\ref{IP}, $\|\cdot\|_{L^2(\mathbb{T})} \circ i_h$ and $\|\cdot\|_{L^2(\mathbb{T})} \circ I_h$ are equivalent norms on $\mathcal{V}_h$. Hence the result of Theorem~\ref{CAC} also holds if the operator $I_h$ is replaced by $i_h$ in the assumptions of Theorem~\ref{CAC}. Moreover, using Lemma~\ref{lem:ihIhdiff} we have
\begin{align*}
\|i_h u^h(t) - u(t)\|_{L^2(\mathbb{T})} &\leq  \|i_h u^h(t) - I_h u(t)\|_{L^2(\mathbb{T})}  +  \|I_h u^h(t) - u(t)\|_{L^2(\mathbb{T})}\\
&\leq  h \left(\varphi_{AC}^h(u(t))\right)^{1/2} +  \|I_h u^h(t) - u(t)\|_{L^2(\mathbb{T})}.
\end{align*}
By Lemma~\ref{lem:Phidecrease}
\[
\varphi_{AC}^h(u(t)) \leq \varphi_{AC}^h(u^h_0) \leq \frac12 \|\nabla_h u^h_0\|_{h,h}^2 + C_h,
\]
with $C_h := \sup\left\{W(s) \bigm| |s|\leq \|u_0^h\|_\infty\right\} = \max\left(W(0), W(\|u_0^h\|_\infty)\right)$. (The last equality follows from the specific definition of $W$ that we are using.) Hence, for $h\in (0,\overline h)\cap \mathfrak{H}$,
\[
 h \left(\varphi_{AC}^h(u(t))\right)^{1/2} \leq \frac{h}{\sqrt 2} \|\nabla_h u^h_0\|_{h,h} + Ch,
\]
where $C:= \max_{h\in (0,\overline h)\cap \mathfrak{H}} C_h$. Thus, by the assumptions on the initial condition $u_0^h$,
\[
\lim_{h\to 0} h \left(\varphi_{AC}^h(u(t))\right)^{1/2} = 0.
\]
Hence we can also replace $I_h$ by $i_h$ in the conclusion of Theorem~\ref{CAC} and the result still holds.
\end{remark}

Unfortunately, we cannot apply our abstract Theorem~\ref{AB2} immediately since the embedding mapping $I_h$ is not an isometry from $L^2_h$ to $L^2(\mathbb{T})$.
To circumvent this difficulty, we proceed in two steps.

First we consider the gradient flow of $\varphi^h_{AC}$ with respect to the $\overline{L}^2_h$ inner product. By Lemma~\ref{IS} its explicit form is

\begin{equation} \label{MDAC}
	\dot{U}_k = \left(\Gamma^{-1} \left(\Delta_h U - W'\circ U\right)\right)_k,
\end{equation}
for $k\in\left[\frac1h-1\right]_0$. As for \eqref{DAC}, standard ODE techniques guarantee that, given initial data, a unique solution exists. Since $I_h: \overline{L}^2_h \to L^2(\mathbb{T})$ {\it is} an isometry, we can apply Theorem~\ref{AB2} to compare a solution of \eqref{MDAC} with a solution of \eqref{AC}. The following theorem makes this precise.

\begin{theorem} \label{CAC2}
Assume the same hypotheses concerning $u_0$, $u^h_0$ and $u$ as in Theorem~\ref{CAC}.
 Let $U^h$ be the solution of \eqref{MDAC} starting from $u^h_0$.
 Then, for all $T>0$,
\[
	\lim_{h \to 0} \sup_{t\in [0,T]} \|I_h U(t)^h -u(t) \|_{L^2(\mathbb{T})} = 0.
\]
\end{theorem}

We shall prove Theorem~\ref{CAC2} as an application of Theorem~\ref{AB2} in Section~\ref{sec:applicationabstract} with $M_h=\overline{L}_h^2$ and $M=L^2(\mathbb{T})$.

The second step in our proof of Theorem~\ref{CAC} is a comparison between $u^h$, the solution of \eqref{DAC}, and $U^h$, the solution of \eqref{MDAC}.

\begin{theorem} \label{CAC3}
Assume the same hypotheses concerning $u^h_0$ and $u^h$ as in Theorem~\ref{CAC} and the same hypothesis concerning $U^h$ as in Theorem~\ref{CAC2}.
 Then
\[
	\lim_{h \to 0} \sup_{t\in [0,T]} \|I_h u^h(t) -I_h U^h(t) \|_{L^2(\mathbb{T})} = 0.
\]
\end{theorem}

We defer the proof of Theorem~\ref{CAC3} to Section~\ref{sec:twodiscreteflows}.

\begin{proof}[Proof of Theorem \ref{CAC}]
Our main convergence result Theorem~\ref{CAC} follows from Theorem~\ref{CAC2} and Theorem~\ref{CAC3}.
\end{proof}

The remaining parts of Section~\ref{sec:ACflow} are organised as follows. We will actually prove our `second step' (Theorem~\ref{CAC3}) first, in Section~\ref{sec:twodiscreteflows}. The reason for this reordering of our steps is that one of the results from that section, Lemma~\ref{MDAC}, will also be required in our proof of Theorem~\ref{CAC2}. In Section~\ref{sec:twodiscreteflows} we prove Theorem~\ref{TDF} and in Section~\ref{sec:applicationabstract} Theorem~\ref{CAC2}, which also completes the proof of Theorem~\ref{CAC}. This proof will be an application of Theorem~\ref{AB2}. We thus require a projection $P_h$ as a counterpart to the embedding $I_h$. We introduce $P_h$ in Section~\ref{sec:projection} and prove some of its properties.

\subsection{Comparison of two discrete flows (proof of Theorem~\ref{CAC3})}\label{sec:twodiscreteflows}

In this subsection, we prove Theorem \ref{CAC3}.
 In fact we consider equations more general than \eqref{MDAC} and \eqref{DAC}. Theorem~\ref{CAC3} will then follow as a special case of Theorem~\ref{TDF}.

As usual, if there is no further specification, we assume that $h\in\mathfrak{H}$.

 Let $F: \mathbb{R} \to \mathbb{R}$ be a nondecreasing continuously differentiable function with the property that $F(0)=0$.
 Let $\lambda\in \mathbb{R}$. We consider, for all $k\in \left[\frac1h-1\right]_0$,
\begin{equation} \label{DAC1}
	\dot{u}_k = (\Delta_h u)_k + \lambda u_k - F(u_k),
\end{equation}
as a generalization of \eqref{DAC}. We recover \eqref{DAC} if $F(x)=W'(x)+\lambda x$ and $\lambda>\alpha$. We note that $W'(x) = \alpha x (x^2-1)$ and $W''(x) = \alpha (3x^2-1)$, so that in this case indeed $F(0) = W'(0) = 0$ and $F'(x) = W''(x)+\lambda > 3\alpha x^2 \geq 0$.

 As a generalization of \eqref{MDAC}, we consider, for all $k\in \left[\frac1h-1\right]_0$,
\begin{equation} \label{MDAC1}
	\Gamma\dot{U}_k = \left( \Delta_h U + \lambda U - F\circ U \right)_k.
\end{equation}

By the Picard--Lindel\"of theorem \cite[Theorem I.3.1]{Hale} the initial value problems corresponding to \eqref{DAC1} and \eqref{MDAC1} are guaranteed to have unique continuously differentiable solutions locally in time. By $C^1\left([0,T^*]; \mathcal{V}_h\right)$ we denote the set of functions $u: [0, T^*] \to \mathcal{V}_h$ that are continuously differentiable (considering one-sided derivatives at the endpoints of the interval). Any operator (such as $\Delta_h$ or $\Gamma$) which is defined on $\mathcal{V}_h$, has a natural action on $C^1\left([0,T^*]; \mathcal{V}_h\right)$, defined by $(\Delta_h u)(t) = \Delta_h u(t)$, $(\Gamma u)(t) = \Gamma u(t)$, etc.

We start by proving a minimum principle for supersolutions and a comparison principle for solutions of \eqref{DAC1}. If $u$ solves equation \eqref{MDAC1}, then it satisfies the required assumptions of part~(i) of Proposition~\ref{CP} below, with
\[
b = -\lambda + \int_0^1 F'(\theta u) \, d\theta \geq -\lambda.
\]

\begin{proposition} \label{CP}
\
\begin{itemize}
\item[(i)] (Minimum principle).
Let $T>0$ and assume that there exist a function $b: [0,T] \times V_h \to \mathbb{R}$ that is bounded below and a function $v\in C^1\left([0,T]; \mathcal{V}_h\right)$ such that, for all $k\in \left[\frac1h-1\right]_0$,
\[
	\dot{v}_k \geq (\Delta_h v)_k - b_k v_k.
\]
 If, for all $k\in \left[\frac1h-1\right]_0$, $v_k(0) \geq 0$, then, for all $k\in \left[\frac1h-1\right]_0$ and for all $t\in [0, T]$, $v_k(t) \geq 0$.

\item[(i\hspace{-1pt}i)] (Growth estimate).
Let $T>0$ and assume that $u \in C^1\left([0,T]; \mathcal{V}_h\right)$ is a solution of \eqref{DAC1} starting from $u_0 \in L^2_h$.
 Then, for all $t\in [0,T_1]$, $\| u(t) \|_\infty \leq \|u_0\|_\infty e^{\lambda t}$.
\end{itemize}
\end{proposition}

\begin{proof}
\begin{itemize}
\item[(i)]
Without loss of generality we may assume that $b< 0$, for if it is not,
let $\tilde\lambda \in \mathbb{R}$ be such that
\[
\tilde \lambda > -\inf_{(k,t)\in\left[\frac1h-1\right]_0\times[0,T]} b_k(t).
\]
Since $b$ is bounded below, this is well-defined. Then we consider the following inequality for $\tilde v(t) := e^{-\tilde\lambda t} v(t)$:
\begin{align*}
\dot{\tilde v}(t) = -\tilde\lambda e^{-\tilde\lambda t} v(t) + e^{-\tilde\lambda t} \dot v(t) &\geq  -\tilde \lambda \tilde v(t) + e^{-\tilde\lambda t} \Delta_h v(t) -  e^{-\tilde\lambda t} b(t) v(t) \\
&= \Delta_h \tilde v(t) - \bigl(b(t)+\tilde \lambda\bigr) \tilde v(t).
\end{align*}
By definition of $\tilde\lambda$, we have $b+\tilde\lambda>0$. Moreover, $\tilde v_k(t) \geq 0$ if and only if $v_k(t) \geq 0$.

Suppose that the conclusion of part~(i) were false.
 Since $v$ is continuous in $t$, the minimum of $v_k(t)$ over all $(k,T)\in \left[\frac1h-1\right]_0\times[0,T]$ exists and is negative (i.e., $<0$).
Take a minimizer $(k_0,t_0) \in \left[\frac1h-1\right]_0\times(0,T]$.
 Then $\dot{v}_{k_0} (t_0) \leq 0$ and
\[
(\Delta_h v)_{k_0}(t_0) = \frac{(v_{k_0+1}(t_0)-v_{k_0}(t_0)) + (v_{k_0-1}(t_0)-v_{k_0}(t_0))}{h^2} \geq 0,
\]
and thus $0 \geq -b_{k_0}(t_0)v_{k_0}(t_0)$.
 This contradicts $v_{k_0}(t_0)<0$.

\item[(i\hspace{-1pt}i)] We will prove that, for all $t\in [0,T]$,
\[
	\max_k u_k (t) \leq \|u_0\|_\infty e^{\lambda t}.
\]
A symmetric argument, of which we will not give the details, yields the estimate from below by $-\|u_0\|_\infty e^{\lambda t}$.

Define $\overline u \in C^1\left(\mathbb{R}; \mathcal{V}_h\right)$ by $\overline u(t) := \|u\|_\infty e^{\lambda t}$. We observe that, for all $t\in \mathbb{R}$, $\overline u(t)$ is constant on $V_h$ and thus $\Delta_h \overline u = 0$. Moreover, for all $k\in \left[\frac1h-1\right]_0$, $(F\circ \overline{u})_k = F\left(\|u\|_\infty e^{\lambda t}\right) \geq 0$, since $F(0)=0$ and $F$ is nondecreasing. Thus $\overline U$ is a supersolution of \eqref{DAC1}, i.e., for all $t\in \mathbb{R}$,
\[
\dot{\overline u}(t) = \lambda \overline u(t) = \Delta_h \overline u(t)+ \lambda \overline u(t)  \geq \Delta_h \overline u(t) +  \lambda \overline u(t) - F\circ \overline u(t)
\]

Thus $v:= \overline u - u$ solves, for all $t\in [0,T]$,
\[
\dot v \geq \Delta_h v + \lambda v - F\circ \overline u + F\circ u =\Delta_h v - b w,
\]
where we defined, for all $k\in\left[\frac1h-1\right]_0$ and for all $t\in [0,T]$,
\[
b_k(t) := -\lambda + \int_0^1 F'\left(\theta \overline u(t) + (1-\theta) u(t)\right) \, d\theta.
\]
Since $u$ and $\overline u$ are continuous, the sets $u([0,T])$ and $\overline u([0, T])$ are compact and thus, by Tychonoff's theorem so is $[0,1]\times u([0,T]) \times \overline u([0,T])$. Since the function $\mathbb{R}^3 \to \mathbb{R}, (\theta, x, y) \mapsto \theta x + (1-\theta) y$ is continuous, the set $\Theta := \left\{\theta x + (1-\theta) y \in \mathbb{R} \Bigm| (\theta, x, y) \in [0,1] \times u([0,T]) \times \overline u([0,T])\right\}$ is compact.
 Because $F'$ is continuous, the set $F'(\Theta)$ is compact and in particular bounded. This implies that the function $b$ is bounded and thus we can apply part~(i) of the current proposition. Doing so yields, for all $t\in [0,T]$ and for all $k\in \left[\frac1h-1\right]_0$, $v_k(t) \geq 0$. The proof is now complete.
\end{itemize}
\end{proof}

\begin{remark}
We note that the comparison principles from Proposition~\ref{CP}~(i\hspace{-1pt}i) allow us to extend our earlier local-in-time existence results for the initial value problems associated with \eqref{DAC1} to global-in-time existence results\footnote{We remark that Proposition~\ref{UE} also gives us global-in-time existence ---based on an energy principle--- but without guarantee of differentiability.}. Given $u_0^h\in L^2_h$, by the local-in-time result there exists $T^*>0$ such that a solution $u^h\in C^1\left([0,T^*]; \mathcal{V}_h\right)$ exists of \eqref{DAC1} with initial condition $u^h(0)=u_0^h$. Since the right-hand side of \eqref{DAC1} is defined for all $u\in L_h^2$ and continuous with respect to $u$, by standard ODE results (e.g \cite[Theorem I.2.1]{Hale}), the maximal right-interval of existence of the solution is $[0,\infty)$ unless there exists a time $T^{**}$ such that the solution can be extended to $[0,T^{**}]$ and $\lim_{t\to T^{**}} u(t) = \pm\infty$. By the comparison princincple it follows that such a $T^{**}$ does not exist and thus the solution exists for all $t>0$.
\end{remark}

In the following proposition we compare a solution $u$ of \eqref{DAC1} with $\Gamma u$.

\begin{proposition} \label{SAC}
Let $h\in \mathfrak{H}$ and let $u$ be a solution of \eqref{DAC1} with initial data $u_0\in \mathcal{V}_h$.
 Then, for all $t$ in the domain of $u$,
\begin{align*}
	e^{-2\lambda t} \|u(t) - \Gamma u(t) \|_h^2
	& + 2 \int^t_0 e^{-2\lambda s} \left\| \nabla_h (u(s) - \Gamma u(s)) \right\|_{h,h}^2 \, ds \\
	& \leq \frac{h^2}{9} \| \nabla_h u_0 \|_{h,h}^2.
\end{align*}

\end{proposition}

\begin{proof}
To simplify notation, we suppress the explicit dependence on $t$ (which is assumed to be in the domain of $u$) in the first half of this proof.

Remembering the definition of $\tau_\pm$ from \eqref{eq:taupm}, we have
\begin{equation}\label{eq:uminusGammau}
	3\left(u - \Gamma u\right) = \frac12\left[\left(u - \tau_+ u\right) + \left( u - \tau_- u\right)\right].
\end{equation}
Hence
\begin{align*}
	\| u - \Gamma u \|_h^2
	= \frac{1}{6^2} \left\| (u - \tau_+ u) + (u - \tau_- u) \right\|_h^2
	&\leq \frac{1}{18} \left( \| u - \tau_+ u \|_h^2 + \| u - \tau_- u \|_h^2 \right) \\
	&= \frac{1}{18} \left( \|v_+ \|_h^2 + \| v_- \|_h^2 \right),
\end{align*}
where we used the notation $v_\pm := u - \tau_\pm u$. We also used that $\|a+b\|^2 \leq (\|a\|+\|b\|)^2 = \|a\|^2+\|b\|^2 + 2\|a\| \|b\| \leq 2 \|a\|^2 + 2 \|b\|^2$, which holds for any norm. Because $\nabla_h$ is a linear operator, we thus also have
\[
	\| \nabla_h(u - \Gamma u) \|_{h,h}^2
	= \frac{1}{6^2} \left\| \nabla_h v_+ + \nabla_h v_- \right\|_{h,h}^2
	\leq \frac{1}{18} \left( \| \nabla_h v_+ \|_{h,h}^2 + \| \nabla_h v_- \|_{h,h}^2 \right).
\]
Since $\tau_\pm u$ solves \eqref{DAC1}, calculating the inner product $\langle u - \tau_\pm u, \dot{u} - \tau_\pm \dot{u} \rangle_h = \frac{1}{2}\frac{d}{dt} \| v_\pm \|_h^2$ yields
\[
	\frac{1}{2}\frac{d}{dt} \| v_\pm \|_h^2
	- \lambda \| v_\pm u\|_h^2
	+ \left\| \nabla_h v_\pm \right\|_{h,h}^2
	= -\left\langle  u - \tau_\pm u,  F\circ u - F\circ \tau_\pm u \right\rangle_h.
\]
Here we used \eqref{eq:DeltaNabla}. The last term is nonpositive since $F$ is nondecreasing.
 This implies
\begin{align*}
	&\hspace{0.4cm} \frac{d}{dt} \left(e^{-2\lambda t} \| v_\pm (t) \|_h^2\right)
	+ 2e^{-2\lambda t} \left\| \nabla_h v_\pm (t) \right\|_{h,h}^2\\
	&= e^{-2\lambda t} \left(\frac{d}{dt} \| v_\pm (t)\|_h^2 - 2\lambda  \| v_\pm (t)\|_h^2 + 2  \left\| \nabla_h v_\pm (t) \right\|_{h,h}^2\right) \leq 0.
\end{align*}
Integrating this inequality over $(0,t)$ and using \eqref{eq:taunabla}, we find
\begin{align}
0 &\geq e^{-2\lambda t} \| v_\pm (t) \|_h^2 -  \| v_\pm (0) \|_h^2 + 2 \int_0^t e^{-2\lambda s} \left\| \nabla_h v_\pm (s) \right\|_{h,h}^2 \, ds\notag\\
&= e^{-2\lambda t} \| v_\pm (t) \|_h^2 - h^2 \| \nabla_h u_0\|_{h,h}^2 + 2 \int_0^t e^{-2\lambda s} \left\| \nabla_h v_\pm (s) \right\|_{h,h}^2 \, ds.\label{eq:vpm}
\end{align}
Recalling the earlier estimates for $\| u - \Gamma u \|_h^2$ and $\| \nabla_h(u - \Gamma u) \|_{h,h}^2$, the desired inequality follows:
\begin{align*}
&\hspace{0.5cm} e^{-2\lambda t} \|u(t) - \Gamma u(t) \|_h^2  + 2 \int^t_0 e^{-2\lambda s} \left\| \nabla_h (u(s) - \Gamma u(s)) \right\|_{h,h}^2\, ds\\
&\leq \frac1{18} \left(e^{-2\lambda t} \|v_+ (t) \|_h^2  + 2 \int^t_0 e^{-2\lambda s} \left\| \nabla_h v_+ (s)) \right\|_{h,h}^2\, ds\right.\\
 &\hspace{1.5cm} \left. + e^{-2\lambda t} \|v_- (t) \|_h^2  + 2 \int^t_0 e^{-2\lambda s} \left\| \nabla_h v_- (s)) \right\|_{h,h}^2\, ds\right)\\
&\leq \frac{h^2}9 \|\nabla_h u_0\|_{h,h}^2.
\end{align*}
\end{proof}

\begin{theorem} \label{TDF}
Let $\lambda \geq 0$ and $T>0$. Assume $u\in C^1([0,T]; \mathcal{V}_h)$ solves \eqref{DAC1} starting from $u_0 \in L^2_h$ and $U\in C^1([0,T]; \mathcal{V}_h)$ solves \eqref{MDAC1} starting from $U_0 \in L^2_h$.
Then there exists a constant $C_N$, depending only on $\lambda$, $T$, and $\|u_0\|_\infty$ (details in \eqref{eq:N} and \eqref{eq:CN}), such that, for all $t\in [0,T]$ and for all $\delta>0$,
\[
	\|\Gamma^{1/2}(u(t) - U(t)) \|_h^2 \leq  \left[ h^2 e^{2\lambda T}\left(\frac{1}{18}+\frac{C_N}{\lambda \delta} \right) \|\nabla_h u_0\|_{h,h}^2 + \Gamma^{1/2}(u_0 - U_0)\|_h^2\right] e^{6(\lambda+\delta)t},
\]
provided that $h\in\mathfrak{H}$.

In particular, if $u_0=U_0$ and $h\in\mathfrak{H}$, then, for all $t\in [0,T]$ and all $\delta>0$, there exists a $C_*>0$ depending on $\lambda$, $\delta$, $T$, and $\|u_0\|_\infty$ (details in \eqref{eq:C*}), such that
\[
	\| u(t) - U(t) \|_h^2 \leq C_* h^2 \| \nabla_h u_0 \|_{h,h}^2 e^{6 (\lambda+\delta) t}.
\]
\end{theorem}

\begin{proof}
Let $h\in (0,h_0) \cap \mathfrak{H}$.
In most lines of the proof, to simplify the notation we will not write the explicit dependence on $t\in [0,T]$.

 Similarly as in the proof of Proposition \ref{SAC}, we shall calculate $(v,v)_h=\langle v,\Gamma v \rangle_h=\|\Gamma^{1/2}v\|^2_h$ with $v := u-U$.
 Using equations \eqref{DAC1} for $u$ and \eqref{MDAC1} for $U$, together with the self-adjointness of $\Gamma$ and the fact that $\Gamma$ and $\frac{d}{dt}$ commute, we obtain that
\begin{multline}\label{eq:vdotv}
	\frac12\frac{d}{dt} \|\Gamma^{1/2}v\|^2_h = \langle v, \Gamma\dot{v} \rangle_h \\
	= \underbrace{\langle v, \Gamma\Delta_h u - \Delta_h U \rangle_h}_{:=I_1} \
	\underbrace{- \left\langle v, \Gamma F\circ u - F\circ U \right\rangle_h}_{:=I_2}
	+ \underbrace{\lambda \langle v, \Gamma u-U \rangle_h}_{:=I_3}.
\end{multline}

We proceed by computing $I_1$, using \eqref{eq:DeltaNabla}:
\begin{align*}
	I_1 &= \langle v, \Delta_h v \rangle_h
	+ \left\langle v, \Gamma\Delta_h u - \Delta_h u \right\rangle_h \\
	&= - \langle \nabla_h v, \nabla_h v \rangle_{h,h}
	- \left\langle \nabla_h v, \nabla_h(\Gamma u - u) \right\rangle_{h,h}
\end{align*}
since $\Delta_h \Gamma=\Gamma\Delta_h$. Applying Young's inequality $ab \leq (a^2 + b^2)/2$, we obtain
\begin{align*}
	I_1 &\leq -\| \nabla_h v \|_{h,h}^2 + \frac{1}{2} \| \nabla_h v \|_{h,h}^2 + \frac{1}{2} \left\| \nabla_h (\Gamma u - u) \right\|_{h,h}^2 \\
	& = -\frac{1}{2} \| \nabla_h v \|_{h,h}^2 + \frac{1}{2} \left\| \nabla_h (u - \Gamma u) \right\|_{h,h}^2.
\end{align*}

We next calculate $I_2$.
 Since $F$ is  nondecreasing, we see that
\begin{align}
	I_2 &= -  \left\langle v, F\circ u - F\circ U \right\rangle_h - \left\langle v, \Gamma F\circ u - F\circ u \right\rangle_h \leq - \left\langle v, \Gamma F\circ u- F\circ u \right\rangle_h\notag\\
	&\leq \|v\|_h \ \| \Gamma F\circ u - F\circ u \|_h,\label{eq:I2estimate}
\end{align}
where we used the Cauchy--Schwarz inequality to obtain the last inequality.
Because
\[
\left(\Gamma F \circ u\right)_k = \frac23 F(u_k) + \frac16 F(u_{k-1}) + \frac16 F(u_{k+1}),
\]
a similar observation as the one for $u-\Gamma u$ in \eqref{eq:uminusGammau} shows that
\[
	F\circ u - \Gamma F\circ u
	= \frac{1}{6} \left[ \left( F\circ u - F \circ \tau_+ u\right) + \left( F\circ u - F \circ \tau_- u\right) \right].
\]
By Proposition~\ref{CP}~(i\hspace{-1pt}i), we know $\| u \|_\infty(t) \leq \|u_0\|_\infty e^{\lambda t}$.
 Thus, pointwise on $V_h$,
\begin{align*}
	\left| F\circ u(t) - F \circ \tau_\pm u(t) \right|
	&\leq \left| \int^1_0 F' \left( \theta u(t)+(1-\theta) \tau_\pm u(t) \right) \, d\theta \right|
	\left| u(t) - \tau_\pm u(t) \right| \\
	&\leq N \left|u(t) - \tau_\pm u(t) \right|
\end{align*}
with
\begin{equation}\label{eq:N}
N := \sup \left\{ F'(s) \bigm| |s| \leq \|u_0\|_\infty e^{\lambda T} \right\}.
\end{equation}
This estimate, combined with \eqref{eq:I2estimate} and the triangle inequality, yields
\begin{align*}
	I_2 &\leq \frac16 \|v\|_h \left(\| F\circ u - F \circ \tau_+ u\|_h + \| F\circ u - F \circ \tau_- u\|_h\right)\\
&\leq  \frac{N}6 \| v\|_h \left( \| u - \tau_+ u \|_h + \| u - \tau_- u \|_h \right).
\end{align*}
Similarly, using \eqref{eq:uminusGammau} and the Cauchy--Schwarz and triangle inequalities again, we estimate
\begin{align*}
		I_3 & = \lambda \| v\|_h^2 + \lambda \langle v, \Gamma u-u \rangle_h \\
	&\leq \lambda \| v\|_h^2 + \frac\lambda6 \| v \|_h \left( \| u-\tau_+ u \|_h +  \| u-\tau_- u \|_h \right).
\end{align*}
Let $\delta>0$. Applying Young's inequality, $ab\leq \frac\delta2 a^2+\frac1{2\delta} b^2$, to the estimates for $I_2$ and $I_3$ above, we find
\begin{align*}
I_2 &\leq \frac\delta2 \|v\|_h^2 + \frac1{2\delta} \left(\frac{N}6\right)^2 \left( \| u - \tau_+ u \|_h + \| u - \tau_- u \|_h \right)^2,\\
I_3 &\leq \lambda \| v\|_h^2 + \frac\delta2 \|v\|_h^2 + \frac1{2\delta} \left(\frac\lambda6\right)^2 \left( \| u-\tau_+ u \|_h +  \| u-\tau_- u \|_h \right)^2.
\end{align*}

Combining our estimates for $I_1$, $I_2$, and $I_3$ with \eqref{eq:vdotv} we find
\begin{multline*}
	\frac{1}{2} \frac{d}{dt} \| \Gamma^{1/2} v\|_h^2(t) + \frac{1}{2} \| \nabla_h v \|_{h,h}^2 \\
	\leq \frac{1}{2} \left\| \nabla_h(u - \Gamma u) \right\|_{h,h}^2 + (\lambda + \delta) \| v \|_h^2
	+ \frac{C_N}{2 \delta} \left( \| u - \tau_+ u \|_h^2 + \| u - \tau_- u \|_h^2 \right)
\end{multline*}
with
\begin{equation}\label{eq:CN}
C_N :=\frac{N^2+\lambda^2}{36}.
\end{equation}
 Integrating over $(0,t)$ and multiplying by $2$, we get, by rearranging terms,
\begin{align*}
	&\|\Gamma^{1/2} v(t)\|_h^2 + \int^t_0 \| \nabla_h v(s) \|_{h,h}^2\, ds
	\leq \int^t_0 \left\| \nabla_h (u - \Gamma u)(s) \right\|_{h,h}^2\, ds \\
	&+ 2 (\lambda + \delta) \int^t_0 \| v(s) \|_h^2\, ds \\
	&+ \frac{C_N}{\delta} \left( \int^t_0 \| u(s) - \tau_+ u(s) \|_h^2\, ds + \int^t_0 \| u(s) - \tau_- u(s) \|_h^2\, ds \right)
	+ \| \Gamma^{1/2}(u_0 - U_0) \|_h^2.
\end{align*}

By Proposition~\ref{SAC} we have
\begin{align*}
\int_0^t \left\| \nabla_h (u(s) - \Gamma u(s)) \right\|_{h,h}^2 \, ds &\leq e^{2\lambda T} \int_0^t e^{-2\lambda s} \left\| \nabla_h (u(s) - \Gamma u(s)) \right\|_{h,h}^2 \, ds \\
&\leq \frac{h^2}{18} e^{2\lambda T} \|\nabla_h u_0\|_{h,h}^2.
\end{align*}
From \eqref{eq:vpm} in the proof of Proposition~\ref{SAC}, we see that
\begin{equation*}
e^{-2\lambda t} \| u(t)-\tau_\pm u(t)\|^2_h
+ 2\int_0^t e^{-2\lambda s} \left\| \nabla_h (u-\tau_\pm u)(s) \right\|^2_{h,h}\, ds
\leq h^2 \|\nabla_h u_0\|_{h,h}^2.
\end{equation*}
This yields
\begin{align*}
\int_0^t \| u(s) - \tau_\pm u(s) \|_h^2\, ds &\leq h^2 \|\nabla_h u_0\|_{h,h}^2 \int_0^t e^{2\lambda s} \, ds = \frac{h^2}{2\lambda} \|\nabla_h u_0\|_{h,h}^2 \left(e^{2\lambda t}-1\right)\\
 &\leq \frac{h^2}{2\lambda} e^{2\lambda T} \|\nabla_h u_0\|_{h,h}^2.
\end{align*}

Applying these estimates, we deduce that
\begin{align*}
	\| \Gamma^{1/2} v(t) \|_h^2 &+\int^t_0 \| \nabla_h v(s) \|_{h,h}^2\, ds\\
&\leq h^2 e^{2\lambda T} \left(\frac{1}{18}+\frac{C_N}{\lambda\delta}\right) \|\nabla_h u_0\|_{h,h}^2 + \|\Gamma^{1/2}(u_0 - U_0)\|_h^2 \\
	&+ 2 (\lambda + \delta) \int^t_0 \| v(s) \|_h^2\, ds.
\end{align*}

After applying $\|v\|^2_h \leq 3 \|\Gamma^{1/2} v\|^2_h$ from \eqref{eq:Gammaproductinequality} to the integrand on the right-hand side, Gr\"onwall's inequality yields, for all $t\in [0,T]$,
\[
	\| \Gamma^{1/2}v(t) \|_h^2 \leq \left[ h^2 e^{2\lambda T} \left(\frac{1}{18}+\frac{C_N}{\lambda\delta}\right) \|\nabla_h u_0\|_{h,h}^2 + \|\Gamma^{1/2}(u_0 - U_0)\|_h^2\right] e^{6(\lambda+\delta)t}.
\]
This proves the first result.

In the special case where $u_0=U_0$, we find
\[
\| v(t) \|_h^2 \leq C_* h^2 \|\nabla_h u_0\|_{h,h}^2 e^{6(\lambda+\delta) t}.
\]
with
\begin{equation}\label{eq:C*}
	C_* := e^{2\lambda T} \left( \frac{1}{6} + \frac{3 C_N}{\lambda\delta} \right).
\end{equation}
This proves the second desired estimate.
\end{proof}

We can now prove Theorem~\ref{CAC3}.

\begin{proof}[Proof of Theorem \ref{CAC3}]
Let $\delta>0$ and $t\in [0,T]$. By Lemma Lemma~\ref{IP} and Theorem~\ref{TDF} (with $u=u^h$, $U=U^h$, $\lambda>\alpha$, and $F(x)=W'(x)+\lambda x$)  we have, for $h$ small enough,
\[
\| I_h U^h(t) - I_h u^h(t) \|_{L^2(\mathbb{T})} = \|\Gamma^{1/2}(U^h(t)-u^h(t)) \|_h \leq  C_*^{1/2} h \| \nabla_h u_0^h \|_{h,h} e^{3(\lambda+\delta) t}.
\]
By assumption we also have that $\|u_0^h\|_\infty$ is bounded uniformly in $h$, for $h$ small enough. Consequently so is $C_*^{1/2}$ (which depends on $h$ only through $\|u_0\|_\infty$).

First taking the supremum over $t\in [0,T]$ in the inequality above and then the limit for $h\to 0$ yields Theorem~\ref{CAC3}, since, by assumption, $h\|\nabla_h u_0^h\|_{h,h} \to 0$ as $h\to 0$.
\end{proof}

\subsection{Projections}\label{sec:projection}

Let $h\in \mathfrak{H}$. We define $H_h$ to be the space of real-valued continuous piecewise linear functions on $\mathbb{T}$ whose derivatives can only have jumps at the points $kh\in \mathbb{T}$ with $k\in\{0,\ldots,1/h-1\}$. (As an aside we observe that, if $f\in H_{h,i}$ ---see \eqref{eq:Hhi}--- then, keeping all coordinates but $x_i$ fixed, the function $x_i \mapsto f(x)$ is in $H_h$.) We note that $I_h \overline{L}^2_h = H_h$. Moreover, $H_h$ is a linear subspace of $L^2(\mathbb{T})$ and, having finite dimension, it is also closed.

To be able to use Theorem~\ref{AB2} to prove Theorem~\ref{CAC2}, we require a projection $P_h$ corresponding to the embedding $I_h: \overline{L}_h^2 \to L^2(\mathbb{T})$, as explained in Section~\ref{sec:framework}. We will define this projection using a few intermediary functions, which will prove useful in their own right.

We define $\overline{P}_h: L^2(\mathbb{T}) \to H_h$ to be the orthogonal projection from $L^2(\mathbb{T})$ to $H_h$. As one of the standard properties of orthogonal projections, it follows immediately that $\overline{P}_h$ satisfies, for all $w\in L^2(\mathbb{T})$,
\begin{equation}\label{eq:nonexpansive}
\|\overline{P}_h w\|_{L^2(\mathbb{T})} \leq \|w\|_{L^2(\mathbb{T})}.
\end{equation}
Moreover, for all $w\in L^2(\mathbb{T})$, there exists a $w^\perp\in H_h^\perp$ (where $H_h^\perp$ denotes the orthogonal complement of $H_h$ in $L^2(\mathbb{T})$, such that $w = \overline{P}_h w + w^\perp$. Since $1\in H_h$, we have
\begin{equation}\label{eq:projectionmass}
\int_{\mathbb{T}} \overline{P}_h w(x) \, dx = \int_{\mathbb{T}} w(x) \cdot 1 \, dx - \int_{\mathbb{T}} w^\perp(x) \cdot 1 \, dx = \int_{\mathbb{T}} w(x) \, dx.
\end{equation}

With $C(\mathbb{T})$ being the set of continuous functions on $\mathbb{T}$, in this one-dimensional case we define the projection $\overline{\pi}_h : C(\mathbb{T}) \to \overline{L}_h^2$ by, for all $w\in C(\mathbb{T})$,
\begin{equation}\label{eq:points}
	(\overline{\pi}_h w)(kh) = w(kh), \qquad \text{for } k\in\left[\frac1h-1\right]_0.
\end{equation}
Now we define $P_h: L^2(\mathbb{T}) \to \overline{L}_h^2$ as $P_h := \overline{\pi}_h \circ \overline{P}_h$. Since the restricted operator $\overline{P}_h|_{H_h}$ is the identity map $\mathrm{Id}_{H_h}$ on $H_h$, we have that $P_h \circ I_h = \overline{\pi}_h \circ \overline{P}_h \circ I_h = \overline{\pi}_h \circ \mathrm{Id}_{H_h} \circ I_h = \mathrm{Id}_{\overline{L}_h^2}$, the identity map on $\overline{L}_h^2$, as required by our discussion in Section~\ref{sec:framework}.

Furthermore, we note that $I_h \circ \overline{\pi}_h|_{H_h} = \mathrm{Id}_{H_h}$. Hence $I_h \circ P_h = I_h \circ \overline{\pi}_h \circ \overline{P}_h = \overline{P}_h$, and thus for all $w \in L^2(\mathbb{T})$ we have
\[
\|P_h w\|_{\overline h} = \|I_h \circ P_h w\|_{L^2(\mathbb{T})} = \|\overline{P}_h w\|_{L^2(\mathbb{T})} \leq \|w\|_{L^2(\mathbb{T})}.
\]
Hence the non-strict contraction requirement from \eqref{eq:contraction} is also satisfied. Moreover, using again that $\overline{P}_h$ is an ortogonal projection, we have, for all $v\in \overline{L}^2_h$ and for all $w\in L^2(\mathbb{T})$,
\begin{align*}
\|v-P_h w\|_{\overline h}^2 + \|I_h P_h w - w\|_{L^2(\mathbb{T})}^2 &= \|I_h v - I_h P_h w\|_{L^2(\mathbb{T})}^2 +  \|I_h P_h w - w\|_{L^2(\mathbb{T})}^2\\
&= \|I_h v - \overline{P}_h w\|_{L^2(\mathbb{T})}^2 +  \|\overline{P}_h w - w\|_{L^2(\mathbb{T})}^2 \\
&= \|I_h v - w\|_{L^2(\mathbb{T})}^2.
\end{align*}
For the final equality, we used that $I_h v\in H_h$. This establishes that $P_h$ satisfies condition \eqref{P}. We postpone our check of condition \eqref{U2} to the proof of Theorem~\ref{CAC2} in Section~\ref{sec:applicationabstract}.

Our goal in the remainder of Section~\ref{sec:projection} is to establish some useful properties of $\overline{P}_h$, which we will need in our proof of Theorem~\ref{CAC2}. We collect these results in Theorem~\ref{PR}. In numerical analysis, these properties sometimes go by the name of stability in $H^1(\mathbb{T})$.

\begin{theorem} \label{PR}
\
\begin{enumerate}
\item[(i)] For all $w\in H^1(\mathbb{T})$ the estimate
\[
	\left\| (\overline{P}_h w)' \right\|_{L^2(\mathbb{T})} \leq C \|w'\|_{L^2(\mathbb{T})}
\]
holds with $C := 4\sqrt{3} /\pi + 2$.
\item[(i\hspace{-1pt}i)] For all $w\in H^1(\mathbb{T})$, $\displaystyle \lim_{h \to 0} \left\|(\overline{P}_h w - w)'\right\|_{L^2(\mathbb{T})}=0$.
\item[(i\hspace{-1pt}i\hspace{-1pt}i)] For all $w\in L^2(\mathbb{T})$, $\displaystyle \lim_{h \to 0} \left\|\overline{P}_h w - w\right\|_{L^2(\mathbb{T})}=0$.
\item[(i\hspace{-1pt}v)] For all $w\in H^1(\mathbb{T})$, $\displaystyle \overline{P}_h w \to w$ uniformly, as $h\to 0$.
\end{enumerate}
\end{theorem}

In \cite[Theorem 1.5]{FSS}, this is proved using the Ritz operator, i.e., the projection operator from $H^1(\mathbb{T})$ to $H_h$, and elliptic regularity in a general domain. A more general estimate in $L^1(\mathbb{T})$ and $W^{1,p}(\mathbb{T})$ has been established in \cite[Theorem 2]{CT}, where the two-dimensional case is also discussed.

Here we will prove Theorem~\ref{PR} by other means. To do so, we first require a different projection operator, $\pi_h$, and its properties which we present in Lemma~\ref{ER}. We present a proof of Theorem~\ref{PR} after we have proven this lemma. We define $\pi_h: C(\mathbb{T}) \to H_h \subset C(\mathbb{T})$, by following $\overline{\pi}_h$ with a linear interpolation on each interval $\bigl(kh, (k+1)h\bigr)$, i.e., $\pi_h := I_h \circ \overline{\pi}_h$.

In the following lemma we collect some basic properties of the error $e_h w := \pi_h w-w \in C(\mathbb{T})$. Recall that $H^2(\mathbb{T}) \subset H^1(\mathbb{T}) \subset C(\mathbb{T})$ by the Sobolev embedding \cite{AF} and so $\pi_h$ is well-defined on $H^1(\mathbb{T})$ and $H^2(\mathbb{T})$. We also note that $H_h \subset H^1(\mathbb{T})$.

In the proof of the following lemma, we require two Poincar{\'e}--Wirtinger(-type) inequalities. In order not to interrupt the flow of the main argument, we defer the statement and proof of these inequalities to Lemma~\ref{lem:PoinWirt} in Appendix~\ref{app:PoincareWirtinger}.

\begin{lemma} \label{ER}
The following estimates are valid.
\begin{enumerate}
\item[(i)] $\| e_h w \|_{L^2(\mathbb{T})}^2 \leq (2 h/\pi)^2 \|w'\|_{L^2(\mathbb{T})}^2$, for all $w \in H^1(\mathbb{T})$.
\item[(i\hspace{-1pt}i)] $\| e_h w \|_{L^2(\mathbb{T})}^2 \leq (h/\pi)^4 \|w''\|_{L^2(\mathbb{T})}^2$, for all $w \in H^2(\mathbb{T})$.
\item[(i\hspace{-1pt}i\hspace{-1pt}i)] $\left\| (e_h w)' \right\|_{L^2(\mathbb{T})}^2 \leq \frac13 h^2 \|w''\|_{L^2(\mathbb{T})}^2$, for all $w \in H^2(\mathbb{T})$.
\item[(i\hspace{-1pt}v)] $\left\| (\pi_h w)' \right\|_{L^2(\mathbb{T})} \leq \|w'\|_{L^2(\mathbb{T})}$, for all $w \in H^1(\mathbb{T})$.
\item[(v)] $\|w'\|_{L^2(\mathbb{T})}^2 \leq (12/h^2) \|w\|_{L^2(\mathbb{T})}^2$, for all $w \in H_h$.
\end{enumerate}
\end{lemma}

\begin{proof}
 We first prove (i\hspace{-1pt}v). Let $w\in H^1(\mathbb{T})$. By definition
\[
	\int^h_0 \left| (\pi_h w)'(x) \right|^2 \, dx
	= \left| \frac{w(h)-w(0)}{h} \right|^2 h
	= \left| \frac{1}{h} \int^h_0 w'(x) \, dx \right|^2 \cdot h.
\]
By H\"older's inequality or Jensen's inequality \cite[Appendix B Theorem 2]{Evans}, this is dominated by
\[
	\frac{1}{h} \int^h_0 \left| w'(x) \right|^2 \, dx \cdot h
	= \int^h_0 |w'|^2 \, dx.
\]
Similarly,
\[
	\int^{(k+1)h}_{kh} |(\pi_h w)'(x)|^2 \, dx
	\leq \int^{(k+1)}_{kh} |w'(x)|^2 \, dx.
\]
Adding these inequalities from $k=0$ to $k=1/h-1$ and taking square roots yields (i\hspace{-1pt}v).

We next prove (i). Let $w\in H^1(\mathbb{T})$.
 Since $e_h w(0) = e_h w(h)=0$, by the Poincar\'e--Wirtinger inequality in Lemma~\ref{lem:PoinWirt} (i) we have
\[
	\int^h_0 |e_h w(x)|^2 \, dx
	\leq (h/\pi)^2 \left\| (e_h w)' \right\|_{L^2((0,h))}^2.
\]
By (i\hspace{-1pt}v) we have $\left\|(\pi_h w)'\right\|_{L^2(\mathbb{T})} \leq \|w'\|_{L^2(\mathbb{T})}$ by (i\hspace{-1pt}v), so that $\left\|(e_h w)'\right\|_{L^2(\mathbb{T})} \leq 2\|w'\|_{L^2(\mathbb{T})}$. Hence
\[
	\int^h_0 |e_h w(x)|^2 \, dx
	\leq (2h/\pi)^2 \|w'\|_{L^2((0,h))}^2.
\]
Similar inequalities can be derived on the intervals $\left(kh,(k+1)h\right)$. Adding these from $k=0$ to $k=1/h-1$ yields (i).

For the proof of (i\hspace{-1pt}i), assume that $w\in H^2(\mathbb{T})$. Since $e_h w(0)=e_h w(h)=0$, by the Poincar\'e--Wirtinger-type inequality in Lemma~\ref{lem:PoinWirt} (i\hspace{-1pt}i) we have
\[
	\int^h_0 |e_h w(x)|^2 \, dx \leq (h/\pi)^4 \int^h_0 \left|(e_h w)''(x)\right|^2 \, dx.
\]
Since $(\pi_h w)''=0$ a.e. on $(0,h)$, this implies
\[
	\int^h_0 |e_h w(x)|^2 \, dx \leq (h/\pi)^4 \int^h_0 |w''(x)|^2 \, dx.
\]
Similarly as before, this yields
\[
	\|e_h w\|^2_{L^2(\mathbb{T})} \leq (h/\pi)^4 \|w''\|^2_{L^2(\mathbb{T})}.
\]

To prove (i\hspace{-1pt}i\hspace{-1pt}i), we again let $w\in H^2(\mathbb{T})$ and consider the integral
\[
	\int^h_0 \left| (e_h w)'(x) \right|^2 \, dx
	= \int^h_0 \left| w'(x) - \frac{w(h) - w(0)}{h} \right|^2 \, dx.
\]
Setting $u=w'$, this integral becomes
\[
	\int^h_0 \left| u(x) - \frac{1}{h} \int^h_0 u(y) \, dy \right|^2 \, dx
	= \int^h_0 \left| \frac{1}{h} \int^h_0 \left(u(x)-u(y)\right) dy \right|^2 \, dx.
\]
This is dominated by
\[
	\int^h_0 \left( \frac{1}{h} \int^h_0 \left| u(y) - u(x) \right|^2 dy \right) \, dx,
\]
by the H\"older or Jensen inequality. Either of these inequalities also implies that
\[
	\left| u(y)-u(x) \right|^2
	= \left| \int^y_x u'(s) \, ds \right|^2
	\leq |x-y| \int^y_x |u'(s)|^2 \, ds
	\leq |x-y| \int_0^h |u'(s)|^2 \, ds,
\]
where the final inequality holds for $x, y\in (0, h)$.
Combining these observations, we have
\begin{align*}
	\int^h_0 \left|(e_h w)'(x) \right|^2 \, dx
	&\leq \frac{1}{h} \int^h_0 \int^h_0 |x-y| \, dx \, dy \int^h_0 |w''(s)|^2 \, ds \\
	&= \tilde C h^2  \int^h_0 |w''(s)|^2 \, ds,
\end{align*}
with, as required,
\begin{align*}
\tilde C &=  \int^1_0 \int^1_0 |x-y| \, dx \, dy = \int_0^1 \left[\int_0^y (y-x) \, dx + \int_y^1 (x-y) \, dx \right] \, dy\\
 &= \int_0^1 \left[ \left(y^2-\frac12 y^2\right) + \left(\frac12 - y - \frac12 y^2 + y^2\right)\right] \, dy = \int_0^1  \left(y^2 - y + \frac12\right) \, dy = \frac13.
\end{align*}
We thus obtain (i\hspace{-1pt}i\hspace{-1pt}i).

Finally, we shall prove the reverse Poincar\'e--Wirtinger-type estimate (v) for $w \in H_h$.
 It suffices to prove this inequality in $(0,h)$ for a linear function $w$, as the full result on $\mathbb{T}$ will then follow by summing the contributions over all the intervals $(kh, (k+1)h)$ as in earlier parts of this proof.  We may assume that $w(x) = ax-b$ for $a>0$, $b\in \mathbb{R}$, by the symmetry provided by the square in the integrand.
 By overall multiplication of $w$ by a constant, we may assume that $a=1$. Then
\[
	\int^h_0 (w'(x))^2 \, dx = h,
\]
while
\begin{align*}
	\int^h_0 w^2(x) \, dx &= \int^h_0 (x-b)^2 \, dx \\
	& \geq \int^h_0 (x-h/2)^2 \, dx
	= 2 \cdot \frac{1}{3} \left(\frac{h}{2}\right)^3 = \frac{1}{12}h^3.
\end{align*}
The inequality follows since $b\mapsto  \int^h_0 (x-b)^2 \, dx$ achieves its mininum value at $b=\frac{h}2$.
Thus we have
\[
	\int^h_0 \left(w'(x)\right)^2 dx \leq (12/h^2) \int^h_0 w^2(x) \, dx.
\]
As is common practice in this proof by now, adding similar inequalities for each interval $\left(kh, (k+1)h \right)$ yields (i\hspace{-1pt}v).
\end{proof}

\begin{proof}[Proof of Theorem \ref{PR}]
\begin{enumerate}
\item[(i)] Let $w \in H^1(\mathbb{T})$. This implies that also $e_h w \in H^1(\mathbb{T})$.
 Since $\pi_h w \in H_h$ and $\overline{P}_h$ is a linear operator, we observe that
\[
	\left( (\overline{P}_h - I)w \right)' = -\left( (\overline{P}_h (e_h w) \right)' + (e_h w)'.
\]
From Lemma~\ref{ER}~(v) we see that
\[
	\left\| \left(\overline{P}_h (e_h w)\right)' \right\|_{L^2(\mathbb{T})}
	\leq \frac{2\sqrt{3}}h \left\| \overline{P}_h (e_h w) \right\|_{L^2(\mathbb{T})}
	\leq \frac{2\sqrt{3}}h \| e_h w\|_{L^2(\mathbb{T})},
\]
where for the last inequality we used that $\overline{P}_h$ is an orthogonal projection.
By Lemma~\ref{ER}~(i\hspace{-1pt}v), we have
\[
	\left\| (e_h w)' \right\|_{L^2(\mathbb{T})} \leq 2 \|w'\|_{L^2(\mathbb{T})}.
\]
Collecting these inequalities, we deduce that
\[
	\left\| \left( (\overline{P}_h - I) w \right)' \right\|_{L^2(\mathbb{T})}
	\leq \frac{2\sqrt{3}}h \left\| e_h w \right\|_{L^2(\mathbb{T})}
	+ 2 \|w' \|_{L^2(\mathbb{T})}.
\]
Estimating the first term of the right-hand side by Lemma~\ref{ER}~(i) yields
\[
	\left\| \left( (\overline{P}_h - I)w \right)' \right\|_{L^2(\mathbb{T})}
	\leq \left( \frac{4\sqrt{3}}{\pi} +2 \right) \|w' \|_{L^2(\mathbb{T})}.
\]
\item[(i\hspace{-1pt}i)] First we assume that $w \in H^2(\mathbb{T})$. As in (i) we have
\[
	\left\| \left((\overline{P}_h - I) w \right)' \right\|_{L^2(\mathbb{T})}
	\leq \frac{2\sqrt{3}}h \left\|(e_h w) \right\|_{L^2(\mathbb{T})}
	+ \left\| (e_h w)' \right\|_{L^2(\mathbb{T})}.
\]
By Lemma~\ref{ER}~(i\hspace{-1pt}i) and~(i\hspace{-1pt}i\hspace{-1pt}i), this estimate yields
\[
	\left\| \left( (\overline{P}_h - I) w \right)' \right\|_{L^2(\mathbb{T})}
	\leq \frac{2\sqrt{3}}h (h/\pi)^2 \| w'' \|_{L^2(\mathbb{T})}
	+ \frac{h}{\sqrt 3} \|w'' \|_{L^2(\mathbb{T})}.
\]
Thus, by part~(i) of this theorem,
\[
	\left\| \left(\overline{P}_h - I) w \right)' \right\|_{L^2(\mathbb{T})} \to 0, \quad\text{as } h\to 0.
\]

Now assume $w\in H^1(\mathbb{T})$ instead and $\delta>0$. By density, there exists a $w_\delta \in H^2(\mathbb{T})$ such that $\left\|(w-w_\delta)' \right\|_{L^2(\mathbb{T})} < \delta$.
 As observed above
\[
	\left\| \left(\overline{P}_h w_\delta - w_\delta\right)' \right\|_{L^2(\mathbb{T})} \to 0,
	\quad\text{as } h \to 0.
\]
Thus,
\begin{align*}
	\left\|\left(\overline{P}_h w - w\right)'  \right\|_{L^2(\mathbb{T})}
	\leq &\left\| \left(\overline{P}_h w_\delta - w_\delta\right)' \right\|_{L^2(\mathbb{T})}
	+ \left\| \left(\overline{P}_h(w - w_\delta) \right)' \right\|_{L^2(\mathbb{T})} \\
	&+ \left\|\left(w - w_\delta\right)' \right\|_{L^2(\mathbb{T})} \\
	&\leq \left\| \left(\overline{P}_h w_\delta - w_\delta\right)' \right\|_{L^2(\mathbb{T})}
	+ (C+1) \left\|\left(w - w_\delta\right)' \right\|_{L^2(\mathbb{T})}.
\end{align*}
Sending $h$ to zero yields
\[
	\overline{\lim_{h \to 0}} \left\| \left(\overline{P}_h w-w\right)' \right\|_{L^2(\mathbb{T})}
	\leq (C+1) \left\|\left(w - w_\delta\right)' \right\|_{L^2(\mathbb{T})} < (C+1)\delta.
\]
Taking $\delta \to 0$ gives the desired convergence.

\item[(i\hspace{-1pt}i\hspace{-1pt}i)] First we assume that $w\in H^1(\mathbb{T})$. From part~(i) of this theorem, we know that also $\overline{P}_h w\in H^1(\mathbb{T})$. By \eqref{eq:projectionmass} we have that $\int_{\mathbb{T}} \left(\overline{P}_h w - w\right) \, dx = 0$, so by the Poincar\'e--Wirtinger inequality for functions with zero `mass' \cite[Theorem 13.27 and Exercise 13.13]{Leoni}, part~(i\hspace{-1pt}i) of this theorem implies that $\overline{P}_h w \to w$ strongly in $L^2(\mathbb{T})$.

Now let $w\in L^2(\mathbb{T})$. By density of $H^1(\mathbb{T})$ in $L^2(\mathbb{T})$, for all $\delta>0$ there exists a $w_\delta \in H^1(\mathbb{T})$ such that $\left\|w-w_\delta\right\|_{L^2(\mathbb{T})} < \delta$.  From \eqref{eq:nonexpansive} it follows that
\begin{align*}
\left\|\overline{P}_h w- w\right\|_{L^2(\mathbb{T})} &\leq \left\|\overline{P}_h w_\delta- w_\delta\right\|_{L^2(\mathbb{T})} + \left\|\overline{P}_h w_\delta - \overline{P}_h w\right\|_{L^2(\mathbb{T})} + \left\|w_\delta - w\right\|_{L^2(\mathbb{T})}\\
 &\leq \left\|\overline{P}_h w_\delta- w_\delta\right\|_{L^2(\mathbb{T})} + 2 \left\|w_\delta - w\right\|_{L^2(\mathbb{T})} \\
 &< \left\|\overline{P}_h w_\delta- w_\delta\right\|_{L^2(\mathbb{T})} + 2\delta.
\end{align*}
By the first part of the proof of (i\hspace{-1pt}i\hspace{-1pt}i) we have that $\lim_{h\to 0} \left\|\overline{P}_h w_\delta - w_\delta\right\|_{L^2(\mathbb{T})} = 0$. Hence, if we first take the limit superior for $h\to 0$ and then the limit for $\delta \to 0$, the required result follows.

\item[(i\hspace{-1pt}v)] Let $(h_j)$ be a sequence of positive real numbers which converges to zero and let $(h_{j_k})$ be a subsequence. By Morrey's inequality \cite[Theorem 11.34]{Leoni}, \eqref{eq:nonexpansive}, and part~(i) of this theorem, there exists some constanst $\tilde C_1, \tilde C_2 >0$ such that
\[
\|\overline{P}_{h_{j_k}} w\|_{L^\infty(\mathbb{T})} \leq \tilde C_1 \|\overline{P}_{h_{j_k}} w\|_{H^1(\mathbb{T})} \leq \tilde C_2 \|w\|_{H^1(\mathbb{T})},
\]
thus the sequence $\left(\overline{P}_{h_{j_k}}w\right)$ is uniformly bounded.

Let $x, y \in \mathbb{T}$ and assume $[x,y]\subset \mathbb{T}$ with appropriate periodicity taken into account. By H\"older's inequality and part~(i) of this theorem we have
\begin{align*}
\left|\overline{P}_{h_{j_k}}w(x) - \overline{P}_{h_{j_k}}w(y)\right| &= \left|\int_{[x,y]} \overline{P}_{h_{j_k}}w'(x) \, dx \right| \\
&\leq \int_{[x,y]} \left|\overline{P}_{h_{j_k}}w'(x)\right| \,dx \leq \int_{\mathbb{T}} \left|\overline{P}_{h_{j_k}}w'(x)\right| \, dx\\
&\leq \|1\|_{L^2(\mathbb{T})} \left\|\overline{P}_{h_{j_k}}w'\right\|_{L^2(\mathbb{T})} = \left\|\overline{P}_{h_{j_k}}w'\right\|_{L^2(\mathbb{T})} \leq \left\|w'\right\|_{L^2(\mathbb{T})}.
\end{align*}
Thus the sequence $\left(\overline{P}_{h_{j_k}}w\right)$ is equicontinuous. By the Arzel\`a--Ascoli theorem, it follows that there exists a $\overline w \in C(\mathbb{T})$ and a subsubsequence $\left(\overline{P}_{h_{j_{k_l}}}w\right)$ such that $\overline{P}_{h_{j_{k_l}}}w \to \overline w$ uniformly as $l\to\infty$. Hence, since $\mathbb{T}$ has finite measure, we also have $\overline{P}_{h_{j_{k_l}}}w \to \overline w$ in $L^2(\mathbb{T})$ and thus, by part~(i\hspace{-1pt}i\hspace{-1pt}i), $\overline w = w$. Since every subsequence of $\left(\overline{P}_{h_j} w\right)$ has a further subsubsequence which converges uniformly to $w$, it follows that $\left(\overline{P}_{h_j} w\right)$ itself also converges to $w$ uniformly. From this we conclude the desired result.
\end{enumerate}
\end{proof}

\subsection{Proof of Theorem~\ref{CAC2}}\label{sec:applicationabstract}

We shall prove Theorem \ref{CAC2} as an application of Theorem \ref{AB2}.

\begin{proof}[Proof of Theorem \ref{CAC2}]
Let $T>0$. In the notation of Theorem~\ref{AB2}, we set $M_h = \overline{L}^2_h$, $M=L^2(\mathbb{T})$, $i_h=I_h$, $p_h=P_h$, $\Phi_h = \varphi^h_{AC}$, and $\Phi = \varphi_{AC}$. Since $\overline{L}^2_h$ is a Hilbert space, it is closed and thus $\overline{\mathcal{D}\left(\varphi_{AC}^h\right)} = \overline{L}_h^2$. As derived at the end of Section~\ref{sec:additionalsetup}, we also have $\overline{\mathcal{D}\left(\varphi_{AC}\right)} = L^2(\mathbb{T})$. By our discussion in Section~\ref{sec:thefunctionalsAC} we know that $\varphi_{AC}^h$ and $\varphi_{AC}$ satisfy condition ($-\alpha$-AGS). In Section~\ref{sec:projection} we established that condition \eqref{P} is satisfied. It remains to show that conditions \eqref{U2}, \eqref{eq:PhiPsi}, and \eqref{eq:Phiinequality} are satisfied.

First we shall prove that \eqref{U2} holds. Let $w\in L^2(\mathbb{T})$. Without loss of generality we may assume that $\varphi_{AC}(w) < +\infty$ and thus $w\in H^1(\mathbb{T})$.
 By \eqref{eq:phiD} and Theorem~\ref{PR}~(i\hspace{-1pt}i), we see that
\[
	\lim_{h\to 0} \varphi_D^h(P_h w) = \lim_{h\to 0} \varphi_D(I_h P_h w) = \lim_{h \to 0} \varphi_D (\overline{P}_h w) = \varphi_D(w).
\]
By Theorem~\ref{PR}~(i\hspace{-1pt}v) we have that $P_h w \to w$ uniformly in $\mathbb{T}$, thus
\[
	\lim_{h \to 0} \varphi_W (\overline{P}_h w) = \varphi_W(w).
\]
We claim that
\begin{equation}\label{eq:varphiWclaim}
\lim_{h\to 0} \left|\varphi_W^h(P_h w) - \varphi_W(\overline{P}_h w)\right| = 0.
\end{equation}
Combining this with the limit above gives
\[
\lim_{h\to 0} \varphi_W^h(P_h w) = \varphi_W(w),
\]
which together with the convergence of $\varphi_D^h(P_h w)$ above establishes \eqref{U2}. To prove the claim, we recall that $\displaystyle \varphi_W^h(P_h w) = \varphi_W^h(\overline{\pi}_h \overline{P}_h w) = \sum_{k=0}^{\frac1h-1} h W\left(\overline{P}_h w (kh)\right)$ and thus
\begin{align}
\left|\varphi_W^h(P_h w) - \varphi_W(\overline{P}_h w)\right| & \leq \sum_{k=0}^{\frac1h-1} \left| h W\left(\overline{P}_h w(kh)\right) - \int_{kh}^{(k+1)h} W\left(\overline{P}_h w(x)\right)\, dx \right|\notag\\
&= \sum_{k=0}^{\frac1h-1} \left| \int_{kh}^{(k+1)h} \left[W\left(\overline{P}_h w(kh)\right) - W\left(\overline{P}_h w(x)\right)\right] \, dx\right|\notag\\
&\leq \sum_{k=0}^{\frac1h-1} \int_{kh}^{(k+1)h} \left|W\left(\overline{P}_h w(kh)\right) - W\left(\overline{P}_h w(x)\right)\right| \, dx.\label{eq:varphiWinequality}
\end{align}
For the integrand on the right-hand side we estimate
\begin{align*}
\left|W\left(\overline{P}_h w(kh)\right) - W\left(\overline{P}_h w(x)\right)\right| \leq &\left|W\left(\overline{P}_h w(kh)\right) - W\left(w(kh)\right)\right| \\
&+ \left|W\left(w(kh)\right) - W\left(w(x)\right)\right|\\
&+ \left|W\left(w(x)\right) - W\left(\overline{P}_h w(x)\right)\right|,
\end{align*}
where we remember that in the integral $x\in [kh, (k+1)h]$ and thus $|x-kh| \leq h$. Let $\eta>0$. Since $w\in H^1(\mathbb{T})$, we know by the Sobolev embedding that $w\in C(\mathbb{T})$ and thus, since $W$ is continuous, $W\circ w \in C(\mathbb{T})$. Since $\mathbb{T}$ is compact this means that $W\circ w$ is uniformly continuous. Hence there exists a $h_1>0$ such that, if $0<h<h_1$, then $\left|W\left(w(kh)\right) - W\left(w(x)\right)\right|<\eta/3$. Moreover, since $\overline{P}_h w \to w$ uniformly and since $W$ is continuous, we have that $W(\overline{P}_h w) \to W(w)$ uniformly. Hence, there exists a $h_2>0$ such that, if $0<h<h_2$, then $\left|W\left(w(kh)\right) - W\left(w(x)\right)\right| < \eta/3$ and $\left|W\left(\overline{P}_h w(kh)\right) - W\left(\overline{P}_h w(x)\right)\right| < \eta/3$. Defining $h^*:=\min(h_1, h_2)$ we find that, if $0<h<h^*$, then $\left|W\left(\overline{P}_h w(kh)\right) - W\left(\overline{P}_h w(x)\right)\right| < \eta$. Substituting this into \eqref{eq:varphiWinequality}, we find, for $h\in (0, h^*)$,
\[
\left|\varphi_W^h(P_h w) - \varphi_W(\overline{P}_h w)\right| < \sum_{k=0}^{\frac1h-1} \int_{kh}^{(k+1)h} \eta \, dx = \frac1h \cdot h \cdot \eta = \eta,
\]
which proves the claim in \eqref{eq:varphiWclaim} and thereby concludes the proof of \eqref{U2}.

Next we turn our attention to establishing the existence of a $\delta>0$ and a nonnegative function $\Psi: L^2(\mathbb{T}) \to \mathbb{R} \cup \{+\infty\}$, such that $\Psi(u(\cdot)) \in L^1(0,T)$ and \eqref{eq:PhiPsi} is satisfied. In fact, we will show that for our choice of $\Psi$, \eqref{eq:PhiPsi} is satisfied for all $h\in \mathfrak{H}$, so that any choice of $\delta>0$ is valid.

By Proposition~\ref{UE} we know that for all $t>0$, $u(t)\in \mathcal{D}(\varphi_{AC}) \subset H^1(\mathbb{T})$. Therefore, we can define $\Psi:=+\infty$ on $L^2(\mathbb{T})\setminus H^1(\mathbb{T})$. Now let $w\in H^1(\mathbb{T})$. By \eqref{eq:phiD} and Theorem~\ref{PR}~(i), we have
\[
	\varphi_D^h(P_h w) = \varphi_D (\overline{P}_h w) \leq \frac{C^2}2 \| w' \|_{L^2(\mathbb{T})}^2.
\]
Recalling that $W(x)=\alpha(x^2-1)^2/4$, we observe that
\[
	\varphi_W (\overline{P}_h w) \leq \frac{\alpha}{4} \left(\int_{\mathbb{T}} |\overline{P}_h w|^4 dx + 1\right).
\]
By Theorem~\ref{PR}~(i), the one-dimensional Sobolev (Morrey) inequality, and \eqref{eq:nonexpansive}, there exists a constant $c_0>0$ such that
\begin{align*}
	\| \overline{P}_h w \|_{L^\infty(\mathbb{T})} &\leq c_0 \left( \left\| (\overline{P}_h w)' \right\|_{L^2(\mathbb{T})} + \| \overline{P}_h w \|_{L^2(\mathbb{T})} \right) \leq c_0 \left( C \left\| w' \right\|_{L^2(\mathbb{T})} + \| w \|_{L^2(\mathbb{T})} \right) \\
	&\leq c_1 \| w \|_{H^1(\mathbb{T})},
\end{align*}
with $c_1 := c_0 C$.

 By the interpolation inequality for $L^p$ spaces \cite[Theorem 2.11]{AF} (a special case of the Gagliardo--Nirenberg interpolation inequality), \eqref{eq:nonexpansive}, and the inequality above, we have
\begin{equation}\label{eq:L4ineq}
	\| \overline P_h w \|_{L^4(\mathbb{T})} \leq \| \overline P_h w \|_{L^2(\mathbb{T})}^{1/2} \| \overline P_h w \|_{L^\infty(\mathbb{T})}^{1/2}\leq C_0 \| w \|_{L^2(\mathbb{T})}^{1/2} \| w \|_{H^1(\mathbb{T})}^{1/2},
\end{equation}
with $C_0 := c_1^{1/2}$.

 Thus
\[
	\varphi_W (\overline P_h w) \leq C_1 \| w \|^2_{L^2(\mathbb{T})} \| w \|^2_{H^1(\mathbb{T})} + \frac\alpha4,
\]
where $C_1 := \alpha C^4_0 / 4$. We note that $C_1$ is independent of $h$ and $w$.

Using the explicit form of $W$, Lemma~\ref{IP} (or, equivalently, \eqref{eq:Lminequality} with $m=2$), \eqref{eq:Lminequality2} with $m=4$, and \eqref{eq:L4ineq}, we find
\begin{align*}
\varphi_W^h(P_h w) - \varphi_W(\overline{P}_h w) &= \varphi_W^h(P_h) - \varphi_W(I_h P_h w)\\
&= \frac\alpha4 \left(\|P_h w\|_{L^4_h}^4 - \|I_h P_h w\|_{L^4(\mathbb{T})}^4 - 2 \|P_h w\|_h^2 + 2 \|I_h P_h w\|_{L^2(\mathbb{T})}^2\right)\\
&\leq \frac\alpha2 \left(\|P_h w\|_{L^4_h}^4 - \|I_h P_h w\|_{L^4(\mathbb{T})}^4 \right) \leq \alpha \|I_hP_h w\|_{L^4(\mathbb{T})}^4 =   \alpha \|\overline P_h w\|_{L^4(\mathbb{T})}^4\\
&\leq C_2  \| w \|_{L^2(\mathbb{T})}^2 \| w \|_{H^1(\mathbb{T})}^2,
\end{align*}
with $C_2 := \alpha C_0^4$. It thus follows that
\[
\varphi_W^h(P_h w) \leq \varphi_W(\overline P_h w) +  C_2  \| w \|_{L^2(\mathbb{T})}^2 \| w \|_{H^1(\mathbb{T})}^2 \leq C_3 \| w \|_{L^2(\mathbb{T})}^2 \| w \|_{H^1(\mathbb{T})}^2 + \frac\alpha4,
\]
where $C_3 := C_1 + C_2 = \frac54 \alpha C_0^4$.

We define, for all $w\in L^2(\mathbb{T})$,
\[
	\Psi(w) := \begin{cases}
\frac{C^2}2 \| w' \|_{L^2(\mathbb{T})}^2 + C_3 \| w \|_{L^2(\mathbb{T})}^2 \| w \|^2_{H^1(\mathbb{T})} + \frac\alpha4, &\text{if } w\in H^1(\mathbb{T}),\\
+\infty, &\text{otherwise}.
\end{cases}
\]
From the inequalities above, it follows that for all $w\in L^2(\mathbb{T})$ and all $h\in\mathfrak{H}$ the inequality in \eqref{eq:PhiPsi} is satisfied. Moreover, since $\alpha>0$, $\Psi$ is nonnegative.

As we observed earlier, by Proposition~\ref{UE} we have that for all $t>0$, $u(t)\in \ H^1(\mathbb{T})$. Thus, for all $t>0$, $t\mapsto \|u(t)\|_{L^2(\mathbb{T})}$ and $t\mapsto \|\nabla u(t)\|_{L^2(\mathbb{T})}$ are measurable and have finite value, so
\[
	\int^T_0 \Psi \left( u(t) \right) dt < \infty.
\]

It remains to prove that there exists a function $\varepsilon: (0,1) \to (0,\infty)$ such that $\varepsilon(h) \to 0$ as $h \to 0$ and \eqref{eq:Phiinequality} is satisfied, i.e., for all $t\in [0, T]$,

\begin{equation} \label{GI}
	\varphi_{AC} \left( I_h U^h(t) \right) \leq
	\varphi^h_{AC} \left( U^h(t) \right) + \varepsilon(h).
\end{equation}

Let $t\in [0, T]$. By \eqref{eq:phiD} we have
$
	\varphi_D ( I_h U^h(t) ) = \varphi^h_D ( U^h(t) ),
$
and \eqref{eq:Lminequality} with $m=4$ shows that
$
	\| I_h U^h(t) \|^4_{L^4(\mathbb{T})} \leq \| U^h(t) \|^4_{L^4_h}.
$
Combining this we the explicit form of $W$, we observe that
\begin{align*}
	\varphi_{AC} ( I_h U^h(t) ) &\leq \varphi^h_{AC} ( U^h(t) )
	+ \frac\alpha2 \left( \| U^h(t) \|^2_h - \| I_h U^h(t) \|^2_{L^2(\mathbb{T})} \right)\\
 &= \frac{\alpha h^2}{12} \| \nabla U^h(t) \|^2_h \leq \frac\alpha4  e^{2\alpha t} \| \nabla u^h_0 \|_h^2,
\end{align*}
where for the equality we used \eqref{eq:u-I_hu} and the last inequality follows from Lemma~\ref{MDAC}.

If we define
\[
	\varepsilon(h) := \frac\alpha4 e^{2\alpha T} \| \nabla u^h_0 \|^2 h^2,
\]
then \eqref{GI} follows by the assumption that $h \| \nabla u^h_0 \|_h \to 0$ as $h\to 0$.

We are now able to apply Theorem \ref{AB2} to conclude the desired convergence.

\end{proof}

\section{Conclusions and future work}\label{sec:conclusions}

In this paper we proved discrete-to-continuum convergence for total variation flow on a discretised $n$-dimensional torus and for Allen--Cahn flow on a discretised $1$-dimensional torus. Two potential generalisations immediately suggest themselves: Allen--Cahn flow convergence on a discrete torus of higher dimension and, for both flows, convergence of the flows on other graphs.

Our proof of convergence for the Allen--Cahn flow is not trivially extended to higher dimensions, since it makes intensive use of the properties of the operator $\Gamma$, which is determined by the linear interpolation method by which discrete functions are associated with continuum functions. If $n$-linear interpolation (bilinear, trilinear, etc.) is used, a different operator $\Gamma$ (and corresponding embedding and projection operators) will need to be constructed. In higher dimensions, the space $H^1$ also loses a lot of its regularity.

Other graphs, besides the discrete torus, that are being considered by Samuel Mercer and the second author, are random geometric graphs, obtained by sampling points from Euclidean space (or possibly from a differentiable manifold) according to some probability density in order to construct the vertex set and subsequently connecting nodes with (potentially weighted) edges depending on their pairwise distances. The identification of discrete functions with continuum functions will then require a method different from the constant and linear interpolation which we used in this paper, such as the method based on optimal transport techniques that was introduced in \cite{GTS}.

Other flows can also be considered. For example, in the field of machine learning there is much current interest in proving consistency of methods through continuum limits of discrete dynamics.

Another possible direction for future research is the relationship between the continuum limit and the singular limit, for those dynamics that have an interesting singular limit. For example, if the potential term in the Allen--Cahn equation of \eqref{AC} is rescaled by a factor $\varepsilon^{-2}$, then in the sharp interface limit $\varepsilon\downarrow 0$ solutions evolve according to flow by mean curvature \cite{dMS90,Chen92,EvansSonerSouganidis92,dMS95,Soner97I,Soner97II,AHM08,FisherLauxSimon20}. It cannot be expected that the continuum limit and singular limit commute, since in the limit $\varepsilon\downarrow 0$ the graph-based Allen--Cahn dynamics decouples per vertex and simply evolves into the well of $W$ that is nearest to the initial condition at the node \cite{vGGOB}. A similtaneous limit, taking both $\varepsilon\downarrow 0$ and $h\downarrow 0$ (or, in a more general setting, the number of nodes $|V|\to\infty$), could potentially lead to nontrivial dynamics if $\varepsilon$ scales in the correct way with $h$. A related example of such a simultaneous limit is the convergence of graph Merriman--Bence--Osher dynamics to mean curvature flow, that has been established recently in \cite{LauxLelmi22MBO}.

\appendix

\section{A proof of \eqref{eq:embeddedTV}}\label{app:TVforpiecewiseconstant}

This section culminates with the proof of \eqref{eq:embeddedTV} in Lemma~\ref{lem:proofofembeddedTV}. Along the way we prove some other useful results.

\begin{lemma}\label{lem:weak*convergence}
 Let $p\in [1,\infty)$. Assume $g, \tilde g\in L^p(\mathbb{T}^n)$ and let $\{g^j\}_{j\in\mathbb{N}}$ be a sequence in $L^p(\mathbb{T}^n)$. If $g^j \to g$ in $L^p(\mathbb{T}^n)$ and $g^j \rightharpoonup^* \tilde g$ in $L^\infty(\mathbb{T}^n)$ as $j\to\infty$, then $g=\tilde g$ a.e. on $\mathbb{T}^n$.

Consequently, if $w\in L^p(\mathbb{T}^n)$ and $\{w^j\}_{j\in \mathbb{N}}$ is a sequence in $L^p(\mathbb{T}^n)$ which is uniformly bounded in $L^\infty(\mathbb{T}^n)$ and such that $w^j \to w$ in $L^p(\mathbb{T}^n)$, then $w^j \rightharpoonup^* w$ in $L^\infty(\mathbb{T}^n)$.
\end{lemma}
\begin{proof}
 Let $U\subset \mathbb{T}^n$ be such that $g-\tilde g \geq 0$
on $U$ and $g-\tilde g<0$
on $U^c := \mathbb{T}^n\setminus U$. We note that as preimages of the measurable sets $[0,\infty)$ and $(-\infty, 0)$, respectively, $U$ and $U^c$ are also measurable. Define $\chi_U = 1$
on $U$ and $\chi_U = 0$
on $U^c$. Then $2\chi_U-1 \in L^1(\mathbb{T}^n) \cap L^{p^*}(\mathbb{T}^n)$, where $p^*$ is defined via $1/p+1/p^*=1$. We compute
 \begin{align*}
  \int_{\mathbb{T}^n} |g(x)-\tilde g(x)|\, dx =& \int_{\mathbb{T}^n} \left(2\chi_U(x)-1\right) \left(g(x)-\tilde g(x)\right) \, dx\\
  =& \int_{\mathbb{T}^n} \left(2\chi_U(x)-1\right) \left(g(x)-g^j(x)\right)\,dx \\
  &+ \int_{\mathbb{T}^n} \left(2\chi_U(x)-1\right) \left(g^j(x)-\tilde g(x)\right)\,dx.
 \end{align*}
The first term on the right-hand side converges to zero as $j\to\infty$, by H\"older's inequality and $g^j \to g$ in $L^p(\mathbb{T}^n)$. The second term converges to zero, since $g^j \rightharpoonup^* \tilde g$ in $L^\infty(\mathbb{T}^n)$. This concludes the proof of the first part.

For the proof of the second part, we note that by the Banach--Alaoglu theorem each bounded sequence in $L^\infty(\mathbb{T}^n)$ has a weakly*-convergent subsequence.  In particular each subsequence of $\{w^j\}$ has a further subsubsequence which converges weakly*. By the first part of this result each such subsubequence has the same limit, $w$. Hence $\{w^j\}$ converges weakly* in $L^\infty(\mathbb{T}^n)$ to $w$.
\end{proof}

For $1\leq p \leq \infty$ we denote by $L^p(\mathbb{T}; \mathbb{R}^n)$ the space of integrable functions $u: \mathbb{T}^n \to \mathbb{R}^n$ for which $\int_{\mathbb{T}^n} |u(x)|^p \, dx < \infty$ (if $p<\infty$), where $|\cdot|$ denotes the Euclidean norm in $\mathbb{R}^n$, or for which, for almost all $x\in \mathbb{T}^n$, $|g(x)|_\infty < \infty$ (if $p=\infty$).

\begin{lemma}\label{lem:weak*derivatives}
 Let $g\in L^\infty(\mathbb{T}^n; \mathbb{R}^n)$ be such that, for all $i\in [n]$, $\frac{\partial g_i}{\partial x_i} \in L^\infty(\mathbb{T}^n)$ as weak partial derivatives and, for almost all $x\in \mathbb{T}^n$, $|g(x)|_\infty \leq 1$. Then there exists a sequence $\{g^j\}_{j\in \mathbb{N}}$ in $C^\infty(\mathbb{T}^n; \mathbb{R}^n)$ such that, as $j\to\infty$, $g^j \rightharpoonup^* g$ in $L^\infty(\mathbb{T}^n; \mathbb{R}^n)$ and, for all $i\in [n]$, $\frac{\partial g_i^j}{\partial x_i} \rightharpoonup^* \frac{\partial g_i}{\partial x_i}$. Moreover, for all $j\in \mathbb{N}$ and for almost all $x\in \mathbb{T}^n$, $|g^j(x)|_\infty \leq 1$.
\end{lemma}
\begin{proof}
This proof is a variant of the proofs in \cite[Proposition 3.3]{EsedogluOsher04}\cite[Lemma A.3]{CvGO11}.

 Let $\eta \in C^\infty(\mathbb{T}^n)$ be the standard mollifier (see for example \cite[Appendix C.4]{Evans}). In particular, $\eta \geq 0$ and $\int_{\mathbb{T}^n} \eta(x)\, dx=1$. Define, for all $j\in \mathbb{N}$, $\eta^j(x):= j^n \eta(jx)$ and $g^j := g \ast \eta^j$. We note that $\eta^j \geq 0$ and $\int_{\mathbb{T}^n} \eta^j(x)\, dx = 1$. Moreover, since the convolution preserves the periodicity, $g^j$ is well-defined on $\mathbb{T}^n$ and by a standard result \cite{Evans}, $g^j \in C^\infty(\mathbb{T}^n)$.

Let $p\in [1, \infty)$. Since $g \in L^\infty(\mathbb{T}^n; \mathbb{R}^n)$ and $|\mathbb{T}^n|<\infty$, also $g\in L^p(\mathbb{T}^n; \mathbb{R}^n)$. Again by a standard result \cite{Evans}, $g^j \to g$ in $L^p(\mathbb{T}^n; \mathbb{R}^n)$. Moreover, since
\[
 \frac{\partial g^j_i}{\partial x_i} = \frac{\partial g_i}{\partial x_i} \ast \eta^j
\]
and $\frac{\partial g_i}{\partial x_i} \in L^\infty(\mathbb{T}^n) \subset L^p(\mathbb{T}^n)$, by the same standard result as above, $\frac{\partial g^j_i}{\partial x_i} \to \frac{\partial g_i}{\partial x_i}$ in $L^p(\mathbb{T}^n)$.

By Lemma~\ref{lem:weak*convergence} it now suffices to prove uniform (in $j$) $L^\infty$ bounds on $g^j$ and $\frac{\partial g^j_i}{\partial x_i}$. For the former we use $|g(x)|_\infty \leq 1$ to compute
\[
 |g^j_i(x)| \leq j^n \int_{\mathbb{T}^n} \eta(jy) \, dy = 1.
\]
For the latter,
\[
\left|\frac{\partial g^j_i}{\partial x_i}(x)\right| = \left| \int_{\mathbb{T}^n} \frac{\partial g_i}{\partial x_i}(y) \eta^j(x-y) \, dy\right| \leq \int_{\mathbb{T}^n} \eta^j(x-y) \,dy \left\|\frac{\partial g_i}{\partial x_i}\right\|_{L^\infty(\mathbb{T}^n)} = \left\|\frac{\partial g_i}{\partial x_i}\right\|_{L^\infty(\mathbb{T}^n)}.
\]
\end{proof}

Using Lemma~\ref{lem:weak*derivatives} we deduce (similar to \cite[Corollary 3]{CvGO11}) that the regularity conditions on the admissible vector fields $g$ in the definition of anisotropic total variation can be relaxed:
\begin{multline}\label{eq:TVrewritten}
 \int_{\mathbb{T}^n} |Du|_{l^1} = \sup \left\{\int_{\mathbb{T}^n} u(x) \operatorname{div} g(x) \, dx \Bigm| g\in L^\infty(\mathbb{T}^n; \mathbb{R}^n),\right.\\
 \left. \forall i \in [n] \ \frac{\partial g_i}{\partial x_i} \in L^\infty(\mathbb{T}^n), \ \ \text{and for a.e. } x\in \mathbb{T}^n \ |g(x)|_\infty \leq 1\right\}.
\end{multline}
Moreover, since Lemma~\ref{lem:weak*derivatives} tells us that each of the weak partial derivatives $\frac{\partial g_i}{\partial x_i}$ converges separately, the summation formula \eqref{eq:sumof1DTV} also holds when the suprema in \eqref{eq:sumof1DTV} are taken over all $g=(g_1, \ldots, g_n) \in L^\infty(\mathbb{T}^n; \mathbb{R}^n)$ that satisfy the conditions in \eqref{eq:TVrewritten}. It follows that identity \eqref{eq:TVrewritten} yields
\begin{equation}\label{eq:decomposition}
	\int_{\mathbb{T}^n} |Du|_{\ell^1} = \sum^n_{i=1} \int_{\mathbb{T}^n} |D_{x_i}u|,
\end{equation}
if we define
\begin{multline}
	\int_{\overline{U}} |D_{x_i}u| := \sup \Biggl\{ \int_U u(x) \frac{\partial f}{\partial x_i}(x) dx \Biggm|
	f \in L^\infty(U),\ \frac{\partial f}{\partial x_i} \in L^\infty(U),\ \text{and} \\
	\text{for a.e.}\ x \in U \ \left| f(x) \right| \leq 1 \Biggr\},\label{eq:TVxi}
\end{multline}
for subsets $U \subset \mathbb{T}^{n}$ that are open (in the Euclidean topology on $\mathbb{T}^n$) and functions $u\in L^1(U)$. We note that we have chosen the notation $\int_{\overline{U}} |D_{x_i}u|$ (rather than $\int_U |D_{x_i}u|$) to emphasise that this is not the standard total variation, since its behaviour at the boundary of $U$ is different. For example, if $U$ is an interval and proper subset of $\mathbb{T}$ and $u$ is constantly equal to $c\in \mathbb{R}$ on $U$, then $\int_{\overline{U}} |D_{x_1}u|=2c$, whereas the total variation of $u$ on $U$ is $\int_U |Du|_{\ell^1} = 0$. In general, if $U$ has smooth boundary, we can interpret $\int_{\overline{U}} |D_{x_i}u|$ as the total variation on $\mathbb{T}^n$ `in the direction of $x_i$' of the function $\overline u$, that is defined by $\overline u := u$ on $U$ and $\overline u := 0$ on $\mathbb{T}^n\setminus U$.

For an open subset $U\subset \mathbb{T}^n$, we say that the boundary of $U$ {\it is parallel to the $x_i$-axis} if $\partial U$ (i.e., the topological boundary of $U$ as subset of $\mathbb{T}^n$) is a nonempty subset of the union of countably many $n-1$-dimensional hyperplanes in $\mathbb{R}^n$ (where we interpret $\mathbb{T}^n$ as a subset of $\mathbb{R}^n$ via identification with the hypercube $\Omega_h$; see Section~\ref{sec:thegraphs}) whose normal vectors are perpendicular to the $x_i$-axis.

We say a function $u: U \to \mathbb{R}$ is {\it independent of $x_j$} if the distributional derivative $\frac{\partial u}{\partial x_j}$ equals zero.

\begin{lemma} \label{A}
Let $i\in[n]$ and $n\geq2$.
\begin{itemize}
\item[(i)] If the boundary of $U$ is parallel to the $x_i$-axis, then, for all $u\in L^1(\mathbb{T}^n)$,
\begin{equation}\label{eq:paralllelboundary}
	\int_{\mathbb{T}^n} |D_{x_i}u| = \int_{\overline{U}} |D_{x_i}u| + \int_{\mathbb{T}^n\backslash U} |D_{x_i}u|.
\end{equation}

\item[(i\hspace{-1pt}i)] Let $j\in [n]$, with $j\neq i$. If $u\in L^1(U)$ is independent of $x_j$, then
\begin{multline*}
	\int_{\overline{U}} |D_{x_i}u| = \sup \Biggl\{ \int_U u(x) \frac{\partial f}{\partial x_i}(x) \, dx \Biggm|
	f \in L^\infty(U),\ \frac{\partial f}{\partial x_i} \in L^\infty(U),\\
	\text{for a.e.}\ x \in U \ \left| f(x) \right| \leq 1, \text{and } f \text{ is independent of } x_j \Biggr\}.
\end{multline*}
\end{itemize}
\end{lemma}
\begin{proof}
\begin{itemize}
\item[(i)] By definition, the inequality $\leq$ holds for every open set $U$.
 Assume that the boundary of $U$ is parallel to the $x_i$-axis. For notational convenience, define $V:=\mathbb{T}^n\setminus\overline{U}$. Let $f_U\in L^\infty(U)$ with $\frac{\partial f_U}{\partial x_i} \in L^\infty(U)$ and, for a.e. $x\in U$, $|f_U(x)|\leq 1$, and let $f_V\in L^\infty(V)$ with $\frac{\partial f_V}{\partial x_i} \in L^\infty(V)$ and, for a.e. $x\in V$, $|f_V(x)|\leq 1$. Let $f: \mathbb{T}^n \to \mathbb{R}$ be such that $f|_U=f_U$ and $f|_V=f_V$. Then $f\in L^\infty(\mathbb{T}^n)$ and, since the boundary of $U$ is parallel to the $x_i$-axis, $\frac{\partial f_U}{\partial x_i} \in L^\infty(U)$ and $\frac{\partial f_V}{\partial x_i} \in L^\infty(V)$. Moreover, for a.e. $x\in \mathbb{T}^n$, $\left|f(x)\right|\leq1$, Thus, if $f_U$ and $f_V$ are admissible functions in the suprema on the right-hand side of \eqref{eq:paralllelboundary}, then $f$ is admissible in the supremum on the left-hand side. This yields the reverse inequality of the identity\footnote{We note that this may not hold in general if the boundary of $U$ is not parallel to the $x_i$-axis.}.

\item[(i\hspace{-1pt}i)] The inequality $\geq$ follows immediately, since the admissible set of functions in the supremum on the right-hand side is a subset of the admissible set in \eqref{eq:TVxi}. To prove the inequality $\leq$, we show that we can replace any admissible function $f$ by another admissible function that is independent of $x_j$, without changing the value of $\int_U u(x) \frac{\partial f}{\partial x_i}(x) \, dx$.

Without loss of generality, we take $j=1$ and $i>1$. Let $f$ be admissible in the supremum in \eqref{eq:TVxi}. We define the $x_1$-average of $f$ ---which is independent of $x_1$--- for $(x_1, \tilde x)\in U$ as
\[
	\overline{f}(x_1,\tilde x) := \frac{1}{|U_{\tilde x}|} \int_{U_{\tilde x}} f(y, \tilde x) \, dy,
\]
where $U_{\tilde x}:=\left\{x_1\in\mathbb{T} \bigm| (x_1,\tilde x)\in U\right\}$ and $|U_{\tilde x}|$ is the one-dimensional Lebesgue measure of $U_{\tilde x}$. We note that, for all $\tilde x\in \mathbb{T}^{n-1}$, $U_{\tilde x}$ is open in $\mathbb{T}$, since $U$ is open in $\mathbb{T}^n$. In particular, if $(x_1, \tilde x)\in U$, then $U_{\tilde x}\neq \emptyset$ and thus $|U_{\tilde x}|>0$. Hence $\overline{f}$ is well-defined. We define $X:=\left\{\tilde x\in \mathbb{T}^{n-1} \bigm| U_{\tilde x}\neq \emptyset\right\}$.

By differentiating under the integral sign \cite[Theorem 7.40 (or Section 11.12)]{Jones82}\footnote{For a more directly applicable result published in the Japanese language, we refer to \cite[Theorem 19.4]{Ito65}.},
we see that, for a.e. $(x_1, \tilde x)\in U$,
\begin{equation}\label{eq:derivativeaverage}
	\frac{\partial\overline{f}}{\partial x_i}(x_1,\tilde x) = \frac{\overline{\partial f}}{\partial x_i}(x_1, \tilde x).
\end{equation}
Hence $\frac{\overline{\partial f}}{\partial x_i}\in L^\infty(U)$. Moreover, we observe that the value of $\displaystyle \int_U u(x) \frac{\partial f}{\partial x_i}(x) \, dx$ remains unchanged if we replace $f$ by $\overline{f}$:
\begin{align}
	\int_U u(x) \frac{\partial \overline{f}}{\partial x_i}(x)  \, dx &= \int_X \int_{U_{\tilde x}} u(x_1, \tilde x) \frac{\overline{\partial f}}{\partial x_i}(x_1,\tilde x) \, dx_1 \, d\tilde x\notag\\
	&= \int_X \int_{U_{\tilde x}} u(x_1, \tilde x) \frac{1}{|U_{\tilde x}|}  \int_{U_{\tilde x}} \frac{\partial f(y, \tilde x)}{\partial x_i} \, dy \, dx_1 \, d\tilde x\notag\\
	&= \int_X \frac{1}{|U_{\tilde x}|}  \int_{U_{\tilde x}} u(y, \tilde x) \frac{\partial f(y, \tilde x)}{\partial x_i} \, dy \int_{U_{\tilde x}} \, dx_1 \, d\tilde x\notag\\
	&= \int_X \int_{U_{\tilde x}} u(y, \tilde x) \frac{\partial f(y, \tilde x)}{\partial x_i} \, dy \, d\tilde x\notag\\
	&= \int_U u(x) \frac{\partial f}{\partial x_i}(x)\, dx,\label{eq:unchangedaverage}
\end{align}
where we have used that, for $x_1, y \in U_{\tilde x}$ and for $\tilde x \in X$, $u(x_1, \tilde x)=u(y,\tilde x)$.
Furthermore, for a.e. $x\in U$ we have $|\overline{f}(x)|\leq1$ since, for a.e. $x\in U$, $|f(x)|\leq1$.  This completes the proof of (i\hspace{-1pt}i).
\end{itemize}
\end{proof}

\begin{remark}\label{rem:independent}
 By iterating the result of Lemma~\ref{A} (i\hspace{-1pt}i), it follows that, if $u\in L^1(U)$ is independent of all elements of $\{x_j\}_{j\in J}$, for a $J\subset [n]\setminus\{i\}$, then
\begin{multline*}
	\int_{\overline{U}} |D_{x_i}u| = \sup \Biggl\{ \int_U u(x) \frac{\partial f}{\partial x_i}(x) \, dx \Biggm|
	f \in L^\infty(U),\ \frac{\partial f}{\partial x_i} \in L^\infty(U),\\
	\text{for a.e.}\ x \in U \ \left| f(x) \right| \leq 1, \text{and } f \text{ is independent of } \{x_j\}_{j\in J}\Biggr\}.
\end{multline*}
In particular, extending the averaging argument from the proof of the Lemma, if we average $f$ over all elements of $\{x_j\}_{j\in J}$, the value of the integral remains unchanged as in \eqref{eq:unchangedaverage}. To be precise, without loss of generality let $\{x_j\}_{j\in J} = \{x_1, \ldots, x_{|J|}\}$ and $i>|J|$, and define the $\{x_j\}_{j\in J}$-average of $f$ as
\begin{equation}\label{eq:multivaraverage}
	\overline{f}(x_1, \ldots, x_{|J|}, \tilde x) := \frac{1}{|U_{\tilde x}|} \int_{U_{\tilde x}} f(y_1, \ldots, y_{|J|}, \tilde x) \, dy_1 \ldots dy_{|J|},
\end{equation}
for $(x_1, \ldots, x_{|J|}, \tilde x)\in U$, $U_{\tilde x} := \left\{(x_1, \ldots, x_{|J|})\in \mathbb{T}^{|J|} \bigm| (x_1, \ldots, x_{|J|}, \tilde x) \in U\right\}$, and $|U_{\tilde x}|$ the $|J|$-dimensional Lebesgue measure of $U_{\tilde x}$, then by a similar argument as in \eqref{eq:unchangedaverage}
\begin{equation}\label{eq:unchangedaverage2}
 \int_U u(x) \frac{\partial \overline{f}}{\partial x_i}(x)  \, dx = \int_U u(x) \frac{\partial f}{\partial x_i}(x)\, dx.
\end{equation}
We conclude this remark by noting that in the argument above, the fact that $U$ is open is only used to guarantee that $|U_{\tilde x}|>0$. The conclusion from \eqref{eq:unchangedaverage2} thus holds for any subset $U\subset \mathbb{T}^n$ on which $f$ and its weak derivative $\frac{\partial f}{\partial x_i}$ are well-defined as elements of (equivalence classes in) $L^\infty(U)$ and for which $U_{\tilde x}|>0$ for all $\tilde x \in\mathbb{T}^{n-|J|}$ for which $U_{\tilde x}\neq \emptyset$. In particular, if $U=Q^h_z$ for an $h\in\mathfrak{H}$ and a $z\in V_h$, then any nonempty $U_{\tilde x}$ will be the Cartesian product of sets $\{I^h_{z_i}\}_{i\in J}$ and thus have positive $|J|$-dimensional Lebesgue measure. This allows us to prove Corollary~\ref{cor:pwconstantpartialg} below, which in turn is used in our proof of Lemma~\ref{ATV}.
\end{remark}

\begin{corollary}\label{cor:pwconstantpartialg}
Let $h\in \mathfrak{H}$ and $i\in [n]$. If $u\in L_h^2$, then
\[
 \int_{\mathbb{T}^n} |D_{x_i} i_h u| = \sup_{g \in H_{h,i}} \int_{\mathbb{T}^n} i_hu(x) \frac{\partial g}{\partial x_i}(x) \, dx,
\]
where
\begin{multline}
H_{h,i} := \big\{f: \mathbb{T}^n \to \mathbb{R} \bigm| f \in L^\infty(\mathbb{T}^n), \ \frac{\partial f}{\partial x_i} \in L^\infty(\mathbb{T}^n), \ \forall x\in \mathbb{T} \ |f(x)|\leq 1, \text{ and }\\
 \forall z\in V_h \ \  \frac{\partial f}{\partial x_i} \text{ is constant a.e. on } \operatorname{int} Q_z^h\big\}.\label{eq:Hhi}
\end{multline}
\end{corollary}
\begin{proof}
Without loss of generality, we assume that $i=n$.

Since $H_{h,n}$ is a subset of the set of admissible functions in the supremum in the definition of $\displaystyle \int_{\mathbb{T}^n} |D_{x_n} i_h u|$ in \eqref{eq:TVxi} (with $i=n$), the inequality $\geq$ follows.

We prove the inequality $\leq$ by showing that for every admissible $f$ in \eqref{eq:TVxi}, there is a $\tilde f\in H_{h,n}$ such that the equality
\begin{equation}\label{eq:unchangedintegral}
  \int_{\mathbb{T}^n} u(x) \frac{\partial \tilde f}{\partial x_n}(x)  \, dx = \int_{\mathbb{T}^n} u(x) \frac{\partial f}{\partial x_n}(x)\, dx
\end{equation}
holds and thus
\[
 \int_{\mathbb{T}^n} u(x) \frac{\partial \tilde f}{\partial x_n}(x)  \, dx \leq \sup_{g \in H_{h,n}} \int_{\mathbb{T}^n} i_hu(x) \frac{\partial g}{\partial x_n}(x) \, dx,
 \]
from which the required inequality follows by first taking the supremum over all admissible $f$ in the definition in \eqref{eq:TVxi}.

To prove the equality \eqref{eq:unchangedintegral}, let $f$ be admissible in \eqref{eq:TVxi}. If $z\in V_h$ and $x\in Q^h_z$, then $i_h u(x) = u(z)$, thus in particular the function $i_h u$ is independent of all the variables in $\{x_j\}_{j\in[n-1]}$ on $Q^h_z$. For all $z\in V_h$ and for all $x\in Q^h_z$, define $\hat f(x) := \overline{f}_z(x)$, where $\overline{f}_z$ is the $\{x_j\}_{j\in[n-1]}$-avarage of $f$ over $Q^h_z$, i.e., $\overline{f}_z$ is as $\overline{f}$ in \eqref{eq:multivaraverage} with $J=[n-1]$ and $U=Q^h_z$. This choice of $U$ is allowed, by the argument at the end of Remark~\ref{rem:independent}.
Then, for all $z\in V_h$ \eqref{eq:unchangedaverage2} holds (with $i_h u$ instead of $u$ and $\overline{f}_z$ instead of $\overline{f}$), i.e.,
\[
 \int_{Q^h_z} i_h u(x) \frac{\partial \hat f}{\partial x_n}(x)  \, dx = \int_{Q^h_z} i_h u(x) \frac{\partial f}{\partial x_n}(x)\, dx.
 \]
Because $\hat f$ is independent of $\tilde x = (x_1, \ldots, x_{n-1})$ on each $Q^h_z$, we have that for a.e. $\tilde x\in \mathbb{T}^{n-1}$, $x_n \mapsto \hat f(\tilde x, x_n)$ is a Lipschitz continuous function\footnote{\label{foot:Lipschitz}We recall that, if $h\in L^\infty(\mathbb{T})$ and $h'\in L^\infty(\mathbb{T})$, then $h$ is Lipschitz continuous (or, more accurately, has a representative in $L^\infty(\mathbb{T})$ that is Lipschitz continuous) \cite[Exercise 11.50]{Leoni}, \cite[Section 5.8, Theorem 4]{Evans}.} on $\mathbb{T}$, and thus in particular pointwise evaluation of $\hat f(\cdot, \tilde x)$ is well-defined for a.e. $\tilde x \in \mathbb{T}^{n-1}$. We recall the definitions of $I_{z_n}^h$ and $\tilde Q_{z,n}^h$ from \eqref{eq:Izh} and \eqref{eq:tildeQ}, respectively. Then
\begin{align*}
 \int_{\mathbb{T}^n} i_h u(x) \frac{\partial f}{\partial x_n}(x)  \, dx &= \sum_{z\in V_h} \int_{Q^h_z} i_h u(x) \frac{\partial f}{\partial x_n}(x)  \, dx = \sum_{z\in V_h} \int_{Q^h_z} i_h u(x) \frac{\partial \hat f}{\partial x_n}(x)\, dx\\
 &= \sum_{z\in V_h} u(z) \int_{\tilde Q^h_{z,n}} \int_{I^h_{z_n}} \frac{\partial \hat f}{\partial x_n}(\tilde x, x_n) \, dx_n \, d\tilde x\\
 &= \sum_{z\in V_h} u(z) \int_{\tilde Q^h_{z,n}} \int_{I^h_{z_n}} \frac{\partial \hat f}{\partial x_n}(x) \, dx_n \, d\tilde x\\
 &= \sum_{z\in V_h} u(z) \int_{\tilde Q^h_{z,n}} \big(\hat f(\tilde x, z_n+h/2) - \hat f(\tilde x, z_n-h/2) \big) \, d\tilde x\\
 &= \sum_{z\in V_h} u(z) \int_{\tilde Q^h_{z,n}} \int_{I^h_{z_n}} \frac{\partial \tilde f}{\partial x_n}(\tilde x, x_n) \, dx_n \, d\tilde x\\
 &= \sum_{z\in V_h} \int_{\tilde Q^h_z} u(x) \frac{\partial \tilde f}{\partial x_n}(x) \, dx\\
 &= \int_{\mathbb{T}^n} u(x) \frac{\partial \tilde f}{\partial x_n}(x)\, dx.
\end{align*}
Here $\tilde f: \mathbb{T}^n \to \mathbb{R}$ is such that, for all $z\in V_h$, for all $x_n\in I^h_{z_n}$, and for all $\tilde x\in \tilde Q^h_{z,n}$, $\tilde f(\tilde x, x_n) := k_z(x_n)$, where $k_z: \mathbb{T} \to \mathbb{R}$ is the unique continuous piecewise affine solution to
\[
 k_z' = h^{-1} \left(\hat f(\tilde x, z_1+h/2)-\hat f(\tilde x, z_1-h/2)\right) \quad \text{on } \operatorname{int} I^h_{z_n}.
\]
We note that the function $k_z$ is independent of $\tilde x \in \tilde Q^h_{z,n}$, because $\hat f$ is independent of $\tilde x$ on each $Q^h_z$. Hence $\frac{\partial \tilde f}{\partial x_n}$ is constant on every $Q^h_z$, which immediately implies that $\frac{\partial \tilde f}{\partial x_n} \in L^\infty(\mathbb{T}^n)$.

Furthermore, on each $Q^h_z$, $\tilde f$ is independent of $\tilde x$ and $x_n\mapsto \tilde f(\tilde x, x_n)$ is affine with bounded derivative ---after all, since $f\in L^\infty(\mathbb{T}^n)$ also $\hat f\in L^\infty(\mathbb{T}^n)$--- and thus $\tilde f\in L^\infty(\mathbb{T}^n)$. Finally, since for a.e. $x\in \mathbb{T}^n$, $|f(x)|\leq 1$, we also have, for a.e. $x\in \mathbb{T}^n$, $|\tilde f(x)|\leq 1$. Thus, $\tilde f\in H_{h,n}$, which concludes the proof.
\end{proof}

It is well known\footnote{The results from \cite[Theorem 7.9]{Leoni}, \cite[Section 1.30]{Giusti} should be adapted to the torus.} that in the one-dimensional case, for functions $u\in BV(\mathbb{T})$,
\begin{multline}\label{eq:1DTV=essVar}
 \int_{\mathbb{T}} |Du|_{l^1} = \operatorname{essVar} u := \inf_{w} \operatorname{Var} w, \quad \text{with} \\
 \operatorname{Var} w := \sup_P \sum_{i=1}^q |w(x_i)-w(x_{i-1})| + |w(x_0) - w(x_q)|,
\end{multline}
where the infimum in $\operatorname{essVar}$ is taken over all functions $w:\mathbb{T} \to \mathbb{R}$ that satisfy $w(x)=u(x)$ for Lebesgue-almost every $x\in \mathbb{T}$, and the supremum in $\operatorname{Var}$ is taken over all partitions\footnote{If $h$ is fixed and we identify $\mathbb{T}$ with $[-h/2,1-h/2)$ with periodic boundary conditions, then by a partition of $\mathbb{T}$ we mean a subset $P=\{x_0, \ldots x_q\} \subset [-h/2,1-h/2)$, for some $q\in\mathbb{N}$, such that $x_0 < x_1 < \ldots < x_q$.} $P$ of $\mathbb{T}$.

Before we prove \eqref{eq:embeddedTV} in Lemma~\ref{lem:proofofembeddedTV}, we require the following result in the one-dimensional setting.

\begin{lemma}\label{lem:essVar}
Let $h$ be fixed, $q\in\mathbb{N}$, and let $\{I_i\}_{i=0}^q$ be a finite collection of nondegenerate disjoint intervals in $[-h/2,1-h/2)$ such that $\big|[-h/2,1-h/2)\setminus \bigcup_{i=0}^q I_i\big|= 0$ and, for all $i\in\{0, \ldots, q-1\}$, for all $x\in I_i$, and for all $y\in I_{i+1}$, $x< y$. Assume $u\in BV(\mathbb{T})$ is constant (up to a null set) on each interval $I_i$ (where we identify $\mathbb{T}$ with $[-h/2,1-h/2)$), taking value $c_i$ on interval $I_i$ (up to a null set). Then
\[
 \operatorname{essVar} u = \sum_{i=1}^q |c_i - c_{i-1}| + |c_0-c_q|.
\]
\end{lemma}
\begin{proof}
Since each interval $I_i$ has positive Lebesgue measure, for every function $w$ that agrees almost everywhere with $u$ and for all $i$, there are $x_i \in I_i$ such that $w(x_i)=c_i$. Hence
\[
 \operatorname{Var} w \geq \sum_{i=1}^q |w(x_i) - w(x_{i-1})| + |w(x_0) - w(x_q)| = \sum_{i=1}^q |c_i - c_{i-1}| + |c_0-c_q|
\]
and thus
\[
\operatorname{essVar} u \geq \sum_{i=1}^q |c_i - c_{i-1}| + |c_0-c_q|.
\]
The complement of the union of finitely many disjoint intervals in $[-h/2,1-h/2)$ is itself a union of finitely many (possibly degenerate) disjoint intervals. Since $[-h/2,1-h/2)\setminus \bigcup_{i=0}^q I_i$ has zero Lebesgue measure, it must consist of finitely many isolated points $\tilde x_j$. Let $\tilde u$ be equal to $u$ on $\bigcup_{i=0}^q I_i$ and for each $\tilde x_j$ define $\tilde u(\tilde x_j) = c_i$, where $I_i$ is the closest interval to $\tilde x_j$ (make an arbitrary choice in case of nonuniqueness). Then
\[
\operatorname{essVar} u \leq \operatorname{Var} \tilde u = \sum_{i=1}^q |c_i - c_{i-1}| + |c_0-c_q|.
\]
\end{proof}

Now we are ready to prove \eqref{eq:embeddedTV}.

\begin{lemma}[Proof of \eqref{eq:embeddedTV}]\label{lem:proofofembeddedTV}
Fix $h\in \mathfrak{H}$. If $u\in L_h^2$, then
\[
 \int_{\mathbb{T}^n} |D i_h u|_{l^1} = \varphi_{TV}^h(u)
\]
and thus consequently $i_h u \in BV(\mathbb{T}^n)$.
\end{lemma}
\begin{proof}
First we note that the second claim in the lemma follows directly from the first, since $\varphi_{TV}^h(u) < +\infty$ for $u\in L_h^2$.

For the proof of the first claim it is useful to introduce some additional notation. For all $i\in [n]$, we define $R^h_i(z) := \{\tilde z\in V_h: \forall j\in [n] \text{ with } j\neq i,\ \ \tilde z_j = z_j\}$. The set $R^h_i(z)$ contains all nodes in $V_h$ that lie on the same `row' as $z$ in the $i^{\text{th}}$ direction.

We decompose the graph total variation functional into functionals that only consider differences along the coordinate axes:
\[
	\varphi^h_{TV} (u) = \sum^n_{i=1} \varphi^h_{TV,i} (u),
\qquad \text{with} \qquad
	\varphi^h_{TV,i} (u) := \frac12 \sum_{\substack{z;\overline{z}\sim z \\ \overline{z} \in R^h_i(z)}} h^{n-1} \left| u(z) - u(\overline{z})\right|,
\]
where we use a variant of the notation from Section~\ref{sec:thefunctionalsTV}:
\[
\sum_{\substack{z;\overline{z}\sim z \\ \overline{z} \in R^h_i(z)}} := \sum_{z\in V_h} \sum_{\overline{z} \in R^h_i(z): \overline{z}\sim z}.
\]

By \eqref{eq:decomposition} it suffices to prove that, for all $i\in [n]$,
\[
	\int_{\mathbb{T}^n} |D_{x_i} i_h u| = \varphi^h_{TV,i}(u).
\]
Let $i\in [n]$. By permutation of variables, we may assume that $i=1$. Given a $z\in V_h$, by $\tilde x = (x_2, \ldots, x_n)$ we denote the coordinates on $\tilde Q^h_{z,1}$ (we recall the definition of $\tilde Q^h_{z,1}$ from \eqref{eq:tildeQ}). We define $V^0_h$ to be the set containing all nodes $z\in V_h$ with $z_1=0$, i.e.,
\begin{equation}\label{eq:V0h}
V^0_h:=\{z\in V_h \Bigm| z_1 = 0\} = \{0\} \times \{0, h, 2h, \ldots, 1-h\}^{n-1}.
\end{equation}

The space $\mathbb{T}^n$ is represented as a disjoint union
\[
	\mathbb{T}^n = \bigcup_{z\in V^0_h} (\mathbb{T}\times \tilde{Q}^h_{z,1}).
\]
For $z\in V^0_h$, we set
\begin{equation}\label{eq:Uh}
	U^h_{z,1} := \mathbb{T} \times \operatorname{int}\tilde{Q}^h_{z,1}.
\end{equation}
Since the boundary of $U^h_{z,1}$ is parallel to the $x_1$-axis,

Repeated application of Lemma~\ref{A} (i) leads to
\begin{equation}\label{eq:repeatedapplication}
	\int_{\mathbb{T}^n} |D_{x_1} i_h u| = \sum_{z\in V^0_h} \int_{\overline{U^h_{z,1}}} |D_{x_1} i_h u|.
\end{equation}
Since $i_h u$ is constant on each $Q^h_z$ and thus independent of $\tilde x$ in each $\tilde{Q}^h_{z,1}$, it is independent of $\tilde x$ on each $U^h_{z,1}$. Temporarily fix $z\in V^0_h$. Then by Lemma~\ref{A} (i\hspace{-1pt}i) and Remark~\ref{rem:independent} we have that\footnote{Cf. footnote~\ref{foot:Lipschitz}.}
\begin{align}
	\int_{\overline{U^h_{z,1}}} |D_{x_1} i_h u| = \sup \Biggl\{ \int_{U^h_{z,1}} (i_h u)(x) f'(x_1) \, dx \biggm|\,
	&f \in L^\infty(\mathbb{T}),\ f'\in L^\infty(\mathbb{T}), \text{ and}\notag\\
	&\forall x \in \mathbb{T}\ \left|f(x)\right| \leq 1 \Biggr\}.\label{eq:Dx1ihu}
\end{align}
Moreover, writing $\overline{i_h u}^z(x_1) := (i_h u)(x_1,\tilde x)$ for $x=(x_1, \tilde x)\in U^h_{z,1}$, we get
\[
	\int_{U^h_{z,1}} (i_h u)(x) f'(x_1) \, dx
	= h^{n-1} \int_\mathbb{T} \overline{i_h u}^z(x_1) f'(x_1) \, dx_1,
\]
and thus
\[
	\int_{\overline{U^h_{z,1}}} |D_{x_1}i_h u|
	= h^{n-1} \int_\mathbb{T} \left|D\overline{i_h u}^z\right|_{\ell^1}.
\]
The problem is now reduced to a one-dimensional setting.
 Since $\overline{i_h u}^z$ is piecewise constant on $\mathbb{T}$, applying the one-dimensional result from Lemma~\ref{lem:essVar} together with \eqref{eq:1DTV=essVar}, we deduce that
\[
	\int_\mathbb{T} \left|D\overline{i_h u}^z\right|_{\ell^1}
	= \frac12 \sum_{\hat z\in R^h_1(z)} \sum_{\overline{z}\in R^h_1(z): \overline{z} \sim \hat z}
	\left| u(\hat z_1,\tilde z) - u(\overline{z}_1,\tilde z)\right|,
\]
where $\hat z=(\hat z_1,\tilde z)\in V_h$ and $\overline{z} = (\overline{z}_1, \tilde z) \in V_h$. Now we unfix $z$ and observe that
\begin{align*}
	\int_{\mathbb{T}^n} |D_{x_1} i_h u| &= \sum_{z\in V^0_h}
	\int_{\overline{U^h_{z,1}}} |D_{x_1} i_h u| = h^{n-1} \sum_{z\in V^0_h}
	\int_\mathbb{T} \left|D\overline{i_h u}^z\right|_{\ell^1}\\
	&= \frac12 h^{n-1} \sum_{z\in V^0_h} \sum_{\hat z\in R^h_1(z)} \sum_{\overline{z}\in R^h_1(z): \overline{z} \sim \hat z}
	\left| u(\hat z_1,\tilde z) - u(\overline{z}_1,\tilde z)\right|
	\\
	&= \frac12 h^{n-1}\sum_{\substack{z; z\sim\tilde{z} \\ \tilde{z}\in R^h_1(z)}}
	\left| u(z) - u(\overline{z})\right| = \varphi^h_{TV,1}(u).
\end{align*}
This proves the lemma, so the proof of \eqref{eq:embeddedTV} is now also complete.
\end{proof}

\section{Properties of $\Gamma$}\label{app:propertiesofGamma}

In this section we prove some properties of the operator $\Gamma$, which is defined in \eqref{eq:Gamma}, and its matrix exponential.

It will be useful to identify $\Gamma$ with a circulant $1/h$ by $1/h$-matrix, which we will again denote by $\Gamma$. This matrix has entries, for all $i, j \in \left[\frac1h\right]$,
\[
\Gamma_{ij} = \begin{cases}
\frac23, & \text{if } i=j,\\
-\frac16, & \text{if } i\equiv j-1 \ (\mathrm{mod}\ {\frac1h}) \text{ or } i \equiv j+1 \ (\mathrm{mod}\ {\frac1h}),\\
0, & \text{otherwise}.
\end{cases}
\]
The general form of the eigenvectors and eigenvalues of circulant matrices is known. In this case it can be checked via direct computation that the eigenvalues of $\Gamma$ are, for $k\in \left[\frac1h-1\right]_0$,
\[
\gamma_k := \frac13 \left(2+\cos(2\pi h k)\right).
\]
The corresponding eigenvectors of $\Gamma$ (in $\mathbb{R}^{1/h}$) are
\[
z^{(k)} := \sqrt{h} \begin{pmatrix}
1\\ \omega^k \\ \omega^{2k} \\ \vdots \\\omega^{(\frac1h-1)k}
\end{pmatrix}, \quad \text{thus} \quad z^{(k)}_i = \sqrt{h} \omega^{(i-1)k}, \text{ for all } i\in \left[\frac1h\right],
\]
where $\omega := e^{2\pi h i_{\mathbb{C}}}$ is the $\frac1h^{\text{th}}$ root of unity and $i_{\mathbb{C}}$ denotes the imaginary unit.
We note that
\begin{equation}\label{eq:gammasym}
\gamma_{\frac1h-k} = \gamma_k.
\end{equation}

An alternative way to describe $\Gamma$ is as the sum of three matrices,
\begin{equation}\label{eq:alternativedescription}
\Gamma = \frac23 I + \frac16 A + \frac16 A^T,
\end{equation}
where $I$ is the $\frac1h$-by-$\frac1h$ identity matrix and $A$ and $A^T$ are the matrix representations of the operators $\tau_+$ and $\tau_-$, respectively, i.e.
\[
A_{ij} := \delta_{i,j+1},
\]
where we always interpret the indices modulo $\frac1h$, so $\delta_{i,\frac1h+1} = \delta_{i,1}$, etc. Remember that the Kronecker delta is defined by
\[
\delta_{i,j} := \begin{cases}
1, \text{if } i=j,\\
0, \text{otherwise}.
\end{cases}
\]
We note that, for all $n\in \mathbb{N}$,
\begin{equation}\label{eq:An}
A_{ij}^n = \delta_{i,j+n},
\end{equation}
as can be proved by induction on $n$: Taking as base case $n=0$ or $n=1$, we see it is true by definition of $I$ or $A$, respectively. Assuming that it is true for $n=k > 1$, we find
\[
A_{ij}^{k+1} = \sum_{l=1}^\frac1h A_{il} A_{lj}^k = \sum_{l=1}^\frac1h \delta_{i,l+1} \delta_{l,j+k} = \delta_{i,j+k+1}.
\]
As a consequence of \eqref{eq:An}, we have that
\begin{equation}\label{eq:A1h=A}
A^{\frac1h} = A.
\end{equation}

Furthermore $A^T=A^{-1}$, since
\[
\left(A A^T\right)_{ij} = \sum_{l=1}^\frac1h A_{il} A_{jl} = \sum_{l=1}^\frac1h \delta_{i,l+1} \delta_{j,l+1} = \delta_{i,j}.
\]
It is worth noting that
\[
\left(A^{-1}\right)_{ij} = \left(A^T\right)_{ij} = A_{ji} = \delta_{j,i+1} = \delta_{i,j-1} \quad \text{and,} \quad \text{for all } n\in \mathbb{N}, \ \left(A^{-n}\right)_{ij} = \delta_{i,j-n}.
\]

These two characterisations of $\Gamma$, one through its spectrum and the other using the matrix $A$, lead to two expresssions for the matrix exponential $e^{\Gamma x}$ in the following lemma.

\begin{lemma}
Let $x\in \mathbb{R}$. For all $i,j \in \left[\frac1h\right]$ we have
\begin{equation}\label{eq:expGammax1}
\left(e^{\Gamma x}\right)_{ij} = h e^{\frac23 x} \sum_{l=0}^{\frac1h-1} e^{\frac13 \cos(2\pi h l) x} \cos\left(2\pi h l (i-j)\right)
\end{equation}
and
\begin{equation}\label{eq:expGammax2}
\left(e^{\Gamma x}\right)_{ij} = e^{\frac23 x} \sum_{l=0}^{\frac1h-1} B_l(x) B_{l+j-i}(x),
\end{equation}
where
\[
B_l(x) := \sum_{m=0}^\infty \frac1{(\frac{m}h + l)!} \left(\frac{x}6\right)^{\frac{m}h+l}
\]
has infinite radius of convergence.
The subscripts of $B$ are interpreted modulo $\frac1h$.
\end{lemma}
\begin{proof}
The proof of the first expression follows an ODE approach: Consider the equation
\begin{equation}\label{eq:GammaODE}
u'(x) = \Gamma u(x).
\end{equation}
By standard ODE theory, this equation has $\frac1h$ linearly independent solutions $u^j: \mathbb{R} \to \mathbb{R}^{1/h}$, for $j\in \left[\frac1h\right]$. Our discussion above shows that $\Gamma$ has $\frac1h$ linearly independent eigenvectors and thus, for all $j\in \left[\frac1h\right]$, $u^j(x) = e^{\gamma_{j-1} x} z^{(j-1)}$. We can now compute
\begin{equation}\label{eq:XXinverse}
e^{\Gamma x} = X(x) X(0)^{-1},
\end{equation}
where $X: \mathbb{R} \to \mathbb{R}^{1/h \times 1/h}$ can be any fundamental matrix solution of \eqref{eq:GammaODE}. We choose the matrix which has the solutions $u^j$ as columns: for all $i,j\in \left[\frac1h\right]$,
\[
X_{ij}(x) := u^j_i(x) = \sqrt{h}  e^{\gamma_{j-1} x} \omega^{(i-1)(j-1)}.
\]
Then
\[
X_{ij}(0) = u^j_i(0) = z^{(j-1)}_i = \sqrt{h} \omega^{(i-1)(j-1)}.
\]
We claim that the inverse $X(0)^{-1}$ is given by, for all $i,j\in \left[\frac1h\right]$,
\[
Y_{ij} := \sqrt{h} \omega^{-(i-1)(j-1)}.
\]
To prove this claim, we compute
\begin{align*}
\left( X(0) Y\right)_{ij} &= \sum_{l=1}^{\frac1h} X(0)_{il} Y_{lj} = \sqrt{h} \sum_{l=1}^{\frac1h} z^{(l-1)}_i \omega^{-(l-1)(j-1)} =  h \sum_{l=1}^{\frac1h}  \omega^{(i-1)(l-1)} \omega^{-(l-1)(j-1)}\\
&= h \sum_{l=1}^{\frac1h} \omega^{(l-1)(i-j)}.
\end{align*}
Taking $i=j$ shows that $\left( X(0) Y\right)_{ii} = 1$. If $i\neq j$, then we multiply by $\omega^{i-j}$:
\[
\omega^{i-j} \left( X(0) Y\right)_{ij} =  h \sum_{l=1}^{\frac1h} \omega^{l(i-j)} = h \sum_{k=2}^{\frac1h+1} \omega^{(l-1)(i-j)} = h \sum_{l=1}^{\frac1h} \omega^{(l-1)(i-j)} =  \left( X(0) Y\right)_{ij},
\]
where the third equality follows from the fact that $\omega^{(1/h+1-1)(i-j)} = \left(\omega^{1/h}\right)^{i-j} = 1^{i-j} = 1 = \omega^{0 (i-j)}$ ---recall that $\omega$ is the $\frac1h^{\text{th}}$ root of unity, and thus $\omega^l=1$ if and only if $l \equiv 0 \ (\mathrm{mod}\ {\frac1h}$).
 Since $i$ is not congruent to $j$ modulo $\frac1h$ and thus $\omega^{i-j} \neq 1$, it follows that $\left( X(0) Y\right)_{ij}=0$. Thus $Y=X(0)^{-1}$.

Using \eqref{eq:XXinverse} we compute
\[
\left(e^{\Gamma x}\right)_{ij} = h \sum_{l=1}^{\frac1h} e^{\gamma_{l-1} x} \omega^{(i-1)(l-1)} \omega^{-(l-1)(j-1)} = h \sum_{l=0}^{\frac1h-1} e^{\gamma_l x}  \omega^{(i-j)l} = h + h \sum_{l=1}^{\frac1h-1} e^{\gamma_l x}  \omega^{(i-j)l}.
\]
By \eqref{eq:gammasym} it follows that, if $\frac1h$ is odd,
\begin{align*}
\left(e^{\Gamma x}\right)_{ij} &=  h + h \sum_{l=1}^{\frac12(\frac1h-1)}  e^{\gamma_l x} \left(\omega^{(i-j)l} + \omega^{(i-j)(1/h-l)}\right) \\
&=  h + \frac{h}2 \sum_{l=1}^{\frac1h-1}  e^{\gamma_l x} \left(\omega^{(i-j)l} + \omega^{(i-j)(1/h-l)}\right),
\end{align*}
where for the second equality we also used that $\omega^{(i-j)l} + \omega^{(i-j)(1/h-l)}$ remains unchanged when $l$ is replaced by $\frac1h-l$.
Since
\begin{align}
&\hspace{0.6cm} \omega^{(i-j)l} + \omega^{(i-j)(1/h-l)}\notag\\
&= \cos\bigl(2\pi h (i-j) l\bigr) + i_{\mathbb{C}} \sin\bigl(2\pi h (i-j) l\bigr) + \cos\bigl(2\pi h (i-j)(1/h-l)\bigr)\notag\\
&\hspace{4cm} + i_{\mathbb{C}} \sin\bigl(2\pi h (i-j)(1/h-l)\bigr)\notag\\
&=  \cos\bigl(2\pi h (i-j) l\bigr) + \cos\bigl(2\pi (i-j) - 2\pi (i-j)hl\bigr)\notag\\
&\hspace{4cm} + i_{\mathbb{C}} \left[\sin\bigl(2\pi h (i-j) l\bigr) + \sin\bigl(2\pi (i-j) - 2\pi (i-j)hl\bigr)\right]\notag\\
&=  \cos\bigl(2\pi h (i-j) l\bigr) + \cos\bigl(- 2\pi (i-j)hl\bigr) + i_{\mathbb{C}} \left[\sin\bigl(2\pi h (i-j) l\bigr) + \sin\bigl(- 2\pi (i-j)hl\bigr)\right]\notag\\
&= 2  \cos\bigl(2\pi h (i-j) l\bigr),\label{eq:cosines}
\end{align}
the desired expression \eqref{eq:expGammax1}, for $\frac1h$ odd, now follows from the definition of $\gamma_l$. If, on the other hand, $\frac1h$ is even, we have
\begin{align*}
\left(e^{\Gamma x}\right)_{ij} &=  h + h \sum_{l=1}^{\frac1{2h} - 1}  e^{\gamma_l x} \left(\omega^{(i-j)l} + \omega^{(i-j)(1/h-l)}\right) + h e^{\gamma_{h/2}x} \omega^{(i-j)/(2h)}\\
&= h + \frac{h}2 \sum_{\substack{l=1 \\ l\neq \frac1{2h}}}^{\frac1h - 1}  e^{\gamma_l x} \left(\omega^{(i-j)l} + \omega^{(i-j)(1/h-l)}\right) + h e^{\gamma_{h/2}x} \omega^{(i-j)/(2h)}.
\end{align*}
For the middle terms on the right-hand side we use \eqref{eq:cosines} again. For the last term, we note that $\omega^{1/(2h)} = -1$ and thus
\[
 \omega^{(i-j)/(2h)} = (-1)^{i-j} = 1 = \cos(\pi (i-j)) = \cos\left(\frac{2\pi h (i-j)}{2h}\right).
\]
Combined with the definition of $\gamma_l$, this now also gives \eqref{eq:expGammax1} when $\frac1h$ is even.

To prove \eqref{eq:expGammax2}, we use the description of $\Gamma$ from \eqref{eq:alternativedescription}. First note that, since
\[
\sum_{m=0}^\infty \left|\frac1{(\frac{m}h + l)!} \left(\frac{x}6\right)^{\frac{m}h+l}\right| = \sum_{m=0}^\infty \frac1{(\frac{m}h + l)!} \left|\frac{x}6\right| ^{\frac{m}h+l} \leq \sum_{m=0}^\infty \frac1{m!} |x|^m = e^x,
\]
the series $B_l$ has an infinite radius of convergence.

Because $A$ and $A^T=A^{-1}$ commute, we have
\begin{equation}\label{eq:becausecommute}
e^{\Gamma x} = e^{(\frac23 I + \frac16 A + \frac16 A^T)x} = e^{\frac23x} e^{\frac16Ax} e^{\frac16 A^{-1} x}.
\end{equation}

Next we employ the series definition of the matrix exponential in combination with \eqref{eq:A1h=A}:
\begin{align*}
e^{\frac16 A x} &= \sum_{m=0}^\infty \frac1{m!} A^m \left(\frac{x}6\right)^m\\
&= \left[\frac1{(0/h)!}  \left(\frac{x}6\right)^{0/h} + \frac1{(1/h)!}  \left(\frac{x}6\right)^{1/h} + \frac1{(2/h)!}  \left(\frac{x}6\right)^{2/h} + \frac1{(3/h)!}  \left(\frac{x}6\right)^{3/h} + \ldots \right] I\\
&+ \left[ \frac1{(0/h+1)!} \left(\frac{x}6\right)^{0/h+1} + \frac1{(1/h+1)!}  \left(\frac{x}6\right)^{1/h+1}\right.\\
&\left.+ \frac1{(2/h+1)!}  \left(\frac{x}6\right)^{2/h+1} + \frac1{(3/h+1)}  \left(\frac{x}6\right)^{3/h+1}\right] A\\
&+ \left[  \frac1{(0/h+2)!} \left(\frac{x}6\right)^{0/h+2} + \frac1{(1/h+2)!}  \left(\frac{x}6\right)^{1/h+2}\right.\\
&\left.+ \frac1{(2/h+2)!}  \left(\frac{x}6\right)^{2/h+2} + \frac1{(3/h+2)}  \left(\frac{x}6\right)^{3/h+2}\right] A^2\\
&+ \ldots \\
&+ \left[  \frac1{(0/h+1/h-1)!} \left(\frac{x}6\right)^{0/h+1/h-1} + \frac1{(1/h+1/h-1)!}  \left(\frac{x}6\right)^{1/h+1/h-1}\right.\\
&\left.+ \frac1{(2/h+1/h-1)!}  \left(\frac{x}6\right)^{2/h+1/h-1} + \frac1{(3/h+1/h-1)}  \left(\frac{x}6\right)^{3/h+1/h-1}\right] A^{1/h-1}\\
&= \sum_{l=0}^{\frac1h-1} B_l(x) A^l.
\end{align*}
Substituting this into \eqref{eq:becausecommute} we find
\[
e^{\Gamma x} = e^{\frac23 x} \sum_{l=0}^{\frac1h-1} \sum_{r=0}^{\frac1h-1} B_l(x) B_r(x) A^l A^{-r}.
\]
Since
\begin{align*}
\left(e^{\Gamma x}\right)_{ij} &= e^{\frac23 x} \sum_{l=0}^{\frac1h-1} \sum_{r=0}^{\frac1h-1} B_l(x) B_r(x) \sum_{k=1}^{\frac1h-1} \delta_{i,k+l} \delta_{k, j-r} = e^{\frac23 x} \sum_{l=0}^{\frac1h-1} \sum_{r=0}^{\frac1h-1} B_l(x) B_r(x) \delta_{i,j-r+l}\\
 &=  e^{\frac23 x} \sum_{l=0}^{\frac1h-1} B_l(x) B_{l+j-i}(x).
\end{align*}
This completes the proof of \eqref{eq:expGammax2}.
\end{proof}

\begin{remark}In the proof above we have given the deductions by which the two expressions for $e^{\Gamma x}$ were derived. Alternatively, we can also directly check that each column of $e^{\Gamma x}$ as given in \eqref{eq:expGammax1} and \eqref{eq:expGammax2} satisfies \eqref{eq:GammaODE} with initial condition $e^{\Gamma 0} = I$.
\end{remark}

\section{Poincar\'e--Wirtinger-type inequalities}\label{app:PoincareWirtinger}

In this section we the Poincar\'e--Wirtinger-type inequalities we used in the proof of Lemma~\ref{ER}. The first one is a classic result, the second one a higher-order extension which is proved using similar methods as the first.

\begin{lemma}\label{lem:PoinWirt}
\begin{enumerate}
\item [(i)] Let $u\in H^1((0,h))$ with $u(0)=u(h)=0$. Then the Poincar\'e--Wirtinger inequality holds:
 \[
  \int_0^h |u(x)|^2 \, dx \leq (h/\pi)^2 \int_0^h |u'(x)|^2 \, dx.
 \]
 The constant $(h/\pi)^2$ is optimal.
\item[(i\hspace{-1pt}i)] Let $u\in H^2((0,h))$ with $u(0)=u(h)=0$. Then there exists a $\mu>0$ such that
\[
	\int^h_0 |u(x)|^2 \, dx \leq (h/\pi)^4 \int^h_0 |u''(x)|^2 \, dx.
\]
The constant $(h/\pi)^4$ is optimal.
\end{enumerate}
\end{lemma}

\begin{proof}
 \begin{enumerate}
  \item[(i)] We note that the boundary conditions are well-defined, since $H^1((0,h))$ is compactly embedded in the Banach space $\{u\in C((0,h)) \bigm| u \text{ is bounded}\}$\footnote{Equipped with the supremum norm $\sup_{x\in (0,h)}|u(x)|$.}. by the Rellich--Kondrachov theorem \cite[Theorem 6.3]{AF}.

  The Poincar\'e--Wirtinger inequality for symmetric elliptic operators is found in \cite[Section 6.5, Theorem 2 and the following remark]{Evans}, from which it immediately follows that there exists a $C>0$ (strict positivity is given by \cite[Section 6.5, Theorem 1]{Evans}) such that, for all $u\in H^1((0,1))$ with $u(0)=u(h)=0$, the inequality
  \[
  \int_0^h |u(x)^2| \, dx \leq C \int_0^h |u'(x)|^2 \, dx
 \]
  holds. It also tells us that the optimal constant is $C=\nu^{-1}$, where $\nu$ is obtained from its Rayleigh quotient formulation
$\displaystyle
 \nu = \min_{\substack{w\in H^1((0,h))\\w\neq 0}} \frac{\int_0^h |w'(x)|^2 \, dx}{\int_0^h |w(x)|^2 \, dx}
$ (where the minimum is achieved),
or as the minimal eigenvalue in the eigenvalue problem
\footnote{Additionally, the set of unit eigenfunctions of the differential operator $T(u):=u''$ is a complete orthonormal basis of the Hilbert space $L^2((0,h))$. This can be shown as follows. We claim that the inverse $T^{-1}$ of the densely defined unbounded differential operator $T$ on $L^2((0,h))$ exists and is compact. If this claim is true, then by the Riesz--Schauder theory \cite[Section X.5, Theorem 2]{Yosida95} all nonzero elements of the spectrum of $T^{-1}$ are eigenvalues of $T^{-1}$ and thus their reciprocals are eigenvalues of $T$ with the same eigenfunctions. Moreover, since $T^{-1}$ is also self-adjoint, its eigenfunctions form a complete orthogonal basis for $L^2((0,h))$ \cite[Section XI.6, Theorem 1, and Section XI.9, Theorem 1 and Corollary 2]{Yosida95}. Hence, the same is true for the eigenfunctions of $T$.
To prove that $T^{-1}$ exists and is compact, we observe that, for $f\in L^2((0,h))$, $T^{-1}f = u$, where $u$ is the unique solution to $u''=f$ on $(0,h)$ and $u(0)=u(h)=0$. Existence follows from an explicit computation that shows that
\[
 (T^{-1}f)(x) = \frac1h \int_0^x \int_0^h \left[\int_0^w f(y) \, dy - \int_0^z f(y) \, dy\right] \, dz \, dw.
\]
If $\|f\|_{L^((0,h))} \leq 1$ and $x,y\in (0,h)$ with $y>x$, then, by the Cauchy--Schwarz inequality,
\begin{align*}
 \left|(T^{-1}f)(x)\right| &\leq \frac2h \int_0^h \int_0^h \int_0^h |f(y)| \, dy \, dz \, dw = 2h \int_0^h |f(y)| \, dy \\
 &\leq 2h \|1\|_{L^2((0,h))} \|f\|_{L^2((0,h))} \leq 2 h^{3/2},
\end{align*}
and
\begin{align*}
\left|(T^{-1}f)(x)-(T^{-1}f)(x)\right| &\leq \frac1h \int_{x_1}^{x_2} \int_0^h \left|\int_0^w f(y) \, dy - \int_0^z f(y) \, dy\right| \, dz \, dw \\
&\leq 2 |x_2-x_1| \int_0^h |f(y)| \, dz\\
&\leq 2 h^{1/2} \|f\|_{L^2((0,h))} |x_2-x_1|.
\end{align*}
Hence by the Arzel{\`a}--Ascoli theorem \cite[Theorem 1.33]{AF}, \cite[Theorem 5.37]{Leoni}, $T^{-1}$ is a compact operator.
}
 \[
  \begin{cases}
   w'' + \tilde\nu w = 0, \quad \text{on } (0,h),&\\
   w(0)=w(h)=0.&
  \end{cases}
 \]
This ordinary differential equation can be solved by ODE standard techniques \cite[Theorem III.4.1]{Hale}; its general solution (without imposing boundary conditions) is given by
\[
w(x) = c_1 \cos(\sqrt{\tilde\nu}x) + c_2 \sin(\sqrt{\tilde\nu}x),
\]
for constants $c_1, c_2 \in \mathbb{R}$. From $w(0)=0$, it follows that $c_1=0$. Since $w(h)=0$, we have $c_2=0$ or $\sin(\sqrt{\tilde\nu}h)=0$. Because $w=0$ is not allowed as eigenfunction, the option $c_2=0$ is excluded, hence $\nu$ is the smallest $\tilde\nu>0$ such that $\sin(\sqrt{\tilde\nu}h)=0$. This proves that $\nu = \pi^2/h^2$, as required.

We note that \cite[Section 1.7(2)]{DM} provides an alternative proof for continuously differentiable functions $u$ based on Fourier series.

\item[(i\hspace{-1pt}i)] Let $u \in \mathcal{A}_h := \{u\in H^2((0,h)) \bigm| u(0)=u(h)=0\}$. We note that by the Rellich--Kondrachov theorem \cite[Theorem 6.3]{AF} $H^2((0,h))$ is compactly embedded in the Banach space $\{u\in C^1((0,h)) \bigm| u, u' \text{ are bounded}\}$\footnote{Equipped with the supremum norm $\max(\sup_{x\in (0,h)}|u(x)|,\sup_{x\in (0,h)}|u'(x)|)$.} and thus the boundary conditions in the definition of $\mathcal{A}_h$ are well-defined.

We define, for this proof,
\begin{equation}\label{eq:Rayleighnu}
 \nu := \inf_{\substack{w\in \mathcal{A}_h\\w\neq 0}} \frac{\int_0^h |w''(x)|^2 \, dx}{\int_0^h |w(x)|^2 \, dx} = \inf_{w\in \mathcal{A}_h} \left\{\|w''\|_{L^2((0,h))}^2 \bigm| \|w\|_{L^2((0,h))}^2 = 1\right\},
\end{equation}
where the second expression follows by rescaling $u$ by $\left(\int_0^h |w(x)|^2 \, dx\right)^{-1}$.
Since the quotient is nonnegative, $\nu \geq 0$ exists. If, as we will prove, $\nu>0$, then it follows that
\[
 \int^h_0 |u(x)|^2 \, dx \leq \nu^{-1} \int^h_0 |u''(x)|^2 \, dx.
\]
Moreover, a minimizer of the minimization problem in the definition of $\nu$ exists by the direct method of the calculus of variations, as the following argument shows. We use the second characterisation of $\nu$ in \eqref{eq:Rayleighnu}. Consider a minimizing sequence $(w_m)_m$ in $\left\{u\in \mathcal{A}_h \bigm| \|u\|_{L^2((0,h))}^2 = 1\right\}$. By definition $(\|w_m\|_{L^2((0,h))}^2)_m$ and $(\|w_m''\|_{L^2((0,h))}^2)_m$ are bounded. By part (i) of this lemma with $u=w_m'$, $(\|w_m'\|_{L^2((0,h))}^2)_m$ is also bounded, hence $(w_m)_m$ is bounded in $H^2((0,h))$ and thus an $H^2((0,h))$-weakly converging subsequence exists with limit $w_*\in H^2((0,h))$. By the compact embedding of $H^2((0,h))$ into a space of continuous functions that was established above, $\|w_*\|_{L^2((0,h))}^2=1$. Moreover, since $u\mapsto \|u''\|_{L^2((0,h))}^2$ is lower semicontinuous with respect to weak-$H^2((0,h))$ convergence \cite[Proposition 3.5]{Brezis11}, $w_*$ minimizes $w\mapsto \|w''\|_{L^2((0,h))}^2$ over the admissible set.

From the first expression for $\nu$ in \eqref{eq:Rayleighnu}, we see that the minimizer $w_*$ has to satisfy
\[
	\left. \frac{d}{d\varepsilon}
	\frac{\int_0^h \left( (w_*+\varepsilon f)''(x)\right)^2\, dx}{\int_0^h (w_*+\varepsilon f)^2(x)\, dx}
	\right|_{\varepsilon=0} = 0,
\]
for all $f\in C^\infty([0,h])$ with $f(0)=f(h)=0$. Computing this derivative and using that $\nu = \frac{\|w_*''\|_{L^2((0,h))}^2}{\|w_*\|_{L^2((0,h))}^2}$, we find
\[
 \frac2{\|w_*\|_{L^2((0,h))}^2} \int_0^h \left[w_*''(x) f''(x) - \nu w_*(x) f(x)\right] \, dx = 0.
\]
Integrating by parts, we see
\begin{align*}
	\int_0^h w_*''(x) f''(x)\, dx &= \left. - \int_0^h w_*'''(x) f'(x) \, dx + w_*'' f' \right|^h_0 \\
	&= \left. \int_0^h w_*''''(x) f(x) \, dx - w_*''' f \right|^h_0
	+ \left. w_*'' f' \right|^h_0.
\end{align*}
Since $f=0$ on $\{0,h\}$, the term $\left. w_*''' f \right|^h_0$ vanishes.
 Thus the weak form implies that
\[
	\left. \int_0^h \big[w_*''''(x) - \nu(x) w_*(x)\big] f(x) \, dx + w_*'' f' \right|^h_0 = 0,
\]
for all $f \in C^\infty([0,h])$ with $f(0)=f(h)=0$.
 Hence $w_*$ solves
\[
	w'''' - \nu w = 0 \quad\text{on } (0,h).
\]
Since $f'$ can be taken arbitrary value at $\{0,h\}$, we get natural boundary conditions.
 Thus $w_*$ solves
\[
	w''(0) = w''(h) = 0.
\]
Since $w_*$ satisfies $w_*|_{\{0,h\}}=0$ and $w_*\neq 0$, we deduce that $\nu$ must be an eigenvalue of
\begin{equation}\label{eq:eigenvalueproblem}
	\begin{cases}
	w'''' - \tilde\nu  w = 0, \quad \text{on } (0,h),&\\
	w''(0) = w''(h) = 0,&\\
	w(0) = w(h) = 0.&
    \end{cases}
\end{equation}
We note that, if $\tilde\nu=0$, then the ODE above implies that $w$ is a polynomial of degree at most three and thus, due to the boundary conditions $w=0$. Hence $\tilde \nu=0$ is not an eigenvalue and thus $\nu\neq 0$. Already we knew that $\nu\geq 0$, hence $\nu>0$.

In fact, substituting the (admissible) eigenfunction $w(x) = \sin(\pi x/h)$ in the minimization problem in \eqref{eq:Rayleighnu}, we find that $\nu \geq (\pi/h)^4$.

Finally we prove that every eigenvalue $\tilde\nu$ in the eigenvalue problem in \eqref{eq:eigenvalueproblem} is larger than or equal to $(\pi/h)^4$, which establishes that $\nu=(\pi/h)^4$, as required.

Let $\tilde\nu>0$ be an eigenvalue of \eqref{eq:eigenvalueproblem} and let $\sigma>0$ be the unique strictly positive real number such that $\tilde\nu=\sigma^4$. By standard ODE methods \cite[Theorem III.4.1]{Hale}, the general solution to the ODE $w'''' - \tilde\nu  w = 0$ is found to be
\[
 w(x) = c_1 e^{\sigma x} + c_2 e^{-\sigma x} + c_3 e^{i \sigma x} +c_4 e^{-i\sigma x},
\]
for constants $c_1, c_2, c_3, c_4\in \mathbb{R}$. Here $i$ is the imaginary unit.
Since $w(0)=0$ and $w''(0)=0$ imply
$c_1 + c_2 + c_3 + c_4 = 0$ and $\sigma^2 (c_1 + c_2) - \sigma^2 (c_3 + c_4) = 0$,
we deduce that $c_1+c_2=c_3+c_4=0$. Therefore
\[
	w(x) = c_1(e^{\sigma x} - e^{-\sigma x}) + c_3(e^{i\sigma x} - e^{-i\sigma x}).
\]
Thus
\begin{align*}
	w(h) &= c_1(e^{\sigma h} - e^{-\sigma h}) + c_3(e^{i\sigma h} - e^{-i\sigma h}) = 0, \\
	\sigma^{-2} w''(h) &= c_1(e^{\sigma h} - e^{-\sigma h}) - c_3(e^{i\sigma h} - e^{-i\sigma h}) = 0,
\end{align*}
which in matrix form reads
\[
\begin{pmatrix}
	e^{\sigma h} - e^{-\sigma h} & e^{i\sigma h} - e^{-i\sigma h}  \\
	e^{\sigma h} - e^{-\sigma h} & -(e^{i \sigma h} - e^{-i\sigma h})
\end{pmatrix}
\begin{pmatrix}
	c_1 \\
	c_3
\end{pmatrix}
		=
	\begin{pmatrix}
	0 \\
	0 \\
\end{pmatrix}.
\]
This equation has a nontrivial solution if and only if the determinant of the above matrix is zero: $-2 (e^{\sigma h} - e^{-\sigma h}) (e^{i \sigma h} - e^{-i\sigma h})$.
 Since $e^{\sigma h} \neq e^{-\sigma h}$, we obtain that $e^{i \sigma h}-e^{-i\sigma h}=0$.
 In other words, $e^{2i\sigma h}=1$.
 The minimum strictly positive value of such $\sigma$ is $\pi/h$.
 Thus $\nu=(\pi/h)^4$. This concludes the proof\footnote{Additionally we can conclude from our computations that the geometric multiplicity of the eigenvalue $\tilde\nu = (\pi/h)^4$ is equal to one, as any solution $w$ to \eqref{eq:eigenvalueproblem} with $\tilde\nu = (\pi/h)^4$ is of the form $C\sin\pi x$ for some constant $C\in \mathbb{R}$.}.

\end{enumerate}

\end{proof}

\section*{Acknowledgments}
The first author was partly supported by the Japan Society for the Promotion of Science through grants No.\ 19H00689 (Kiban A), No.\ 18H05323 (Kaitaku) and by Arithmer Inc.\ and Daikin Industries Ltd.\ through a collaborative grant.
The second author has received funding from the European Union's Horizon 2020 research and innovation programme under the Marie Sk{\l}odowska-Curie grant agreement No.\ 777826.
The work of the third author was done when he was a graduate student at the University of Tokyo and a Junior Research Associtate at RIKEN. His work was partly supported by the Program for Leading  Graduate Schools ``Frontier of Mathematical Sciences and Physics'', MEXT, Japan.

\bibliographystyle{plain}
\bibliography{grgrflows}

(Y.~Giga)
Graduate School of Mathematical Sciences, The University of Tokyo

\textit{Email address}: \url{labgiga@ms.u-tokyo.ac.jp}

(Y.~van Gennip)
Delft Institute of Applied Mathematics, Technische Universiteit Delft

\textit{Email address}: \url{y.vangennip@tudelft.nl}

(J.~Okamoto)
Institute for the Advanced Study of Human Biology, Kyoto University

\textit{Email address}: \url{okamoto.jun.8n@kyoto-u.ac.jp}

\end{document}